\documentclass[reqno,11pt]{amsart}
\usepackage[top=2.0cm,bottom=2.0cm,left=3cm,right=3cm]{geometry}
\usepackage{amsthm,amsmath,amssymb,dsfont}
\usepackage{mathrsfs,amsfonts,functan,extarrows,mathtools}
\usepackage[colorlinks]{hyperref}
\usepackage{marginnote}
\usepackage{xcolor}
\usepackage{stmaryrd}
\usepackage{esint}
\usepackage{graphicx}
\usepackage{bbm}
\usepackage{tikz}
\usetikzlibrary{positioning}
\usepackage{algorithm}
\usepackage{algpseudocode}
\usepackage{caption}
\usepackage{subcaption}
\usepackage{comment}
\usepackage[normalem]{ulem}

\newtheorem{theorem}{Theorem}[section]
\newtheorem{proposition}{Proposition}[section]

\newtheorem{lemma}{Lemma}[section]
\newtheorem{assumption}{Assumption}[section]
\numberwithin{equation}{section}
\allowdisplaybreaks
\def\d{\mathrm{d}}
\def\R{\mathbb{R}}
\def\E{\mathbb{E}}

\def\CLip{C_{\text{Lip}}}
\def\Cbdd{C_{\text{bdd}}}

\def\eps{\varepsilon}

\def\l{\langle}
\def\r{\rangle}
\def\exp{\mathrm{exp}}

\newcounter{wronumber}

\begin{document}
\title[Inference of interacting kernel]{Inference of interacting kernel in the mean-field regime}
\author[Chen]{Peiyi Chen}
\address[Peiyi Chen]
		{\newline Department of Mathematics, University of Wisconsin-Madison, Madison, WI, 53706, USA}
\email{pchen345@wisc.edu}

\author[Li]{Qin Li}
\address[Qin Li]
		{\newline Department of Mathematics, University of Wisconsin-Madison, Madison, WI, 53706, USA}
\email{qinli@math.wisc.edu}

\author[Wang]{Li Wang}
\address[Li Wang]
		{\newline School of Mathematics, University of Minnesota-Twin Cities, Minneapolis, MN, 55455, USA}
\email{liwang@umn.edu}

\author[Yang]{Yunan Yang}
\address[Yunan Yang]
		{\newline Department of Mathematics, Cornell University, Ithaca, NY, 14853, USA}
\email{yunan.yang@cornell.edu}

\thanks{\today}

\begin{abstract} 
We study the problem of reconstructing interaction kernels in systems of interacting agents from macroscopic measurements when posed as an optimization problem. The reconstruction procedure depends on the formulation of the forward model, which may be given either by a finite-dimensional coupled ODE system tracking individual agent trajectories or by a mean-field PDE describing the evolution of the agent density. We investigate the similarities and differences between these two formulations in the mean-field regime. While the first variation derived from the particle system does not provide an unbiased estimator of the first variation associated with the limiting PDE, we prove that, under mild assumptions, the two are close in a weak sense with a convergence rate $\mathcal{O}(N^{-1/2})$. This rate is further confirmed by numerical evidences.
\end{abstract}

\maketitle

\section{Introduction}

Newton’s \emph{hypotheses non fingo} reflects an early instance of inferring interaction laws directly from observed dynamics. In the case of gravitation, this approach led to the inverse-square force law and a system of coupled ordinary differential equations (ODEs) governing the motion of interacting bodies. A similar data-driven philosophy underlies Coulomb's law in electrostatics. From a modern perspective, both can be viewed as instances of \emph{inverse problems}: recovering interaction laws from trajectory data. Such inverse formulations continue to play a central role in the current modeling of complex systems~\cite{Lu_Zhong_Tang_2019_nonparam, He2022NumericalIO, Li-Wang-Yang-2023-MC-Gradient, Yang-Caflisch-2023-DSMC-Boltzmann, Li-Sun-Lai-2024-doping-profile, SAWANT2023115836, Uy_Peherstorfer_2021_inference}.

In many current applications, the systems of interest involve large populations of interacting agents~\cite{Cucker_Smale_2007_flock, Ha_Tadmor_2008_flocking, Ahn_Ha_2010_flocking, Tadmor2021OnTM, Motsch2011-mw}, including molecular systems~\cite{Kolokolnikov_Carrillo_2013_particle_interaction}, biological collectives~\cite{Stevens_1997_bio, Greene_2023}, animal groups~\cite{Lukeman_2010_PNAS_scoter, Katz_2011_PNAS_fish}, robotic swarms, and traffic flows. In such settings, the objective is to infer effective interaction laws governing collective behavior from observational data, typically modeled through interaction kernels~\cite{Bongini_2017_learn_interaction, Hellmuth-Li-2025-chemotaxis, carrillo_inverse_kernel2025, Lu2020LearningIK, ZHONG_2020_data_collect, Lu_2021_kernel_trajectory, Lang_Lu_interact_2022, Miller2020LearningTF, blickhan2025dicediscreteinversecontinuity}.

A fundamental distinction {between multi-agent systems and} classical two-body experiments is that individual interactions cannot generally be isolated. Instead, observations consist of the simultaneous evolution of many agents and are often available only through aggregated or macroscopic measurements that do not resolve individuals. Consequently, the inference problem shifts from low-dimensional systems of coupled ordinary differential equations to high-dimensional or continuum partial-differential equation models.

This shift raises an identifiability question: when agents are indistinguishable and only collective behavior is observed, under what conditions can interaction kernels be recovered from the data? The answer depends on the structure of the dynamics, the observation model, and the available measurements. In this work, we focus on a mathematical formulation of this problem arising in the \emph{mean-field regime}.

The mean-field regime corresponds to a scaling limit in which the number of interacting particles becomes large and the influence of any individual particle is negligible relative to the collective effect~\cite{PE_Jabin_2014_mean_field_review,  Jabin_Wang_2017_stochastic_particle}. In this limit, the system dynamics are described by the evolution of a probability distribution governed by a partial differential equation (PDE), with interactions represented through averaged, nonlocal terms~\cite{Spohn1991LargeSD, Golse_2016_lecture, Spohn_1980_kinetic_micro}. Although this description sacrifices particle-level resolution, it offers substantial analytical and computational advantages and has been extensively studied in the theory of interacting particle systems~\cite{Cardaliaguet_2019_MFG, Chaintron_Diez_2022_review1, Chaintron_Diez_2022_review2, Huang_2020_meanfield}.

From the perspective of inverse problems, however, the mean-field regime presents a subtle challenge~\cite{Zuazua_2019_agent_network, Zhu_Zuazua_2018_mac_control_mic}. Although mean-field models provide a tractable description of large systems, it is not \emph{a priori} clear whether inferring interaction kernels at the particle level is consistent with inferring them from the corresponding mean-field PDE. In particular, when interaction laws are identified via optimization-based formulations that minimize discrepancies between model predictions and observed data, it is natural to ask whether optimizing over a finite particle system yields results consistent with those obtained by optimizing directly at the mean-field level.

Addressing this question is the central focus of this paper. Specifically, we ask:
\vspace{3mm}
\begin{center}
\emph{What are the similarities and disparities between inversion at the particle level and inversion at the mean-field level?}
\end{center}
\vspace{3mm}

Our analysis is built upon two sets of theories.

The first set of theories is the variational formulation of inverse problems, in which unknown coefficient functions are inferred from observational data by minimizing suitable objective functionals and analyzing the associated first-order optimality conditions. This formulation can be posed both at the particle level, where the dynamics is governed by finite-dimensional systems of ordinary differential equations, and at the continuum level, where the evolution is described by a partial differential equation for the underlying probability measure. Evaluating the similarities and differences between the two inverse problems then amounts to tracking the adjoint systems involved in the variational formulations.

The second technical preparation is mean-field analysis that rigorously connects the particle system to its continuum limit as the particle number $N \to \infty$. We leverage existing results on the derivation and analyzing the mean-field limit, and utilize them in our setting. In particular, we rely on the fact that for interactive particles, under mild regularity assumptions on the interaction kernel, the difference between the particle presentation and the mean-field PDE decays at the Monte Carlo rate $\mathcal{O}(N^{-1/2})$ in the weak sense.

These two sets of theories are instrumental for us to obtain the following results:
\begin{itemize}
    \item We start with a linear transport equation where particles do not interact. In particular, the transport equation:
\begin{equation}\label{eq:transport_eq_f}
\partial_t f(t,x) + \nabla_x \cdot \big(a(x) f(t,x)\big) = 0\,,
\qquad f(0,x) = f_0(x)\,,
\end{equation}
admits a Lagrangian formulation in terms of the characteristic flow
\begin{equation}\label{eq:transport_forward_X}
\dot{X}(t) = a(X(t))\,, \qquad X(0) = X_0\,.
\end{equation}
If $X_0$ is distributed according to the density $f_0$, then the law of $X(t)$ has the density $f(t,\cdot)$ for all $t$. The objective is to recover the velocity field $a(x)$. In this linear setting, we show that the gradient obtained from the particle-based formulation is unbiased with respect to the gradient derived from the continuum formulation, and that this equivalence holds for any finite number of particles~\cite{Hoogenboom_1977_adjoint_MC}.
\item We then extend our analysis to particles that interact. In the mean-field regime, this deduces a nonlinear equation:
\begin{equation}\label{eq:df_wr_mean_field}
\partial_t f + \nabla_x \cdot \big((w \ast f)\, f\big) = 0\,,
\qquad f(0,x) = f_0(x)\,,
\end{equation}
with its particle presentation given as:
\begin{equation}\label{eq:dX_wr_interact}
\dot{X}_i = \frac{1}{N} \sum_{j=1,\, j\neq i}^N w(X_i - X_j)\,,
\qquad X_i(0) = X_{i,0}\,.
\end{equation}
The goal is to recover the interaction kernel $w:\mathbb{R}^d \to \mathbb{R}^d$. Unlike in the linear case, the finite $N$ reconstruction is a biased estimate of its associated mean-field reconstruction, but this bias is small when $N\to\infty$.
\end{itemize}

The remainder of the paper is organized as follows. In Section~\ref{sec:linear_transport_equation} we present our first result. We will derive the variational approach for linear transport equations and establish gradient consistency between particle and continuum formulations. Section~\ref{sec:interaction} is dedicated to kernel reconstruction for interacting particles in the mean-field regime. We highlight the emergence of bias and its asymptotic resolution in the mean-field limit. Finally, Section~\ref{sec:numerics} presents numerical experiments that validate the theoretical results, confirm the predicted Monte Carlo scaling with respect to the number of particles, and explore the practical performance of gradient-based methods for interaction kernel reconstruction.

\section{First-Order Optimality for Particle and PDE Inverse Problems}
\label{sec:linear_transport_equation}



Inferring unknown parameters in dynamical system is formulated as optimization, and when the underlying dynamics admit both an Eulerian (PDE) and a Lagrangian (particle) description, their roles in inverse problems may not be automatically equivalent. This section is dedicated to examining this relation for linear transport equations with the velocity field as the to-be-reconstructed unknown parameter.

While the problem is deliberately formulated in a simple manner, this model admits both a PDE formulation and an equivalent particle formulation via characteristics. Hence, it is an ideal setting in which we can isolate and analyze the structural features of optimization-based inversion. The results presented here are of independent interest and will serve as a conceptual and technical foundation for the interacting particle systems studied in subsequent sections.

We consider a density field evolving under an unknown velocity field. We assume access to the initial density and to a macroscopic observation of the solution at a final time $t=T$, tested against a prescribed observable $\nu$.

The governing equation is the linear transport equation
\[
\partial_t f(t,x) + \nabla_x \cdot (a(x) f(t,x)) = 0,
\qquad f(0,x) = f_0(x),
\]
where the velocity field $a(x)$ is an unknown function on $\mathbb{R}^d$. Given the initial condition $f_0(x)$ and the measurement
\begin{equation}\label{eq:Jf_nu_measure}
\mathcal{J}[f] := \int \nu(x)\, f(T,x)\, \d x,
\end{equation}
the inverse problem consists in inferring $a(x)$. When formulated as an optimization problem, this leads to the following PDE-constrained optimization:
\begin{equation}\label{eq:df_ax_PDE_opt}
\begin{aligned}
\min_{a(\cdot)} \quad & \mathcal{J}[f] \\
\text{subject to} \quad &
\begin{cases}
\partial_t f + \nabla_x \cdot (a(x) f) = 0, \\
f(0,x) = f_0(x).
\end{cases}
\end{aligned}
\end{equation}
Since the solution $f(T,\cdot)$ is uniquely determined by the velocity field $a(x)$ and the initial distribution $f_0$, we may transfer the dependence and abuse notation by writing $\mathcal{J}$ as $\mathcal{J}[a; f_0]$. A standard approach to solving this optimization problem is to evolve the predicted $a$ according to the gradient flow
\[
\partial_s a(x) = -\frac{\delta \mathcal{J}}{\delta a}(x; f_0),
\]
where $s$ denotes the artificial learning (or training) time, and $f_0$ is treated as a fixed parameter. 

The same problem can be formulated in a particle (Lagrangian) framework. The transport dynamics are described by the characteristic equation
\begin{equation*}
\dot{X}(t) =\frac{\d}{\d t} X(t)= a(X(t))\,, \qquad X(0) = X_0\,.
\end{equation*}

The measurement now takes the form of
\begin{equation}\label{eq:JX_nu_measure}
\mathsf{J}[X(T)] := \nu(X(T))\,,
\end{equation}
and the corresponding inverse problem becomes an ODE-constrained optimization problem:
\begin{equation}\label{eq:dX_aX_ODE_opt}
\begin{aligned}
\min_{a(\cdot)} \quad & \mathsf{J}[X(T)] \\
\text{subject to} \quad &
\begin{cases}
\dot{X}(t) = a(X(t))\,, \\
X(0) = X_0\,.
\end{cases}
\end{aligned}
\end{equation}
As in the PDE setting, the terminal state $X(T)$ is uniquely determined by $a(x)$ and the initial state $X_0$. We therefore write $\mathsf{J}$ as $\mathsf{J}[a; X_0]$ and consider the associated gradient flow
\[
\partial_s a(x) = -\frac{\delta \mathsf{J}}{\delta a}(x; X_0),
\]
where $s$ again denotes the learning time.

The PDE and particle formulations describe the same underlying dynamics and are therefore expected to be closely related. In particular, if $f_0 = \delta(x-{X_0})$, then the time-$t$ PDE solution takes the form:
\begin{equation}\label{eq:characteristics}
f(t,x) =
\delta(x-X(t))\,.
\end{equation}
This correspondence is reflected in the objective functionals. Indeed, comparing~\eqref{eq:Jf_nu_measure} and~\eqref{eq:JX_nu_measure}:
\[
\mathsf{J}[X(T)] = \int_{\mathbb{R}^d} \nu(x)\, \delta(x-{X(T)})\, \d x
= \mathcal{J}[f]\,,
\]
where we see a direct relation in terms of the objective functions. A natural question is how the gradients $\delta \mathcal{J}/\delta a$ and $\delta \mathsf{J}/\delta a$ are related in the two formulations. 

\begin{theorem}\label{thm:gradient_consistency}
The variational derivatives of the objective functionals in the PDE and particle formulations are consistent in the following sense:
\begin{itemize}
    \item[a.] If $f_0 = \delta(x-{X_0})$, then
    \[
    \frac{\delta \mathcal{J}}{\delta a}(x; f_0)
    = \frac{\delta \mathsf{J}}{\delta a}(x; X_0).
    \]
    \item[b.] For general $f_0$, one has
    \[
    \frac{\delta \mathcal{J}}{\delta a}(x; f_0)
    = \int \frac{\delta \mathsf{J}}{\delta a}(x; y)\, f_0(y)\, \mathrm{d}y.
    \]
\end{itemize}
\end{theorem}

To prove part (a) of this theorem, one only needs to recognize that $\mathcal{J}[a;\delta_{X_0}]=\mathsf{J}[a;X_0]$ for every $a$, their first variations against $a$ thus must be equal. Part (b) can be deduced by recognizing $\mathcal{J}[a;f_0]$ as a linear functional over the initial data $f_0$, and thus $\frac{\delta\mathcal{J}}{\delta a}$ is a linear operator.

Another route to prove this theorem is to explicitly express the two variational derivatives using their adjoints. This route is more versatile and can be used in nonlinear settings in later discussions. To facilitate these discussions, we present the derivation in the following two lemmas.

\begin{lemma}\label{prop:PDE_gradient}
For a fixed initial density $f_0$, the first variation of $\mathcal{J}$ with respect to the velocity field $a(\cdot)$ is given by
\begin{equation}\label{eq:dJ_da_transport_fg}
\frac{\delta \mathcal{J}}{\delta a}(x)
= \int_0^T \nabla_x g(t,x)\, f(t,x)\, \mathrm{d}t\,,
\end{equation}
where $g$ satisfies the adjoint transport equation:
\begin{equation}\label{eq:transport_adj_eq_g}
\partial_t g(t,x) + a(x)\cdot \nabla_x g(t,x) = 0\,,
\qquad
g(T,x) = \nu(x)\,.
\end{equation}
\end{lemma}

\begin{lemma}\label{prop:particle_gradient}
For a fixed initial data $X_0$, the first variation of $\mathsf{J}$ with respect to $a(\cdot)$ admits the representation
\begin{equation}\label{eq:dJ_da_transport_Y}
\frac{\delta \mathsf{J}}{\delta a}(x) = \int_0^T Y(t)\delta\!\left(x-X(t)\right)\,\d t\,, 
\end{equation}
where $Y(t)$ solves:
\begin{equation}\label{eq:transport_adjoint_Y}
\dot{Y}(t) + \nabla_x a(X(t))\, Y(t) = 0\,,
\qquad
Y(T) = \nabla_x \nu(X(T))\,.
\end{equation}
\end{lemma}
The proof for these two lemmas is the standard calculus-of-variation argument. We leave the proof in Appendix~\ref{app:A}.
These lemmas provide essential calculations to connect the two variations in Theorem~\ref{thm:gradient_consistency}. It suggests that to connect $\frac{\delta \mathcal{J}}{\delta a}$ and $\frac{\delta \mathsf{J}}{\delta a}$, one needs to carefully examine the relation between the two adjoint variables $g$ and $Y$.

\begin{proof}[Proof for Theorem~\ref{thm:gradient_consistency}]
We prove the two items separately.
\begin{itemize}
    \item[--]Proof for a. To show the equivalence between $\frac{\delta \mathcal{J}}{\delta a}$ and $\frac{\delta \mathsf{J}}{\delta a}$, we are to compare the two formula~\eqref{eq:dJ_da_transport_fg} and~\eqref{eq:dJ_da_transport_Y}. Noting in this case, we already have~\eqref{eq:characteristics}, reducing $\frac{\delta \mathcal{J}}{\delta a}$ to be $\int \nabla_xg(t,X(t))\delta(x-X(t))\,\d t$. To equate the two variational derivatives amounts to showing $\nabla_x g(t,X(t))= Y(t)$, or equivalently, showing $h_j(t,X(t)) = Y_j(t)$ for all $j$ where we denote $h_j=\partial_j g$.
    
    To do so, we take $j$-th derivative of~\eqref{eq:transport_adj_eq_g}:
\begin{equation}
\partial_t h_j + a(x)\cdot \nabla_x h_j + \nabla_x a(x) \cdot h = 0\,.
\end{equation}
Along the characteristics $X(t)$, according to~\eqref{eq:transport_forward_X}, the equation becomes
\begin{equation}
\frac{\d h_j}{\d t}(t,X(t)) 
+ \nabla a(X(t))\cdot h(t,X(t)) = 0\,,
\end{equation}
that coincides with~\eqref{eq:transport_adjoint_Y}. Furthermore, $h(T,X(T)) =\nabla_xg(T,X(T))=\nabla_x\nu(X(T))= Y(T)$, therefore $h(t,X(t))=Y(t)$ for all $t$, concluding the proof.

\item[--]Proof for b. This is achieved by noticing the linear dependence on $f$ of $\frac{\delta \mathcal{J}}{\delta a}$ (as seen in~\eqref{eq:dJ_da_transport_fg}) and the linear dependence on $f_0$ of $f$ (as seen in~\eqref{eq:transport_eq_f}). The linear relation thus pass on, making:
\[
\frac{\delta \mathcal{J}}{\delta a}(x;f_0)=\int  \frac{\delta \mathcal{J}}{\delta a}(x;\delta(x-y))f_0(y)\,\d{y}\,.
\]
Substituting $\frac{\delta \mathcal{J}}{\delta a}$ by $\frac{\delta \mathsf{J}}{\delta a}$ according to case a., the statement is shown.
\end{itemize}
\end{proof}

Although the analysis in this section is carried out for a simple linear transport equation, it reveals several features that are essential for understanding inverse problems in more general particle systems. In particular, the relationship between the adjoint variables $g$ and $Y$ is subtle and, at first glance, counterintuitive. While the particle trajectory $X(t)$ provides a Lagrangian representation of the density $f(t,\cdot)$, the particle-level adjoint variable $Y(t)$ is \emph{not} a Lagrangian representation of the PDE adjoint $g(t,\cdot)$. Instead, the two are related through the identity
\begin{equation}\label{eq:adjoint_dg_Y}
\nabla_x g(t,X(t)) = Y(t)\,,
\end{equation}
which holds along characteristics.

This structural distinction between forward variables and adjoint variables plays a crucial role in connecting particle-level and continuum-level optimization problems. As we show in later sections, the same phenomenon persists in interacting particle systems, where it underlies the discrepancy between particle-based and mean-field gradients and ultimately leads to biased estimators at finite particle numbers.

\section{Inverse Problems for Interacting Particles in the Mean-Field Regime}\label{sec:interaction}

This section is dedicated to the recovery of interaction kernels in systems of interacting particles operating in the mean-field regime. We consider dynamics generated by a large number $N \gg 1$ of indistinguishable agents, where each binary interaction is weak but collectively produce nontrivial macroscopic behavior. In this scaling, the influence of any single particle on the system is negligible, and the cumulative effect of interactions can be represented through an averaged, or mean-field, description.

From a modeling perspective, the mean-field regime provides a natural bridge between microscopic particle dynamics and macroscopic continuum equations governing the evolution of probability distributions. From the standpoint of inverse problems, however, this regime introduces fundamental challenges. While the forward mean-field limit offers a tractable and well-understood description of the system dynamics, it is not a priori clear whether interaction laws inferred from macroscopic observations are consistent with those governing the underlying finite-particle system.

Before formulating and analyzing the inverse problem of kernel recovery, we therefore review, in subsection~\ref{sec:back_mean_field}, the forward theory of interacting particle systems in the mean-field regime. We then utilize these results to study the associated two inverse problems in subsection~\ref{sec:nonlinear_particle_pde}.

\subsection{Background on Forward Problems for Interacting Particle Systems}\label{sec:back_mean_field}
The analysis of mean-field limits has a long history and has attracted substantial interest over the past decades; see, for example, the survey~\cite{PE_Jabin_2014_mean_field_review}. Much recent progress has focused on treating singular interaction kernels, including Coulomb-type potentials~\cite{Serfaty_2020_mean_field_coulomb, Nguyen_Rosenzweig_Serfaty_2022_mean_field}. In settings with sufficiently regular interactions, classical approaches date back to~\cite{Dobrushin_1979_vlasov, Sznitman_1991_prop_chaos, Meleard_1996_McKean_Vlasov}. We briefly summarize the framework and results most relevant for our subsequent analysis.

Let $w:\mathbb{R}^d \to \mathbb{R}^d$ denote an interaction kernel, where $w(r)$ represents the force exerted by a particle located at displacement $r$ relative to another particle. We assume that interactions are pairwise and translation invariant, meaning the kernel depends only on the relative position $r = X_i - X_j$. The dynamics of the $N$-particle system are given by the coupled ODEs
\begin{equation}\label{eq:dX_couple_w}
\dot{X}_i = \frac{1}{N} \sum_{j=1,\, j\neq i}^N w(X_i - X_j)\,,
\qquad
X_i(0) = X_{i,0}\,.
\end{equation}

It is convenient to introduce the empirical measure
\begin{equation}\label{eq:empirical}
f_N(t,x) = \frac{1}{N} \sum_{i=1}^N \delta(x - X_i(t))\,,
\end{equation}
so that $f_N(0,\cdot) = \frac{1}{N} \sum_{i=1}^N \delta(\cdot - X_{i,0})$. With this notation, the particle system can be written compactly as
\[
\dot{X}_i = (w * f_N)(X_i)\,.
\]
This suggests that in the mean-field limit, the velocity field is the convolution of the kernel and the probability density $f(t,\cdot)$ itself, forming the nonlinear transport equation:
\begin{equation}\label{eq:dtf_wf_mean_field}
\partial_t f + \nabla_x \cdot \big( (w * f) f \big) = 0\,,
\qquad
f(0,x) = f_0(x)\,.
\end{equation}


Since \eqref{eq:dX_couple_w} provides a Lagrangian description of the mean-field PDE~\eqref{eq:dtf_wf_mean_field}, one expects a notion of equivalence between $f_N$ and $f$. The precise meaning of this equivalence, however, depends crucially on the choice of initial data.

\begin{itemize}
\item[\textbf{Setting A}.]
If the initial density coincides exactly with the empirical measure, $f_0 = f_N(0,\cdot)$, then this equivalence propagates in time:
\[
f(t,\cdot) \equiv f_N(t,\cdot)\,.
\]
In this case, $f_N$ is an exact measure-valued solution of the mean-field PDE.

\item[\textbf{Setting B}.]
If the initial particle positions are sampled independently according to $f_0$, i.e.\ $X_{i,0} \overset{\mathrm{i.i.d.}}{\sim} f_0$, then the empirical measure is unbiased initially but does not remain so at later times~\cite[Theorem 2]{Lacker_2023_mean_field_conv_rate}: 
\[
f_0 = \mathbb{E}[f_N(0,\cdot)]
\quad \not\!\!\implies \quad
f(t,\cdot) = \mathbb{E}[f_N(t,\cdot)]\,.
\]
Nevertheless, under suitable regularity assumptions on $w$, unbiasedness is recovered in the limit $N \to \infty$,
\[
f(t,\cdot) = \lim_{N\to\infty} \mathbb{E}[f_N(t,\cdot)]\,.
\]
\end{itemize}

We impose the following assumptions on the interaction kernel.

\begin{assumption}\label{assump:property_w(r)}
The interaction kernel $w:\mathbb{R}^d \to \mathbb{R}^d$ satisfies:
\begin{enumerate}
\item $w(-r) = -w(r)$ and $w(0) = 0$;
\item $w(\cdot) \in C_b^2$, i.e., there exists some positive constant $\Cbdd>0$ such that 
    \begin{equation}\label{eq:assump_Cb2}
    \|w\|_{L^{\infty}} + \|\nabla_r w\|_{L^{\infty}} + \|\nabla_r^2 w\|_{L^{\infty}} \leq \Cbdd\,.
    \end{equation}
\end{enumerate}
\end{assumption}
As a consequence, $w$ and $\nabla_r w$ are globally Lipschitz, with the Lipschitz constant denoted by $\CLip > 0$.

Under these assumptions, the classical mean-field convergence results hold.

\begin{theorem}\label{thm:mean_field_fN_f}
Let $\{X_i\}$ solve the particle system {with $f_N$ being its empirical measure,} and let $f$ solve the mean-field PDE. If Assumption~\ref{assump:property_w(r)} holds and $X_{i,0} \overset{\mathrm{i.i.d.}}{\sim} f_0$, then:
\begin{itemize}
\item[(i)] For any test function $\varphi \in C_b^1(\mathbb{R}^d)$,
\begin{equation}\label{eq:phi_weak_fN_f_err}
\mathbb{E}\left|\langle \varphi, (f_N - f)(t,\cdot) \rangle\right|^2
\le \frac{C(t,\CLip,\|\varphi\|_{C^1},f_0)}{N};
\end{equation}
\item[(ii)] The empirical measure converges in the Wasserstein distance,
\begin{equation}\label{eq:wass_dist_f_N_f}
\mathbb{E}\big(\mathcal{W}_p(f_N,f)\big)
\le C(d,p,f)\, N^{-\lambda},
\end{equation}
where $\lambda$ depends on $p$, the dimension $d$, and the moment bounds of $f$.
\end{itemize}
\end{theorem}

Such results can be established either through coupling arguments~\cite{Sznitman_1991_prop_chaos, Fournier_Guillin_2015_empirical} or through stability estimates for the mean-field PDE~\cite{Dobrushin_1979_vlasov}. These convergence properties will be repeatedly invoked in the inverse problem analysis that follows.

\subsection{Variational Formulation for Particle and PDE Inverse Problems}\label{sec:nonlinear_particle_pde}

Building on the optimization framework introduced in Section~\ref{sec:linear_transport_equation} and the mean-field theory summarized in Section~\ref{sec:back_mean_field}, we now turn to the central inverse problem of this paper: the reconstruction of the interaction kernel $w$ from macroscopic observations. When the system is described at the continuum level, this leads to the following PDE-constrained optimization problem:
\begin{equation}\label{eq:Jf_wr_continuous_opt}
\begin{aligned}
\min_{w(\cdot)} \quad & \mathcal{J}[f] = \int_{\mathbb{R}^d} \nu(x)\, f(T,x)\,\mathrm{d}x\,, \\
\text{subject to} \quad &
\begin{cases}
\partial_t f + \nabla_x \cdot \big((w * f)\, f\big) = 0\,, \\
f(0,x) = f_0(x)\,.
\end{cases}
\end{aligned}
\end{equation}

At the particle level, the corresponding inverse problem is formulated as an ODE-constrained optimization problem:
\begin{equation}\label{eq:JX_wr_particle_opt}
\begin{aligned}
\min_{w(\cdot)} \quad &
\mathsf{J}[\{X_i\}] = \frac{1}{N} \sum_{i=1}^N \nu(X_i(T))\,, \\
\text{subject to} \quad &
\begin{cases}
\dot{X}_i = \frac{1}{N} \sum_{\substack{j=1 \\ j\neq i}}^N w(X_i - X_j)\,, \\
X_i(0) = X_{i,0}\,, \quad i = 1,\dots,N\,.
\end{cases}
\end{aligned}
\end{equation}

When the underlying density $f$ coincides with the empirical measure $f_N$, the two objective functionals agree, in the sense that $\mathcal{J}[f_N] = \mathsf{J}[\{X_i\}]$. However, as in the forward mean-field analysis, the equivalence of the associated first variations is more subtle. In particular, the relationship between the gradients $\delta \mathcal{J}/\delta w$ and $\delta \mathsf{J}/\delta w$ depends critically on the choice of initial data and on the behavior of the system in the mean-field limit.

In the following subsections, we examine this relationship in detail and quantify the discrepancy between the particle-level and continuum-level formulations of the inverse problem.

\subsubsection{Finite particles}\label{sec:interaction_finite}

This subsection is dedicated to \textbf{Setting A} where one has finitely many particles. The initial configuration is $f_0 =f_N(0,\cdot)= \frac{1}{N}\sum_{i=1}^N \delta(x-X_{i,0})$, and the empirical measure $f_N(t,x)$ is a measure-valued solution of~\eqref{eq:df_wr_mean_field}. Then the two notions of the first variations are completely equivalent:
\begin{theorem}\label{thm:fN_particle_gradient_equal}
Let $\mathcal{J}$ and  $\mathsf{J}$  be defined in~\eqref{eq:Jf_wr_continuous_opt} with the initial condition $f_0$, and ~\eqref{eq:JX_wr_particle_opt} with the initial condition $\{X_{i,0}\}_{i=1}^N$, respectively. If $f_0 = \frac{1}{N}\sum_{i=1}^N \delta(x-X_{i,0})$, then:
\begin{equation}
\begin{aligned}
\frac{\delta \mathcal{J}}{\delta w}(r)
= \frac{\delta \mathsf{J}}{\delta w}(r) \,. 
\end{aligned}
\end{equation}
\end{theorem}
The proof of this theorem relies on the following lemmas.

\begin{lemma}\label{lem:continuous_dj/dw}
The first variation of $\mathcal{J}$ with respect to $w(r)$ is:
\begin{equation}\label{eq:nonlinear_dJdw_fg}
\frac{\delta \mathcal{J}}{\delta w}(r) = \int_0^T \int f(t,x-r) f(t,x) \nabla_x g(t, x)\,\d x \d t\,,
\end{equation}
where $g$ solves the adjoint system of~\eqref{eq:df_wr_mean_field}:
\begin{equation}\label{eq:nonlinear_adj_g}
\begin{aligned}
\partial_t g - w\ast (\nabla_x g f) + (w\ast f)\cdot \nabla_x g = 0\,, \quad
g(T,x) = \nu(x)\,.
\end{aligned}
\end{equation}
\end{lemma}

\begin{lemma}\label{lem:particle_dj/dw}
The first variation of $\mathsf{J}$ with respect to $w(r)$ is
\begin{equation} \label{eq:nonlinear_gradient_NdYi}
    \dfrac{\delta \mathsf{J}}{\delta w}(r) 
    = \frac{1}{N^2}\sum_{i=1}^N \sum_{j\not=i} \int_0^T Y_i(t)\delta(r-(X_i-X_j)(t))\,\d t \,,
\end{equation}
where $Y_i$ solves the following adjoint system of~\eqref{eq:dX_wr_interact}:
    \begin{equation} \label{eq:nonlinear_adjoint_NdYi}
    \dot{Y}_i 
    + \frac{1}{N}\sum_{j\not=i}^N \nabla w(X_i-X_j) Y_i
    - \frac{1}{N}\sum_{j\not=i}^N \nabla w(X_j-X_i) Y_j = 0\,,  \quad 
    Y_i(T) = \nabla_x \nu(X_i(T))\,.
    \end{equation}
    
\end{lemma}

\begin{proof}[Proof of Lemma~\ref{lem:continuous_dj/dw}]
Denote $g(t,x), g_0(x)$ as the adjoint states, and define the following Lagrangian:
\begin{equation}
\begin{aligned}
\mathcal{L}[f,f_{in},w,g,g_0] 
= & \int (\nu(x) - g(T,x))f(T,x)\,\d x + \int f(0,x) (g(0,x)-g_0) + f_0(x)g_0 \,\d x \\
& + \int_0^T \int (\partial_t g) f
+ (w\ast f)f \cdot \nabla_x g\,\d x\d t \,.
\end{aligned}
\end{equation}
In order to obtain the adjoint equation, we first linearize the Lagrangian with respect to $f$,
\begin{equation*}
\begin{aligned}
\mathcal{L}[f + \eps\tilde f] - \mathcal{L}[f] = & \eps\int_0^T \int (\partial_t g) \tilde f\,\d x \d t 
- \eps \int_0^T \int (w\ast (\nabla_x g f)) \tilde f\,\d x\d t \\
& + \eps \int_0^T \int (w\ast f)\tilde f \cdot \nabla_x g\,\d x\d t + O(\eps^2)\,,
\end{aligned}
\end{equation*}
then let $\eps\to0$:
\begin{equation}
\begin{aligned}
\frac{\delta \mathcal{L}}{\delta f}(t,x) = & \partial_t g - w\ast (\nabla_x g f) + (w\ast f)\cdot \nabla_x g\,,
\end{aligned}
\end{equation}
where $w\ast (\nabla_x g f)(y) = \sum_{i=1}^d w_i(x-y) \partial_{x_i} g(t,x) f(t,x)$.
The variation with respect to $f(T)$ is straightforward:
\begin{equation}
\begin{aligned}
\frac{\delta \mathcal{L}}{\delta f|_{t=T}}(x) &= \nu(x) - g(T,x)\,.
\end{aligned}
\end{equation}
Moreover, the variation with respect to $w(r)$ is computed by
\begin{equation}
\begin{aligned}
\mathcal{L}[w+\eps \tilde{w}] - \mathcal{L}[w] = & \eps \int_0^T \int_{\R^d} \int_{\R^d} \tilde{w}(x-y) f(t,y) f(t,x) \cdot \nabla_x g(t, x)\,\d y\d x \d t\,, \notag \\
= & \eps \int_0^T \int_{\R^d} \int_{\R^d} \tilde{w}(r) f(t,x-r) f(t,x) \cdot \nabla_x g(t, x)\,\d r\d x \d t\qquad r:=x-y, \notag
\end{aligned}
\end{equation}
where we recall that $\left|\frac{\partial y}{\partial r}\right|=1$, and letting $\eps\to0$ gives the desired quantity~\eqref{eq:nonlinear_dJdw_fg}. 
\end{proof}

\begin{proof}[Proof of Lemma~\ref{lem:particle_dj/dw}]
Similarly, denote $Y_i, Y_i^0$ as the Lagrange multipliers, and $\mathsf{L}$ as the Lagrangian function,
\begin{equation}
\begin{aligned}
\mathsf{L}[X_i,X_i^{in},w,Y_i,Y_i^0] = & 
\frac{1}{N} \sum_{i=1}^N \nu(X_i(T)) 
- \frac{1}{N}\sum_{i=1}^N X_i(T)Y_i(T)  \\
& + \frac{1}{N}\sum_{i=1}^N  \int_0^T \dot{Y}_i X_i \,\d t 
+ \frac{1}{N^2}\sum_{i=1}^N \sum_{j\not=i}^N \int_0^T w(X_i-X_j) Y_i\,\d t \\
& + \frac{1}{N}\sum_{i=1}^N  X_i(0)Y_i(0) - \frac{1}{N}\sum_{i=1}^N \left(X_i(0) - X_{i,0}\right)Y_i^0\,.
\end{aligned}
\end{equation}
and we take the variation of each $X_i, X_i(T)$,
\begin{equation}
\begin{aligned}
\frac{\delta \mathsf{L}}{\delta X_i}(t)
= & \frac{1}{N} \dot{Y}_i 
+ \frac{1}{N^2}\sum_{j\not=i}^N \nabla w(X_i-X_j) Y_i
- \frac{1}{N^2}\sum_{j\not=i}^N \nabla w(X_j-X_i) Y_j\,, \\ 
\frac{\delta \mathsf{L}}{\delta X_i|_T}
= & \frac{1}{N} \nabla_x\nu(X_i(T)) - \frac{1}{N}Y_i(T)\,.
\end{aligned}
\end{equation}
The variation with respect to $w$ is
\begin{equation}
\begin{aligned}
\mathsf{L}[w+\eps \tilde{w}] - \mathsf{L}[w] = & \eps \frac{1}{N^2}\sum_{i=1}^N \sum_{j\not=i}^N \int_0^T \tilde{w}(X_i-X_j) Y_i(t)\,\d t \,,\\
= & \eps \frac{1}{N^2}\sum_{i=1}^N \sum_{j\not=i}^N \int_0^T \int \tilde{w}(r) Y_i(t)\delta(r-(X_i(t)-X_j(t)))\,\d r\d t \,,
\end{aligned}
\end{equation}
then take $\eps\to0$ we obtain $\frac{\delta\mathsf{L}}{\delta w}$ as in~\eqref{eq:nonlinear_gradient_NdYi}.
\end{proof}

The relation between adjoint states~\eqref{eq:adjoint_dg_Y} also appears in the interacting setting, which we will justify in proving Theorem~\ref{thm:fN_particle_gradient_equal}. To this end, denote $h(t,x):=\nabla_x g(t,x) \in \R^d$ for given $g(t,x)$ solving~\eqref{eq:nonlinear_adj_g}, then it satisfies
\begin{equation}\label{eq:nonlinear_adj_dxg_h}
\partial_t h + (w\ast f)\cdot \nabla_x h - \nabla w\ast (h f) + (\nabla w\ast f) \cdot h  = 0\,,\quad 
h(T,x) = \nabla_x \nu (x)\,.
\end{equation}
Moreover, for any empirical measure $f_N(t,x):=\frac{1}{N}\sum_{i=1}^N \delta(x-X_i(t))$, we denote $h_N(t,x)$ as the classical solution to the following:
\begin{equation}\label{eq:nonlinear_fN_induce_hN}
\partial_t h_N + (w\ast f_N) \cdot \nabla_x h_N - \nabla w\ast (h_N f_N) + (\nabla w\ast f_N) \cdot h_N = 0\,,\quad 
h_N(T,x) = \nabla_x \nu(x)\,.
\end{equation}
Note that in~\eqref{eq:nonlinear_fN_induce_hN}, $f_N$ always appears in the form of convolution, hence the coefficients $(w\ast f_N)=\frac{1}{N} \sum_{i=1}^N w(x-X_i(t))$, $(\nabla w\ast f_N)=\frac{1}{N} \sum_{i=1}^N \nabla w(x-X_i(t))$ are globally bounded by Assumption~\ref{assump:property_w(r)}. The well-posedness of~\eqref{eq:nonlinear_fN_induce_hN} can be shown by the semigroup property (see later sections).


\begin{proof}[Proof of Theorem~\ref{thm:fN_particle_gradient_equal}]
We can first show that with $f_N=\frac{1}{N}\sum_{i=1}^N\delta(x-X_i(t))$, it holds that
\begin{equation}\label{eq:Yi=1/N_dgdx}
Y_i(t) = h_N(t,x)|_{x=X_i(t)}\,.
\end{equation}
First note that $X_i(t)$ that solves~\eqref{eq:dX_couple_w} is a characteristics of~\eqref{eq:nonlinear_fN_induce_hN}. Then along this characteristics, 
we have:
\begin{equation}\label{eq:nolinear_dhdt_Xi}
\begin{aligned}
\frac{\d}{\d t} h_N(t,X_i(t)) 
& + \frac{1}{N}\sum_{j=1}^N  \nabla w(X_i(t)-X_j(t)) \cdot h_N(t,X_i(t)) \\
& - \frac{1}{N}\sum_{j=1}^N  \nabla w(X_i(t)-X_j(t))  \cdot h_N(t,X_j(t))
= 0\,,
\end{aligned}
\end{equation}
which is exactly the equation~\eqref{eq:nonlinear_adjoint_NdYi} that $Y_i(t)$ solves. Moreover, as their final time conditions coincide, $Y_i(T)=h_N(T,X_i(T))$ (recall~\eqref{eq:nonlinear_adjoint_NdYi} and~\eqref{eq:nonlinear_fN_induce_hN}), we conclude that $Y_i(t)=h_N(t,X_i(t))$.

Then by direct computation and the formulation of the two first variations given in Lemmas~\ref{lem:continuous_dj/dw} and~\ref{lem:particle_dj/dw}, we have
\begin{equation*}
\begin{aligned}
\frac{\delta \mathcal{J}}{\delta w}(r) & =\int_0^T \int f_N(t,x-r) f_N(t,x) h_N(t,x)\,\d x\d t \,,\\
& = \frac{1}{N^2} \sum_{i,j=1}^N \int_0^T \int  \delta(x-r-X_j(t)) \delta(x-X_i(t)) h_N(t,x)\,\d x\d t \,,\\
& = \frac{1}{N^2} \sum_{i,j=1}^N \int_0^T \delta(X_i(t)-X_j(t)-r) h_N(t,X_i(t))\,\d t \,. 
\end{aligned}
\end{equation*}
The conclusion is shown by applying~\eqref{eq:Yi=1/N_dgdx}.
\end{proof}

\subsubsection{Infinite particles: mean-field}\label{sec:mean-field}
This subsection is dedicated to \textbf{Setting B} where one draws finitely many particles from an underlying distribution: $X_{i,0}\overset{\text{i.i.d.}}{\sim} f_0$. According to Theorem~\ref{thm:mean_field_fN_f}, in the forward simulation, the empirical distribution $f_N$ is not an unbiased estimator of $f$, but converges to it in the weak sense, with $\frac{1}{\sqrt{N}}$ being the convergence rate. We are to translate this result into the inverse problem setting, with the goal being to show the two first variations $\frac{\delta \mathcal{J}}{\delta w}$ and $\frac{\delta \mathsf{J}}{\delta w}$ are also within $\frac{1}{\sqrt{N}}$ distance.

\begin{theorem} \label{thm:gradient_err_estimate} 
Let $\mathcal{J}$ be defined in~\eqref{eq:Jf_wr_continuous_opt} with the initial condition $f_0$. Let $\mathsf{J}$ be defined in~\eqref{eq:JX_wr_particle_opt} with the initial condition $\{X_{i,0}\}_{i=1}^N$ where $X_{i,0}\overset{\text{i.i.d.}}{\sim} f_0$. If $w$ satisfies Assumption~\ref{assump:property_w(r)}, the following weak convergence holds: for any test function $\varphi(\cdot): \R^d\to \R^d, \varphi(\cdot) \in C^1_b(\R^d)$, there exists some positive constant $C=C(d,T,\varphi,f_0, \nu, \Cbdd, \CLip)$ such that,
\begin{equation}
\E
\, \left|\left\l \frac{\delta \mathcal{J}}{\delta w}- \frac{\delta \mathsf{J}}{\delta w}\,, \varphi \right\r_r \right|^2 \leq \frac{C}{{N}}\,.
\end{equation}
Here $\left\l \psi, \phi \right\r_r:= \sum_{j=1}^d \int_{\R^d} \psi_j(r)\phi_j(r)\,\d r$, and the expectation is taken over all realizations of the initial data for $\{X_{i,0}\}_{i=1}^N$.
\end{theorem}
\begin{proof}
For any fixed test function $\varphi$, by Lemmas~\ref{lem:continuous_dj/dw} and~\ref{lem:particle_dj/dw} and Theorem~\ref{thm:fN_particle_gradient_equal}, we rewrite $\left\l \frac{\delta \mathcal{J}}{\delta w}\,, \varphi \right\r_r - \left\l \frac{\delta \mathsf{J}}{\delta w}\,, \varphi \right\r_r $ as: 
\begin{equation}\label{eq:weak_sum_D123}
\begin{aligned}
\left\l \frac{\delta \mathcal{J}}{\delta w}\,, \varphi \right\r_r - \left\l \frac{\delta \mathsf{J}}{\delta w}\,, \varphi \right\r_r 
= &
\int \int_0^T \int (f-f_N)(t,x-r) f(t,x) h(t,x) \varphi(r)\,\d x\d t \d r \\
& + \int \int_0^T \int f_N(t,x-r) (f-f_N)(t,x) h(t,x) \varphi(r)\,\d x\d t \d r \\
& + \int \int_0^T \int f_N(t,x-r) f_N(t,x) (h-h_N)(t,x) \varphi(r)\,\d x\d t \d r \\
=: & D_1 + D_2 + D_3 \,.
\end{aligned}
\end{equation}
It can be shown in Propositions~\ref{prop:D1},~\ref{prop:D2}, and~\ref{prop:D3} that all three terms satisfy
\begin{equation}
\E |D_k|^2 \leq \frac{C_k}{N}\,,\quad  k = 1,2,3\,,
\end{equation}
with $C_k$ independent of $N$. By the triangle inequality, the proof is concluded.
\end{proof}
With the classical mean-field limit results Theorem~\ref{thm:mean_field_fN_f}, Theorem~\ref{thm:gradient_err_estimate} is not too surprising. The classical $O(N^{-1/2})$ weak convergence is preserved. Considering, according to Theorem~\ref{thm:fN_particle_gradient_equal}, $\frac{\delta \mathcal{J}}{\delta w}=\frac{\delta \mathsf{J}}{\delta w}$ for an empirical measure $f_0$ of $N$ particles, Theorem~\ref{thm:gradient_err_estimate} essentially states that the map between the initial data $f_0$ to the first variation $\frac{\delta \mathcal{J}}{\delta w}$ is Lipschitz when tested on the function space of $C_b^1$. 
Such Lipschitzness is also observed for the adjoint variable, in the sense that $h_N$ approximates $h$ for large $N$. To show this, recall~\eqref{eq:nonlinear_adj_dxg_h} and~\eqref{eq:nonlinear_fN_induce_hN}, and we define the operator that maps the forward variable to its adjoint:
\begin{equation}
\begin{aligned}
\mathcal{H}_{w}[\cdot]: L^{\infty}([0,T],\mathcal{P}(\R^d)) &\to L^{\infty}([0,T],C^1_b(\R^d))\,, \\
f(t,x) &\mapsto h(t,x)\,.
\end{aligned}
\end{equation}
This allows us to rewrite $h(t,x)=\mathcal{H}_w[f](t,x)$, and $h_N(t,x) = \mathcal{H}_w[f_N]$. Since $f_N\to f$ in the Wasserstein sense (Theorem~\ref{thm:mean_field_fN_f}), we are to prove the operator $\mathcal{H}_w$ preserves convergence rates:
\begin{proposition}[Mean-field for adjoint states] \label{thm:H_op_f_mean_field}
Let $f_N(t,x)=\frac{1}{N}\sum_{i=1}^N \delta(x-X_i(t))$ be the empirical measure where $\{X_i(t)\}$ solves~\eqref{eq:dX_wr_interact}, and $f$ solve~\eqref{eq:df_wr_mean_field}, with $X_{i,0}\overset{\text{i.i.d.}}{\sim} f_0$. Then there exists some constant $C=C(\CLip, \Cbdd, T,\nu)>0$ such that for any fixed $(t,x)\in[0,T]\times \R^d$:
\begin{equation}
\E\,\left|\mathcal{H}_w[f_{N}](t,x) - \mathcal{H}_w[f](t,x)\right| \leq C N^{-\lambda}\,, \quad N\to\infty\,,
\end{equation}
where the expectation is taken over the initial realization of $X_{i,0}$ and the convergence rate $\lambda$ depends on the dimension $d$.
\end{proposition}
The proof essentially is the stability estimate of the PDE~\eqref{eq:nonlinear_adj_dxg_h}. We leave the proof to Appendix~\ref{subsec:pf_prop_31_deltah}. 


\subsubsection{Auxiliary propositions for Theorem~\ref{thm:gradient_err_estimate}}
\begin{proposition} \label{prop:D1}
Under the assumptions of Theorem~\ref{thm:gradient_err_estimate}, there exists some positive constant $C_1=C_1(\|\varphi\|_{W^{1,\infty}}, T, \|f_0\|_{L^1}, \Cbdd, \|\nu\|_{W^{1,\infty}})$ such that
\begin{equation}
\E\,|D_1|^2 = \E\left|\int \int_0^T \int (f-f_N)(t,x-r) f(t,x) h(t,x) \varphi(r)\,\d x\d t \d r \right|^2
\leq \frac{C_1}{N}\,,
\end{equation}
{where $\|\varphi\|_{W^{1,\infty}}:=\max (\|\varphi\|_{L^{\infty}}, \|\nabla \varphi\|_{L^{\infty}})$ is defined on the space of bounded Lipschitz functions.}
\end{proposition}

\begin{proof}
By Fubini's theorem and the change of variable that $y:=x-r$, we obtain
\begin{equation}
\begin{aligned}
D_1 
& = \int_0^T \int_{\R^d} \int_{\R^d} (f-f_N)(t,y) f(t,x) h(t,x) \varphi(x-y)\, \d x\d y\d t\,,\\
& = \int_0^T \int_{\R^d} (f-f_N)(t,y) \psi(t, y)\,\d y\d t\,,
\end{aligned}
\end{equation}
where $\psi(t, y)$ is defined as follows,
\begin{equation}\label{eq:D1_psi_ty_def}
\psi(t,y):= \int_{\R^d} f(t,x) h(t,x) \varphi(x-y)\,\d x \,.
\end{equation}
To deploy Theorem~\ref{thm:mean_field_fN_f} and estimate~\eqref{eq:phi_weak_fN_f_err}, we need to show that $\psi\in C^1_b$. Note that by using the H\"older inequality, we have:
\begin{equation*}
\begin{aligned}
\|\psi(t, \cdot)\|_{L^{\infty}(\R^d)} & \leq \|\varphi\|_{L^{\infty}} \|f(t,\cdot)\|_{L^1(\R^d)} \|h(t,\cdot)\|_{L^{\infty}(\R^d)}\,,\\
\|\nabla_y \psi(t,\cdot)\|_{L^{\infty}(\R^d)} & \leq \|\nabla \varphi\|_{L^{\infty}} \|f(t,\cdot)\|_{L^1(\R^d)} \|h(t,\cdot)\|_{L^{\infty}(\R^d)} \,.
\end{aligned}
\end{equation*}
Since the forward equation~\eqref{eq:df_wr_mean_field} preserves mass and $L^1$ norm, we have $f\in L^{\infty}([0,T],L^1(\R^d))$, and as stated in Lemma~\ref{thm:semigroup_At_h_xi_sol}, 
$h\in L^{\infty}([0,T],L^{\infty}(\R^d))$ with upper bound depending on $T$, $\Cbdd$ and $\|\nu\|_{W^{1,\infty}}$, then by taking the supremum over $[0,T]$, we have:
\begin{equation*} 
\|\psi\|_{L^{\infty}([0,T], W^{1,\infty}(\R^d))} \leq C(T,\|\varphi\|_{W^{1,\infty}}, \|f_0\|_{L^1}, \Cbdd, \|\nu\|_{W^{1,\infty}} )\,,
\end{equation*}
concluding the proposition.
\end{proof}


\begin{proposition} \label{prop:D2}
Under the assumptions of  Theorem~\ref{thm:gradient_err_estimate}, there exists some positive constant $C_2=C_2(\|\varphi\|_{W^1}, T, \Cbdd, \|\nu\|_{W^{2,\infty}})$ such that
    \begin{equation}
    \E\,|D_2|^2 = \E\left|\int \int_0^T \int f_N(t,x-r) (f-f_N)(t,x) h(t,x) \varphi(r)\,\d x\d t \d r \right|^2
    \leq \frac{C_2}{N}\,.
    \end{equation}
\end{proposition}
\begin{proof}
By plugging in $f_N(t,x)$ to $D_2$, and the change of variable that $y:=x-r$, we have
\begin{equation*}
\begin{aligned}
D_2 
= & \frac{1}{N}\sum_{i=1}^N \int \int_0^T \int \delta(y-X_i(t)) (f-f_N)(t,x) h(t,x) \varphi(x-y)\,\d y\d x\d t\,, \\
= & \frac{1}{N}\sum_{i=1}^N \int_0^T \int (f-f_N)(t,x) h(t,x) \varphi(x-X_i(t))\,\d x \d t\,. 
\end{aligned}
\end{equation*}
Define $\psi^{i}(t,x)=h(t,x) \varphi(x-X_i(t))$. To deploy Theorem~\ref{thm:mean_field_fN_f}, we need $\psi^i\in C^1_b$ for all $i$. By Lemma~\ref{thm:semigroup_At_h_xi_sol}, we already have $\psi^i(t,\cdot)\in C_b(\R^d)$. We are to show $\nabla\psi^i$ is bounded. Since $\varphi\in C_b^1$, it becomes showing $u(t,x):=\nabla_x h\in\R^{d\times d}$ is bounded entry-wise, then from~\eqref{eq:nonlinear_adj_dxg_h}, we deduce:
\begin{equation}\label{eq:dxh_u_eq}
    \partial_t u + (w\ast f)\cdot \nabla_x u
    + L u - \nabla^2 w \ast (hf) + (\nabla^2 w\ast f) \cdot h = 0\,,
    \ \text{with}\ u(T,x) = \nabla_x^2\nu(x)\,,
\end{equation}
where
\begin{equation*}(L u)_{jk} = \sum_{i=1}^d (\partial_{x_k} w_i\ast f)u_{ij} 
+ \sum_{i=1}^d (\partial_{x_j} w_i\ast f) u_{ik}\,.
\end{equation*}

Let $X(t)$ be the characteristics driven by the velocity field $w\ast f$, and following this characteristic, we have:
\begin{equation}
\begin{aligned}
    \frac{\d}{\d t} u(t,X(t)) + L(t,X(t)) u(t,X(t)) - \nabla^2 w \ast (hf)(t,X(t)) + (\nabla^2 w\ast f)\cdot h(t,X(t)) 
    = 0\,.
\end{aligned}
\end{equation}
The solution is as follows: 
\begin{equation}
\begin{aligned}
    u(t,X(t)) = & e^{\int_t^T L(\tau,X(\tau))\,\d \tau} u(T,X(T)) \\
    & - \int_t^T e^{-\int_s^t L(\tau,X(\tau))\,\d \tau} \left[\nabla^2 w \ast (hf)(s,X(s)) - (\nabla^2 w\ast f)h(s,X(s))\right] \,\d s\,.
\end{aligned}
\end{equation}
Therefore, for any $t\in[0,T]$,
\begin{equation*}
\begin{aligned}
    |u(t,X(t))| \leq & e^{\int_t^T L(\tau,X(\tau)) \,\d \tau} u(T, X(T)) \\
    & + \left|\int_t^T e^{-\int_s^t L(\tau,X(\tau))\,\d \tau} \left[\nabla^2 w \ast (hf)(s, X(s)) - (\nabla^2 w\ast f)h(s, X(s))\right] \,\d s\right|\,.
\end{aligned}
\end{equation*}
By the assumption that $w\in C_b^2$ (see~\eqref{eq:assump_Cb2}) and $f(t,\cdot)\in L^1(\R^d)$ for all time, the  factor $e^{\int_t^T L(\tau, X(\tau))\,\d \tau}$ is bounded. Thus, by using Lemma~\ref{thm:semigroup_At_h_xi_sol}, the fact that $f\in L^{\infty}([0,T],L^1(\R^d))$ and Assumption~\ref{assump:property_w(r)}, one has:
\begin{equation}\label{eq:u_on_xi_bdd}
\begin{aligned}    
    |u(t, X(t))| 
    \lesssim & C\left(|u(T, X(T))| \right.\\
    & + \left.\int_t^T \left|\nabla^2 w \ast (hf)(s, X(s))\right| \,\d s
    + \int_t^T \left|(\nabla^2 w\ast f)h(s, X(s))\right| \,\d s\right)\,,\\
    \leq & C\left(|\nabla_x^2\nu( X(T))| 
    + \sup_t \left(\|h(t,\cdot)\|_{L_x^{\infty}} \| f(t,\cdot)\|_{L_x^1}\right)\right) \,,\\
    \leq & C(\Cbdd, T, \|\nu\|_{W^{2,\infty}}, \|\varphi\|_{W^1})\,.
\end{aligned}
\end{equation}
With boundedness in $u$, we have $\psi^i\in C_b^1$. Deploying~\eqref{eq:phi_weak_fN_f_err}, we conclude the proof.
\end{proof}

\begin{proposition} \label{prop:D3}
Under the assumptions of  Theorem~\ref{thm:gradient_err_estimate}, there exists some positive constant $C_3=C_3(\|\varphi\|_{C^1}, T,\|h\|_{W^{1,\infty}}, \|h_N\|_{W^{1,\infty}}, \Cbdd)$ such that
\begin{equation}
\E\,|D_3|^2
= \E\,\left|\int \int_0^T \int f_N(t,x-r) f_N(t,x) (h-h_N)(t,x) \varphi(r)\,\d x\d t \d r \right|^2
\leq \frac{C_3}{N}\,.
\end{equation}
\end{proposition}


\begin{proof}
Plugging in $f_N$ to the expression of $D_3$ defined  in~\eqref{eq:weak_sum_D123} yields
\begin{equation}\label{eq:D3}
\begin{aligned}
D_3
= & \int_0^T \frac{1}{N^2} \sum_{i,j=1}^N (h-h_N)(t,X_i(t)) \varphi(X_i(t)-X_j(t))\,\d t\,.
\end{aligned}
\end{equation}
Then
\begin{equation}\label{eq:ED3_sum_deltah}
\begin{aligned}
|D_3|^2
\leq & \int_0^T \Big| \frac{1}{N} \sum_{i=1}^N (h-h_N)(t,X_i(t)) \frac{1}{N} \sum_{j=1}^N\varphi(X_i(t)-X_j(t)) \Big|^2\,\d t\,,\\
\leq & C(\|\varphi\|_{L^{\infty}}) \int_0^T \Big|\frac{1}{N}\sum_{i=1}^N |h-h_N|(t, X_i(t)) \Big|^2\,\d t\,,\\
\E\,|D_3|^2 \leq & C(\|\varphi\|_{L^{\infty}}) \int_0^T \E \Big|\frac{1}{N}\sum_{i=1}^N |h-h_N|(t, X_i(t)) \Big|^2\,\d t\,.
\end{aligned}
\end{equation}

To estimate $h-h_N$, we subtract~\eqref{eq:nonlinear_adj_dxg_h} and~\eqref{eq:nonlinear_fN_induce_hN} to deduce the equation for $h-h_N$. With the calculation in Lemma~\ref{lem:deltah_sol}, 
we obtain 
the following estimate 
\begin{equation}\label{eq:mean_delta_h_Gronwall}
\begin{aligned}
\frac{1}{N}\sum_{i=1}^N |h-h_N|(t, X_i(t)) 
\leq &  C(\Cbdd,T) \int_0^t \frac{1}{N} \sum_{i=1}^N |S(s,X_i(s))|\d s\,,
\end{aligned}
\end{equation}
where $S(t,x)$ denotes the source term
\begin{equation}\label{eq:H_dh_source_def}
S(t,x) = \underbrace{-(w\ast (f-f_N)) \nabla_x h
}_{S_1}
+ \underbrace{\nabla w\ast (h_N (f-f_N))}_{S_2}   
+ \underbrace{(- \nabla w \ast (f-f_N)) h_N}_{S_3}\,.
\end{equation}

Upon taking the expectation of~\eqref{eq:mean_delta_h_Gronwall}, one has:
\begin{eqnarray}
\E\left|\frac{1}{N}\sum_{i=1}^N |h-h_N|(t, X_i(t)) \right|^2  \nonumber 
& \leq &
C(\Cbdd,T)\, \E\, \left|\int_0^t \frac{1}{N} \sum_{i=1}^N |S(s,X_i(s))|\d s \right|^2\,, \nonumber  \\
& \leq & C(\Cbdd,T) \int_0^t \E\,\left|\frac{1}{N} \sum_{i=1}^N |S(s,X_i(s))|\right|^2 \,\d s \,,  \nonumber \\
& \leq & C(\Cbdd,T) \int_0^t \frac{1}{N} \sum_{i=1}^N \E\, |S(s,X_i(s))|^2 \d s \,, \nonumber \\
& \leq &  C(\Cbdd,T) \sup_{t,x} \E\, |S(t,x)|^2\,,\label{eq:E_sum_deltah_source}
\end{eqnarray}
where we apply Jensen's inequality.

Then we are left to show that $\E\,|S(t,x)|^2\leq \frac{C}{N}$ for some $C$. According to~\eqref{eq:H_dh_source_def}, we are to examine each term and apply  estimate~\eqref{eq:phi_weak_fN_f_err} in Theorem~\ref{thm:mean_field_fN_f}  repeatedly. For $S_1$:
\begin{equation}\label{eq:EH1_sqrtN}
\begin{aligned}
\E\, |S_1(t,x)|^2 = & \E\,|(w\ast (f-f_N))(t,x) \nabla_x h(t,x)|^2\,, \\
\leq & \sup_t \|\nabla_x h(t,\cdot)\|^2_{L^{\infty}(\R^d)}\, \E\,|(w\ast (f-f_N))(t,x)|^2\,,  \\
\leq & C(\|h\|_{L^{\infty}([0,T],W^{1,\infty}(\R^d))}, \Cbdd) \frac{1}{N}\,.
\end{aligned}
\end{equation}
For $S_2$: by Lemma~\ref{thm:semigroup_At_h_xi_sol}, $h_N\in L^{\infty}([0,T],L^{\infty}(\R^d))$:
\begin{equation}\label{eq:EH2_sqrtN}
\begin{aligned}
\E\, |S_2(t,x)|^2 = & \E\,|\nabla w\ast (h_N (f-f_N))(t,x)|^2\,,\\
\leq & C(\|h_N\|_{L^{\infty}([0,T],L^{\infty}(\R^d))})\E\,|\nabla w\ast (f-f_N)(t,x)|^2\,,\\
\leq & C(T, \|h_N\|_{L^{\infty}}, \CLip, \Cbdd) \frac{1}{N}\,.
\end{aligned}
\end{equation} 
Similarly for $S_3$:
\begin{equation}\label{eq:EH3_sqrtN}
\begin{aligned}
\E\, |S_3(t,x)|^2 = & \E\,|(\nabla w\ast (f-f_N))(t,x) h_N(t,x)|^2
\leq C(T, \|h_N\|_{L^{\infty}}, \CLip, \Cbdd) \frac{1}{N}\,.
\end{aligned}
\end{equation} 
Plugging in the above estimates into~\eqref{eq:E_sum_deltah_source} and ~\eqref{eq:ED3_sum_deltah} concludes the proof.

\end{proof}

\section{Numerical results}\label{sec:numerics}
Numerical evidences confirm our theoretical findings, and we present them here. In particular:
\begin{itemize}
    \item For single-particle transport system, we numerically verify the consistency between particle-based gradient and transport equation-based gradient (Theorem~\ref{thm:gradient_consistency}).
    \item For interacting particle system, we numerically present the error between $\frac{\delta \mathsf{J}}{\delta w}$ and $\frac{\delta \mathcal{J}}{\delta w}$ and demonstrate its Monte Carlo (MC) rate of $\mathcal{O}(N^{-1/2})$ (Theorem~\ref{thm:gradient_err_estimate}).
    \item As an exploratory effort, we also run gradient-based optimization solvers for the reconstruction of interaction kernels.
\end{itemize}
These three components are respectively discussed in the following three subsections.

\subsection{First variation for the transport equation.}\label{subsec:numerics_grad_ax}
We dedicate this subsection to numerically demonstrate Theorem~\ref{thm:gradient_consistency}.

As a simple example, we set $x\in \R$, a constant velocity $a(x)\equiv 1$, with a uniform initial distribution, that is, $f_{0}(x)=\mathbf{1}_{[-0.5, 0.5]}(x)$ and $X_0\sim \mathrm{Unif}[-0.5, 0.5]$. The simulation is run up to the final time $T=0.5$, when we measure the solution using the test function of 
$$
\nu(x):=\frac{1}{\sqrt{0.1\pi}}\exp\left(- 10 x^2\right).
$$

In this simple setup, explicit solutions are available:
\begin{equation*}
\begin{aligned} 
& f(t,x) = f_0(x-t)\,, & g(t,x) = \nu(x+(T-t))\,,\\
& X(t) = X_0 + t\,, & Y(t) = Y(T) = \nabla_x\nu(X(T))\,,
\end{aligned}
\end{equation*}
to be assembled in $\frac{\delta \mathcal{J}}{\delta a}$ and $\frac{\delta \mathsf{J}}{\delta a}$ (see~\eqref{eq:dJ_da_transport_fg} and~\eqref{eq:dJ_da_transport_Y} for their definitions). The theorem states $\frac{\delta \mathcal{J}}{\delta a}(x;f_0)=\mathbb{E}_{y\sim f_0}\left[\frac{\delta \mathsf{J}}{\delta a}(x;y)\right]$. For this example, 
\begin{itemize}
    \item $\frac{\delta \mathcal{J}}{\delta a}(x;f_0)=\int_0^T \frac{-20(x+(T-t))}{\sqrt{0.1\pi}}\exp\left(- 10 (x+(T-t))^2\right) \mathbf{1}_{[-0.5,0.5]}(x-t)\,\d t$.
    \item To simulate $\mathbb{E}_{y\sim f_0}\left[\frac{\delta \mathsf{J}}{\delta a}(x;y)\right]$, we draw $\{y_i\}_{i=1}^N\sim f_0$ with $N\gg 1$ to compute an empirical mean, and numerically approximate the Dirac delta in~\eqref{eq:dJ_da_transport_Y} by a smooth mollifier $\phi_{\eps}(r):=\frac{1}{\eps}\phi\left(\frac{r}{\eps}\right), \eps\ll 1$. We set $\eps=0.02$. We denote the solution $\frac{\delta\mathsf{J_N}}{\delta a}$.
\end{itemize}


In Figure~\ref{fig:compute_grad_const_a1}(A), we plot $\frac{\delta \mathcal{J}}{\delta a}$ as a function of $x$, and the empirical mean approximation to $\frac{\delta\mathsf{J_N}}{\delta a}$ using $N=10^4$ particles. They are in good agreement visually. To quantify $\|\frac{\delta\mathsf{J_N}}{\delta a} - \mathbb{E}\frac{\delta\mathsf{J}}{\delta a}\|_\infty$,
we show in Figure~\ref{fig:compute_grad_const_a1}(B) the empirical mean averaged over $20$ independent runs (circles), for an increasing number of particles. 

\begin{figure}[htbp]
    \centering
    \begin{subfigure}[b]{0.42\textwidth}
         \centering
         \includegraphics[width=\textwidth]{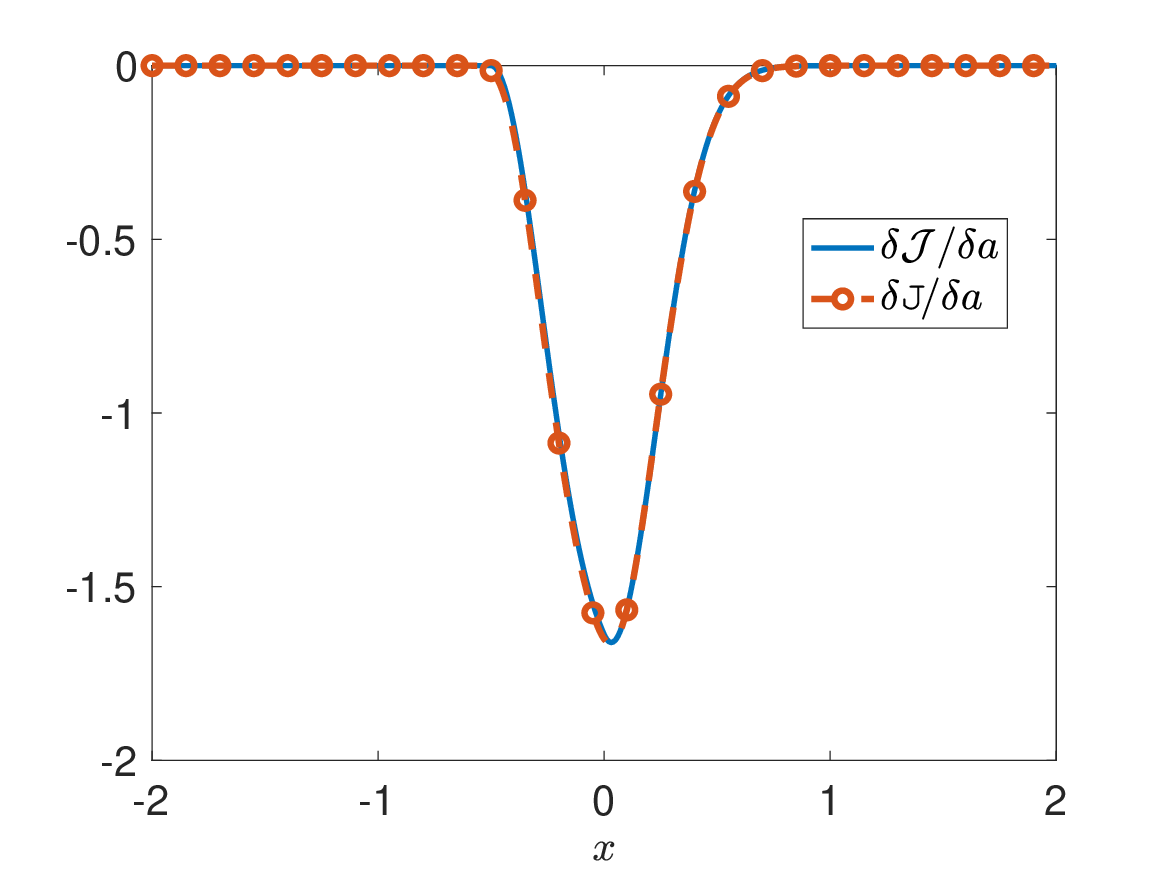}
    \caption{First variation computed using $N=10^4$ particles (red circle), and the solution of the transport equation (blue solid line).}
    \label{fig:2gradients_a1}
     \end{subfigure}
     \hfill
     \begin{subfigure}[b]{0.42\textwidth}
         \centering
         \includegraphics[width=\textwidth]{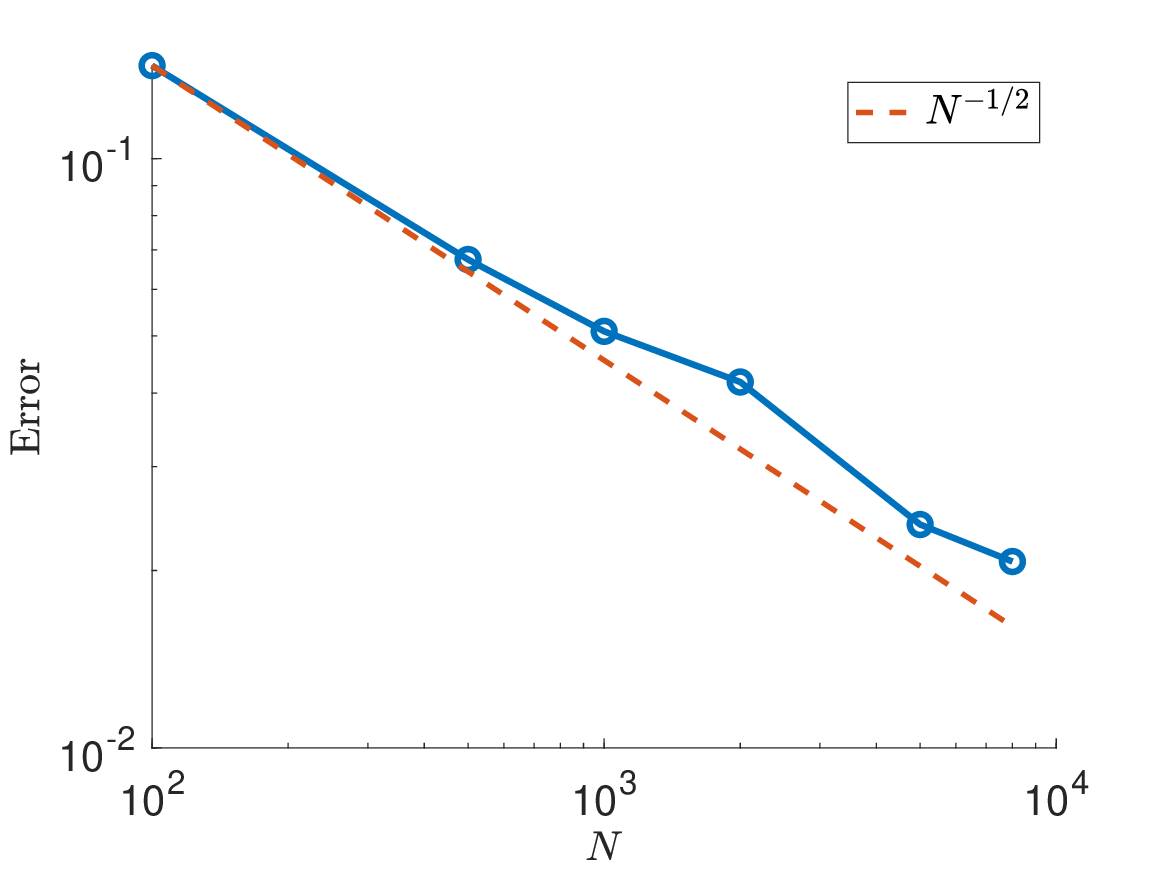}
    \caption{$l^\infty$-error. Blue line shows the error decay as $N$ changes from $100$ to $8000$ and red dashed line is the MC rate of $\mathcal{O}(N^{-1/2})$.}
    \label{fig:err_2grad_a1}
    \end{subfigure}
    \caption{First variation of the objective function with respect to the velocity field $a$.}
    \label{fig:compute_grad_const_a1}
\end{figure}

\subsection{First variation for the interaction kernel.} \label{subsec:numerics_grad_wr}
In this subsection, we provide numerical evidence for Theorem~\ref{thm:gradient_err_estimate} by comparing the two first variations computed by the PDE and the particle system.

We set $x\in \R$, $T=1$, and the interaction kernel and the measurement function are
\begin{equation}\label{eq:kernel_wr_gauss_measure}
w(r)=\frac{5}{\sqrt{2\pi}} r e^{-\frac{r^2}{2}}\,,  \quad
\nu(x) = \frac{1}{\sqrt{2\pi}} e^{-\frac{x^2}{2}}\,,
\end{equation}
which appear in~\eqref{eq:Jf_wr_continuous_opt} and~\eqref{eq:JX_wr_particle_opt}. Initially, the system follows the mollified uniform distribution on $[-0.5, 0.5]$, 
i.e., with $\phi_{\eps}(x) = \frac{1}{\eps}\phi(\frac{x}{\eps}), \eps \ll 1$, being a Gaussian mollifier, 
\begin{equation}\label{eq:fin_gauss_001}
f_0(x):=\phi_{\eps}\ast \mathbf{1}_{[-0.5, 0.5]}\,.
\end{equation}
The PDE~\eqref{eq:df_wr_mean_field} is then simulated by a forward Euler scheme together with an upwind method to compute the numerical flux. The domain is truncated at some large $L$ and the time step is chosen so that CFL condition is satisfied. To compute the first variation, we zero-extend $f$ to the whole space and compute $\nabla_x g$ numerically by central difference.

To simulate the particle system, we first draw initial samples from $f_0$ using rejection sampling. Forward Euler with time discretization $\Delta t = 0.01$ is used to simulate $\{X_{i}(t)\}$,  
and for compatibility, backward Euler is used to simulate the adjoint variable $\{Y_i(t)\}$ with the same time stepping. To assemble the first variation~\eqref{eq:nonlinear_gradient_NdYi}, numerically we approximate the Dirac delta by a smooth mollifier $\phi_{\eps}(r):=\frac{1}{\eps}\phi\left(\frac{r}{\eps}\right), \eps\ll 1$ and compute:
\begin{equation} \label{eq:particle_grad_approx}
\begin{aligned}
    \dfrac{\delta \mathsf{J}}{\delta w}(r) &  
    \approx \frac{1}{N^2}\sum_{i=1}^N \sum_{j\not=i} \int_0^T Y_i(t)\phi_{\eps}(r-(X_i-X_j)(t))\,\d t \,.
\end{aligned}
\end{equation}
with $\varepsilon =0.02$.

For six different choices of $N$ ranging from $20$ to $400$, the first variation~\eqref{eq:particle_grad_approx} is computed and averaged over $20$ independent runs. In Figure~\ref{fig:wr_gradients}(A), we plot the reference first variation $\frac{\delta\mathcal{J}}{\delta w}$ and the first variation $\frac{\delta\mathsf{J}}{\delta w}$ computed using particle systems using $N=20$ (red circles) and $N=400$ (blue plus signs) particles. It is evident that that $20$-particle-system is not able to capture some fine feature of the true gradient, such as some stiff derivative near $r=0$. In Figure~\ref{fig:wr_gradients}(B), the relative $l^\infty$-error with respect to the particle number is shown in the solid line in the $\log$-$\log$ scale. We numerically observe a convergence rate slightly larger than the predicted MC rate $\mathcal{O}(1/\sqrt{N})$.

\begin{figure}[htbp]
    \centering
    \begin{subfigure}[b]{0.42\textwidth}
         \centering
         \includegraphics[width=\textwidth]{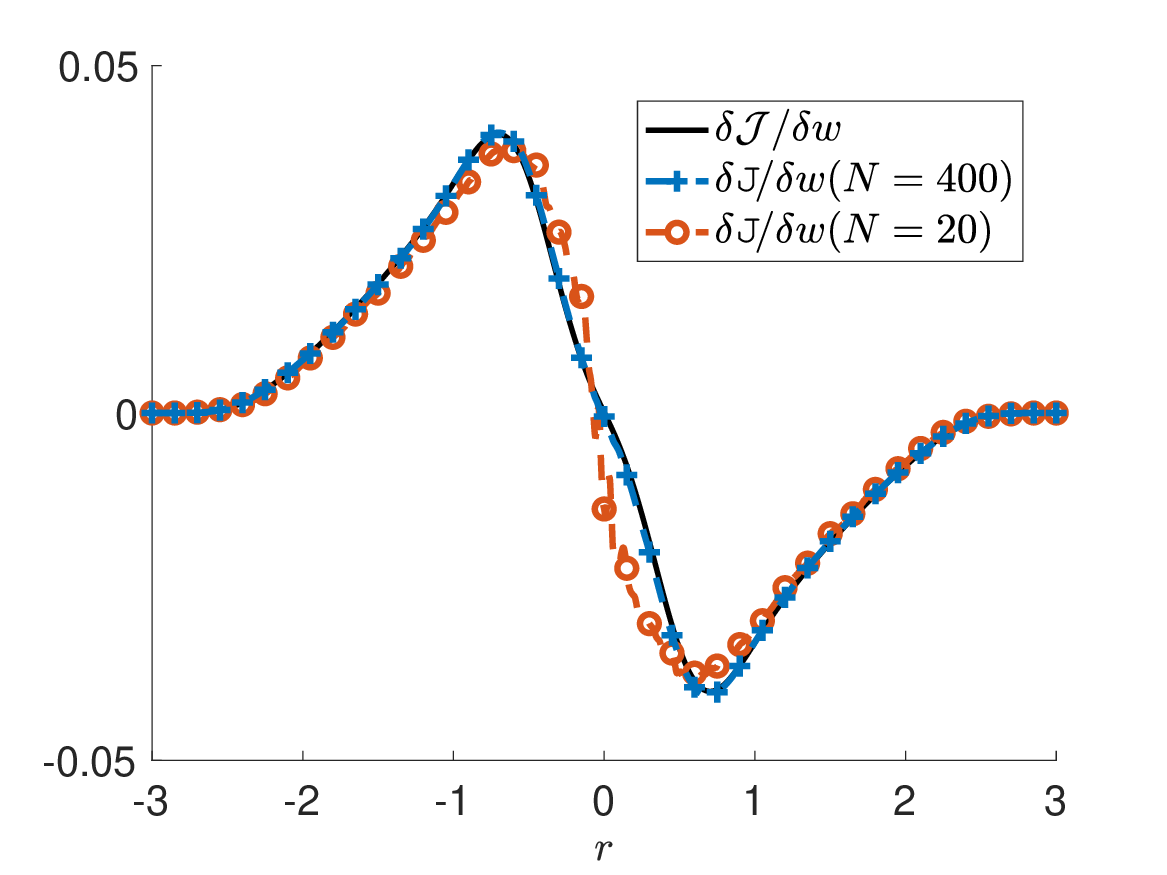}
    \caption{First variation computed using $N=20$ (red circle), $N= 400$ (blue plus) particles. The black solid line is the first variation computed by the corresponding PDE.}
    \label{fig:2grad_wr_N20}
    \end{subfigure}
    \hfill
    \begin{subfigure}[b]{0.42\textwidth}
         \centering
         \includegraphics[width=\textwidth]{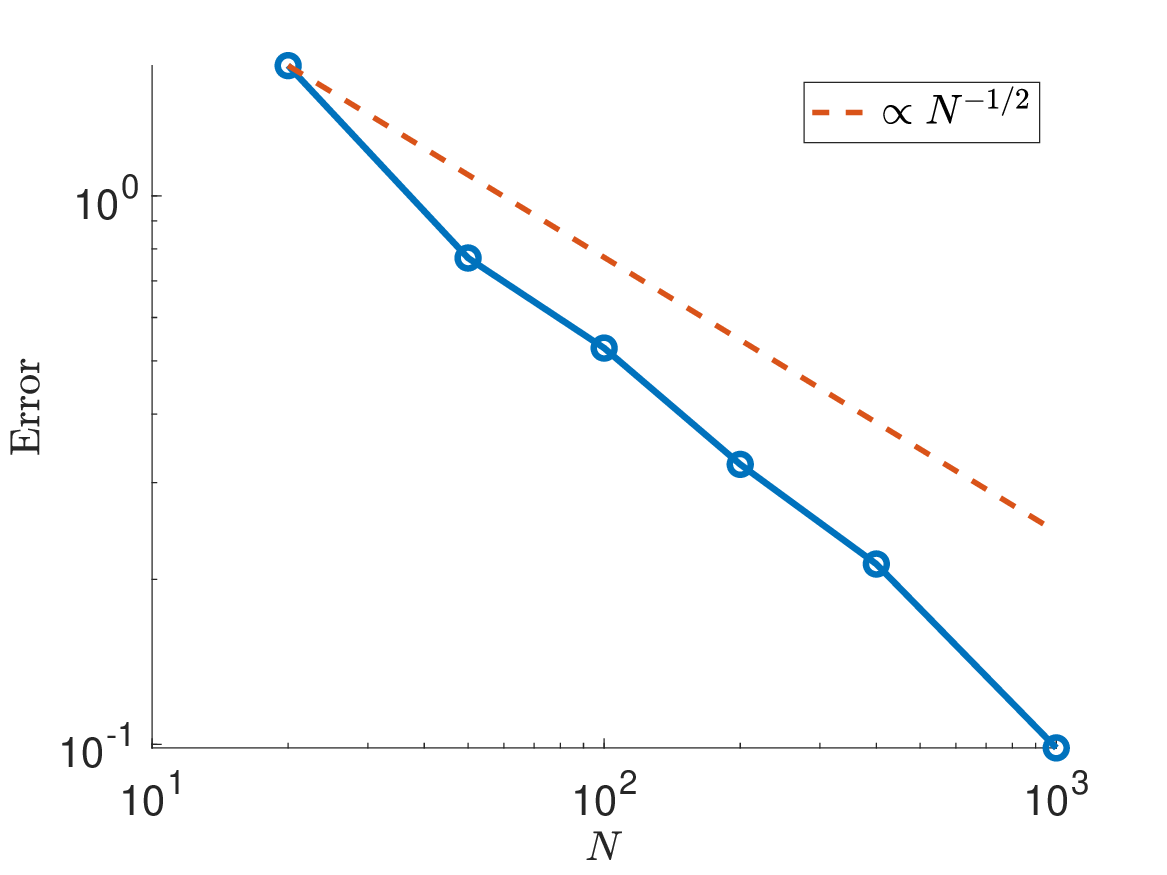}
    \caption{Relative $l^\infty$-error in terms of the number of particles $N$ from $20$ to $1000$, and the dashed red line showing the theoretical MC rate $\mathcal{O}(N^{-1/2})$.}
    \label{fig:2grad_wr_err}
    \end{subfigure}
    \caption{First variation of the objective function with respect to the interaction kernel $w$.}
    \label{fig:wr_gradients}
\end{figure}

\subsection{Reconstruction of the interaction kernel}\label{subsec:inverse_problem}
This section is exploratory: we apply the computation of the first variation to run the optimization for the reconstruction of the interacting kernel. The data is produced by the underlying density $f$, while the inversion is produced on the particle level through solving the ODE-constrained optimization~\eqref{eq:JX_wr_particle_opt}.

To setup the problem, we first generate synthetic data. Let the ground truth interaction kernel be $w^\ast(r)$. We solve for $f^\ast$ by simulating~\eqref{eq:df_wr_mean_field}. The measurement is thus $d^\ast = \int_{\R^d} \nu(x) f^\ast(T,x)\,\d x$, and we have the objective function $\mathsf{J}$:
\begin{equation}\label{eq:particle_inverse_J}
\mathsf{J}[\{X_i\}] = \frac{1}{2}\left(\frac{1}{N}\sum_{i=1}^N \nu(X_i(T)) - d^\ast\right)^2\,.
\end{equation}

This objective function is not immediately in the form that agrees with~\eqref{eq:JX_wr_particle_opt}, and accordingly we adjust the terminal time condition of the adjoint state $\{Y_i\}$:
\begin{equation}\label{eq:inverse_YT_BC}
Y_i(T) = \nabla_x\nu(X_i(T))\left(\frac{1}{N}\sum_{i=1}^N \nu(X_i(T)) - \int \nu(x) f^\ast(T,x)\,\d x\right)\,.
\end{equation}

Throughout this section, we assume that the interaction kernel $w(r)$ can be parameterized by:
\begin{equation}\label{eq:wr_param}
w(r):=\sum_{l=1}^L a_l b_l(r)\,,
\end{equation} 
with $\{b_l(r)\}_{l=1}^L$ being a set of fixed odd basis functions and $a_l$ being the coefficients. This translates the reconstruction of $w$ into that of $\{a_l\}$. We assume the ground truth kernel also shares the same parameterization form $w^\ast(r)=\sum_{l=1}^L a^\ast_l b_l(r)$.

For optimization, we run the classical Gradient Descent with backtracking line search, as summarized in Algorithm~\ref{alg:GD}. It should be noted that the algorithm updates the parameters $\{a_l\}_{l=1}^L$ with:
\begin{equation}\label{eq:opt_param_update}
a_l^{(k+1)} = a_l^{(k)} - \tau_l^{(k)} \partial_{a_l}\mathsf{J}\,,
\end{equation}
where the derivative $\partial_{a_l}\mathsf{J}$ is the first variation~\eqref{eq:nonlinear_gradient_NdYi} projected to the basis of $b_l$:
\begin{equation}\label{eq:opt_param_gradient}
\partial_{a_l}\mathsf{J}=\left\l \frac{\delta \mathsf{J}}{\delta w}\,, \frac{\partial w}{\partial a_l} \right\r_r 
= \frac{1}{N^2} \sum_{i=1}^N \sum_{j\not=i} \int_0^T Y_i(t)  b_l(X_i(t)-X_j(t))\, \d t\,.
\end{equation}
To quantify the performance of the algorithm, we will also calculate the error of the reconstructed interaction kernel and the ground truth at the $k$-th iteration, given by
\begin{equation}\label{eq:L_infty_error}
    E^{(k)}:=\|w^{(k)}(r) - w^\ast(r)\|_{L^{\infty}}/\|w^\ast(r)\|_{L^{\infty}}\,.
\end{equation}
\begin{algorithm}
\caption{Gradient Descent for Interaction Kernel Reconstruction.}
\label{alg:GD}
\begin{algorithmic}
\Require Error tolerance $\delta$, initial condition $f_0$, particle number $N$.
\State Draw initial particle positions independently $X_i^{in} \sim f_0$, for $i=1,\cdots,N$.
\While{Error $\mathsf{J}^{(k)}[w]~\eqref{eq:particle_inverse_J} \geq \delta$}
    \State Run forward~\eqref{eq:dX_wr_interact} and adjoint~\eqref{eq:nonlinear_adjoint_NdYi} particle simulations to obtain $\{X_i(t), Y_i(t)\}_{i=1}^N$.
    \State Compute particle gradients through~\eqref{eq:nonlinear_gradient_NdYi}.
    \State Update the parameters according to~\eqref{eq:opt_param_update}, where the step-size $\tau^{(k)}$ is determined by backtracking (Armijo--Goldstein) line search.
    \State Update the predicted kernel $w^{(k+1)}=\sum_{l=1}^L a_l^{(k+1)} b_l(r)$.
    \State Measure the data mismatch $\mathsf{J}^{(k+1)}[w]$ based on~\eqref{eq:particle_inverse_J}.
\EndWhile
\end{algorithmic}
\end{algorithm}

\textbf{Example 1: 1D attractive-repulsive kernel.}
In this example, we set $L=3$ and consider the following basis functions,
\begin{equation}\label{eq:grad_gaussian_basis}
    b_l(r) = c_l r e^{-\frac{|r|}{2}} L^{(2)}_{l-1}(|r|)\,, 
\end{equation}
where $L^{(2)}_l$ is the $l$-th generalized Laguerre polynomial associated with parameter $2$, meaning any two generalized Laguerre polynomials are orthogonal with each other with respect to basis function $x^2e^{-x}$. The constant $c_l$ in~\eqref{eq:grad_gaussian_basis} normalizes the basis, making 
$\int_0^{\infty} b_l(r) b_m(r)\,\d r = \delta_{l,m}$. The ground-truth is set to be $w^\ast(r)=\sum_{l=1}^3 a_l^{\ast} b_l(r)$, with $(a_1^\ast, a_2^\ast, a_3^\ast) = (0.4, 0.5, 0.8)$.

We first examine the property of this interaction kernel in the forward setting. In Figure~\ref{fig:wr_three_basis_forward} we present the ground-truth kernel $w^\ast$ (Figure~\ref{fig:wr_three_basis_forward}(A)). This plot suggests that particles are repulsive to each other when they are close (with distance $r<3$) and attractive when they are far apart (with distance $r>3$). To vividly see this, we plot the dynamics of the density $f(t,x)$ that solves ~\eqref{eq:df_wr_mean_field} in Figure~\ref{fig:wr_three_basis_forward}(B), where the initial condition is set as the superposition of two Gaussian distributions.
These two Gaussians, in time, have widened width, reflecting the repulsion effect for small $r$, while their centers approach each other (dashed red lines), reflecting the attraction effect for big $r$.
\begin{figure}[htbp]
    \centering
    \begin{subfigure}[b]{0.42\textwidth}
         \centering
         \includegraphics[width=\textwidth]{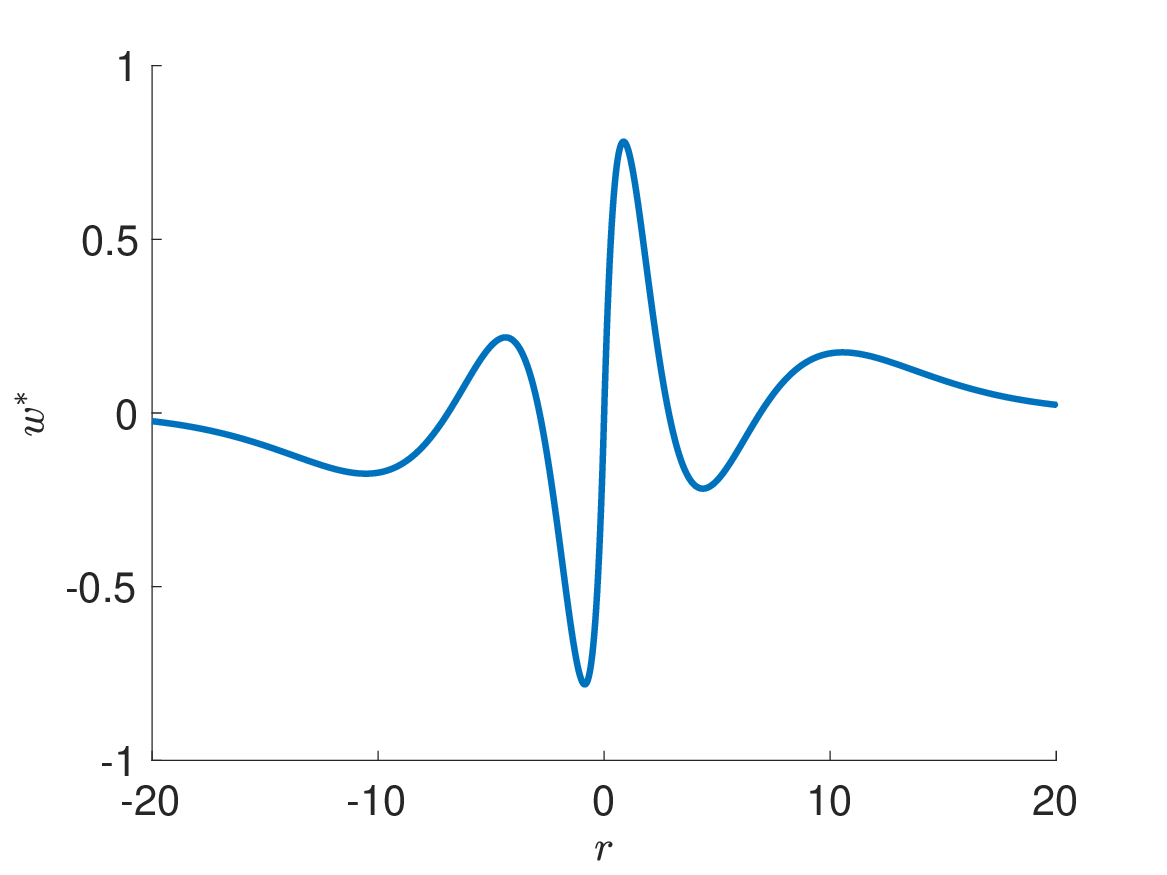}
    \caption{$w^{\ast}(r)$}
    \label{fig:ex2_wr_a123_star}
    \end{subfigure}
    \hfill
    \begin{subfigure}[b]{0.42\textwidth}
         \centering
         \includegraphics[width=\textwidth]{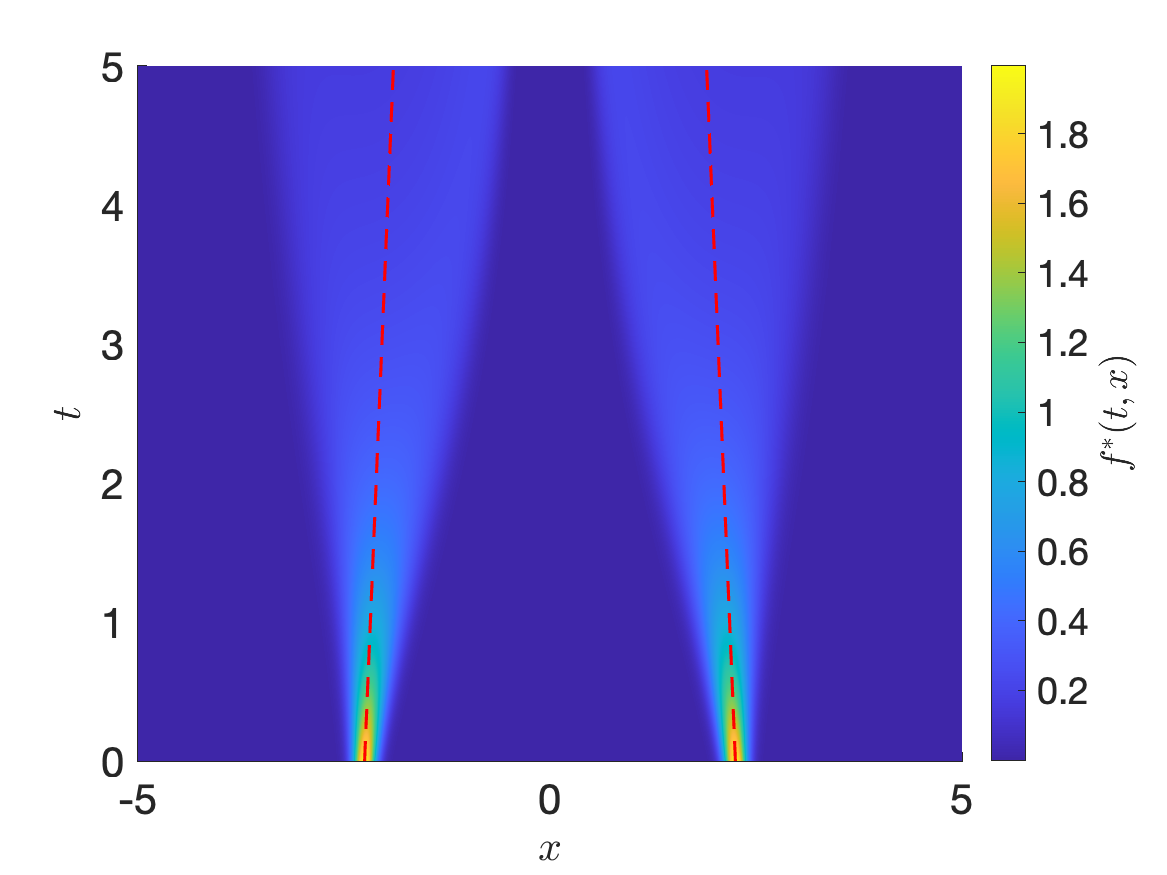}
    \caption{$f(t,x)$}
    \label{fig:ex2_wr_a123_ftx}
    \end{subfigure}
    \caption{Example~1. The two panels show the ground-true kernel $w^*$ and a forward PDE simulation.}
    \label{fig:wr_three_basis_forward}
\end{figure}

To run the inverse problem, we set the measurement function as $\nu(x):=\frac{1}{2}x^2$ so that the objective function captures the second moment. The initial condition of the particle system is the bimodal distribution as follows, 
\begin{equation}\label{eq:X0_bimodal_dist}
    X_{i,0} \overset{\text{i.i.d.}}{\sim} 
    f_0(x) = \frac{1}{2}\frac{1}{\sqrt{0.1\pi}} e^{-\frac{(x+0.5)^2}{0.1}} + \frac{1}{2}\frac{1}{\sqrt{0.1\pi}} e^{-\frac{(x-0.5)^2}{0.1}}\,,
\end{equation}
and is run until $T=0.5$.  We set the initial guess to be $(a_1^{(0)}, a_2^{(0)}, a_3^{(0)}) = (0.2, 0.1, 0.3)$.

The reconstruction results are presented in Figure~\ref{fig:wr_three_basis_reconstruct} and Figure~\ref{fig:wr_three_basis_f_fNhist_Np4000}. Figure~\ref{fig:wr_three_basis_reconstruct}(A) shows the initial guess, and the reconstruction of the interaction kernel at the 4th and final $K=78$ iteration using $N=4000$ particles, against the ground truth. At the last iteration, our reconstruction accurately capture the ground truth. The relative $l^\infty$-error~\eqref{eq:L_infty_error} is evaluated at each iteration and its evolution is plotted in Figure~\ref{fig:wr_three_basis_reconstruct}(B). The error drops by $10$-fold within $100$ iterations, showing fast convergence.

\begin{figure}[htbp]
    \centering
    \begin{subfigure}[b]{0.42\textwidth}
         \centering
         \includegraphics[width=\textwidth]{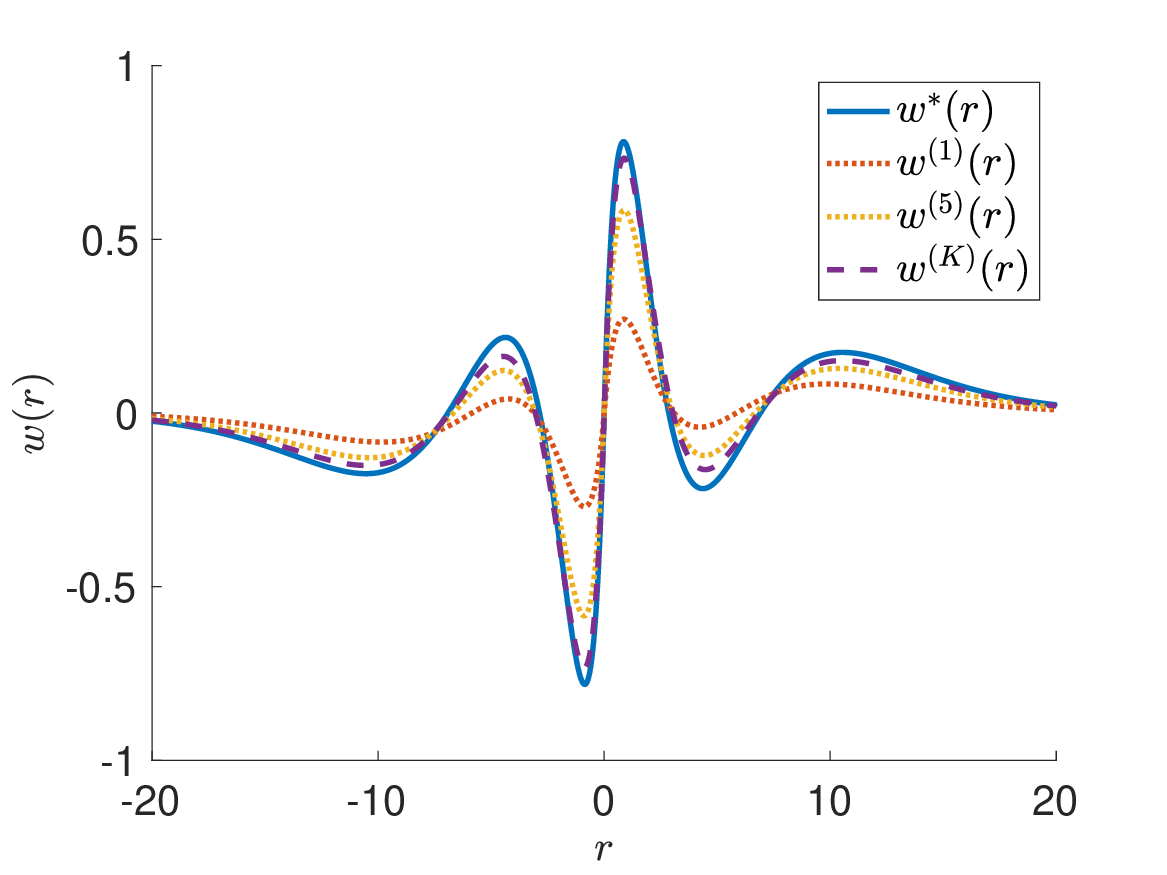}
    \caption{$w^{(k)}(r)$}
    \label{fig:ex2_wr_a123_N4000}
    \end{subfigure}
    \hfill
    \begin{subfigure}[b]{0.42\textwidth}
         \centering
         \includegraphics[width=\textwidth]{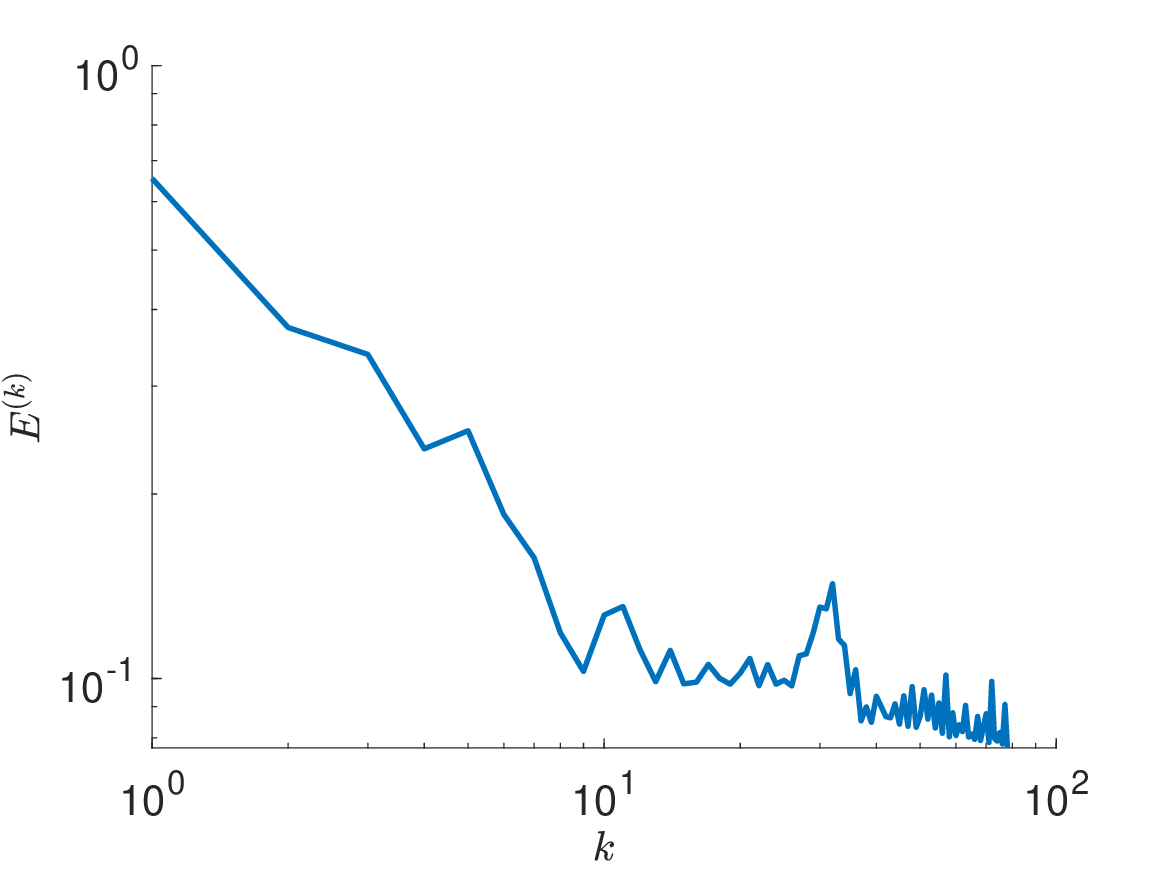}
    \caption{Error, $N=4000$}
    \label{fig:ex2_wr_a123_Err}
    \end{subfigure}
    \caption{Example~1. (A) shows reconstructed $w(r)$ at three different iterations, and (B) shows the decay of error in optimization measured by $E^{(k)}$ following~\eqref{eq:L_infty_error}.}
    \label{fig:wr_three_basis_reconstruct}
\end{figure}

Figure~\ref{fig:wr_three_basis_f_fNhist_Np4000} demonstrates the success of the reconstruction. The 2D surface maps of the ground-truth distribution $f^\ast(t,x)$ (see Figure~\ref{fig:wr_three_basis_f_fNhist_Np4000}(A)), and differences between the ground-truth and the distributions generated by $N$ particles (histogram) at the initial and the final $K=78$-th iteration respectively. That is, the initial error of distributions, $f_N^{(0)}(t,x) - f^\ast(t,x)$ is shown in Figure~\ref{fig:wr_three_basis_f_fNhist_Np4000}(B) and the error at the last iteration, $f_N^{(K)}(t,x) - f^\ast(t,x)$ is shown in Figure~\ref{fig:wr_three_basis_f_fNhist_Np4000}(C).
It is evident that the distribution obtained by the final reconstruction, in comparison to the distribution from the initial guess of $w$, better captures the true dynamics of the density.

\begin{figure}[htbp]
    \centering
    \begin{subfigure}[b]{0.32\textwidth}
         \centering
         \includegraphics[width=\textwidth]{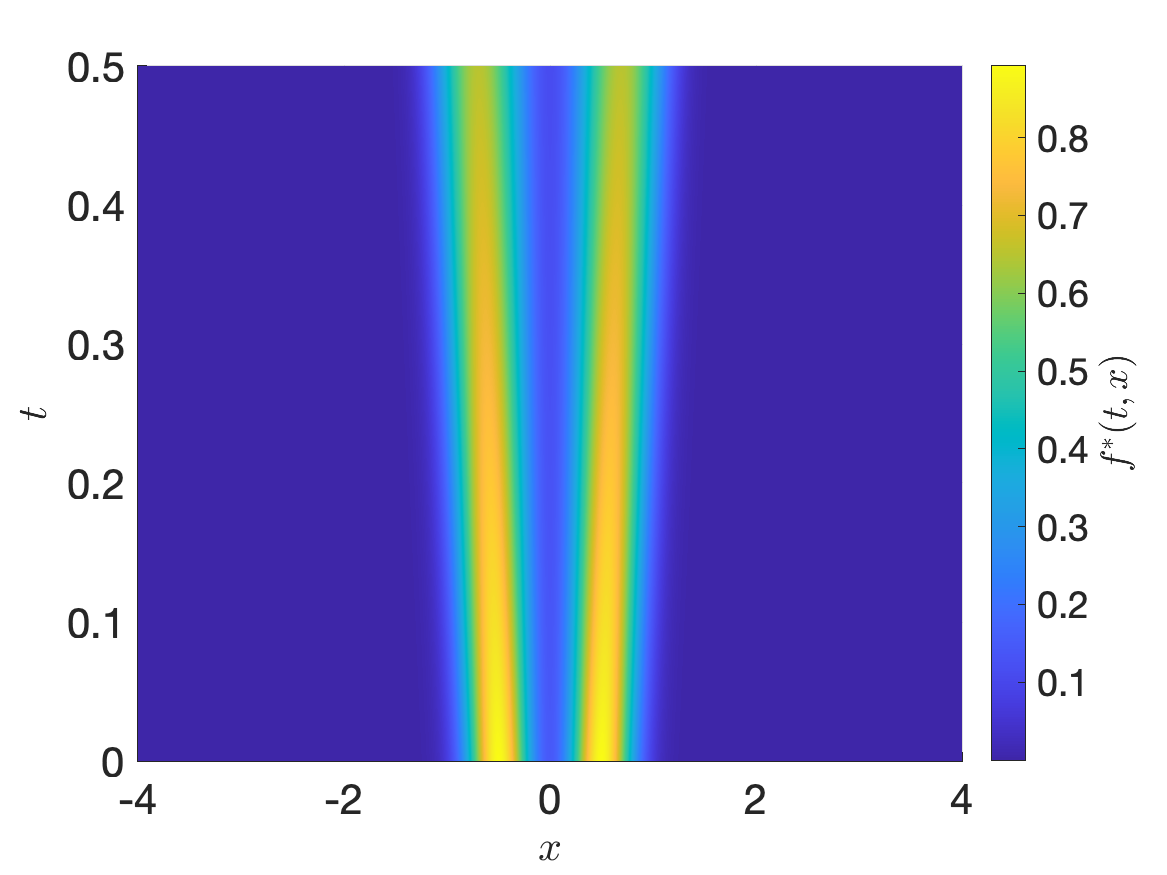}
    \caption{$f^\ast(t,x)$}
    \label{fig:wr_three_basis_kernel_fhist_wstar}
    \end{subfigure}
    \hfill
    \begin{subfigure}[b]{0.32\textwidth}
         \centering
         \includegraphics[width=\textwidth]{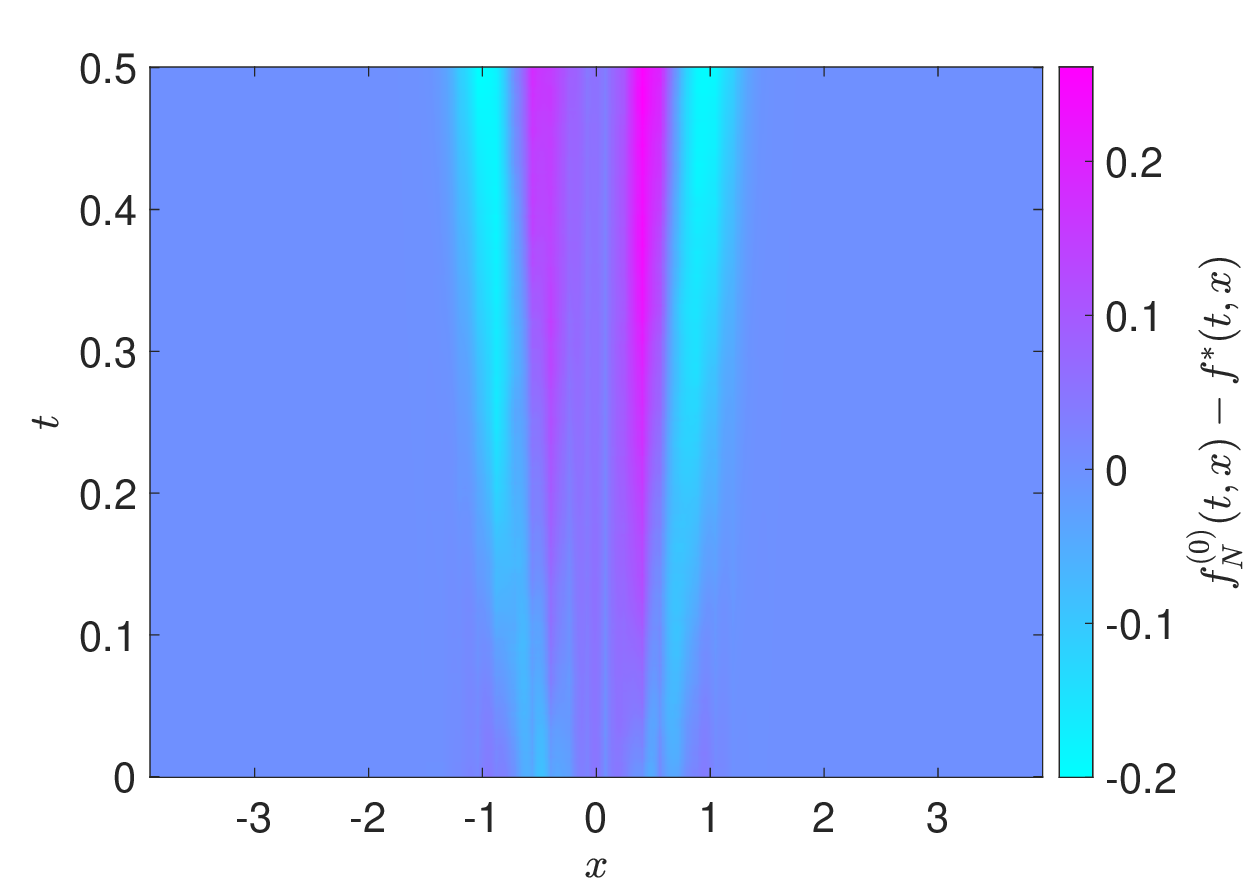}
    \caption{$f_N^{(0)}(t,x)-f^\ast(t,x)$}
    \label{fig:wr_three_basis_kernel_fhist_wK}
    \end{subfigure}
    \hfill
    \begin{subfigure}[b]{0.32\textwidth}
         \centering
         \includegraphics[width=\textwidth]{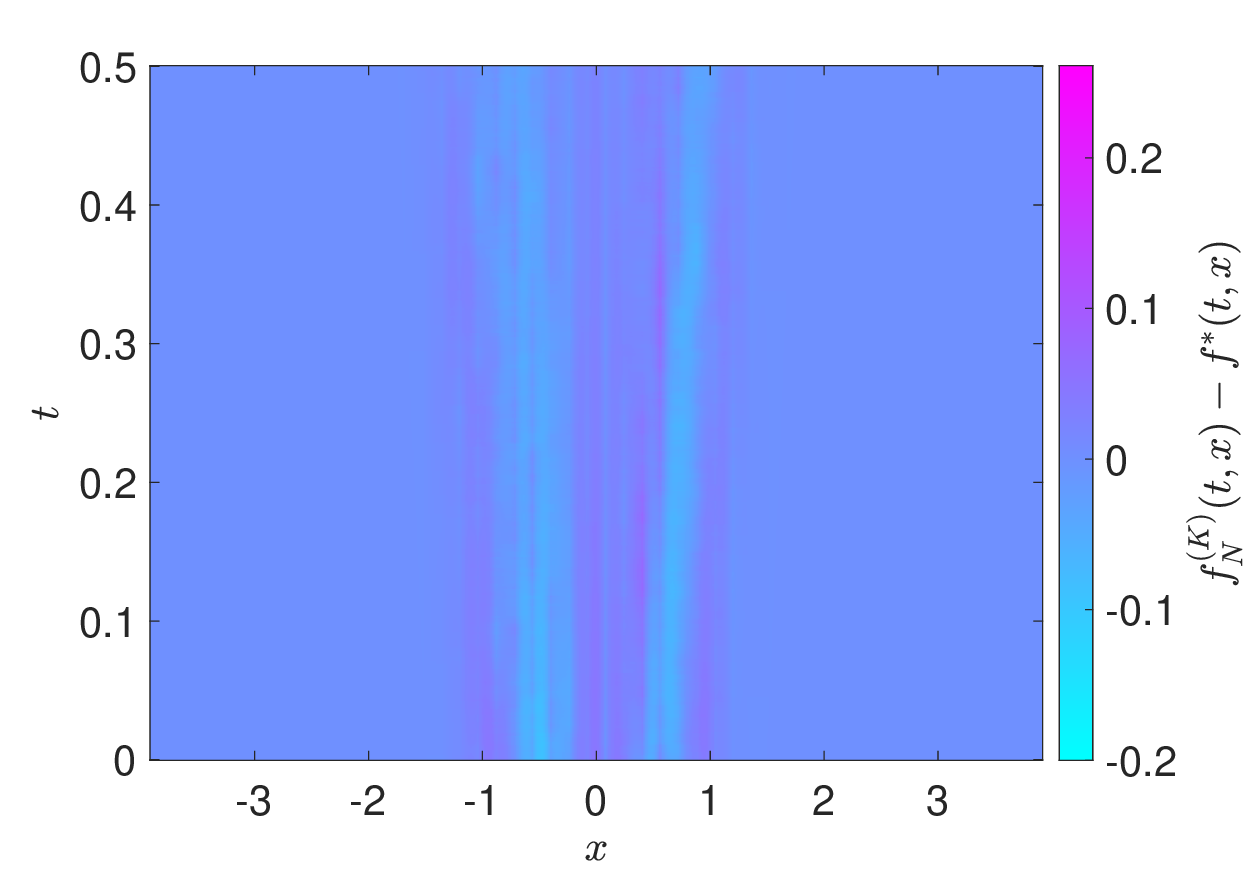}
    \caption{$f_N^{(K)}(t,x)-f^\ast(t,x)$}
    \label{fig:wr_three_basis_kernel_fNhist_wK}
    \end{subfigure}
    \caption{Example~1: The three panels present the ground-true distribution evolution in time, error between the guessed and the true distributions at the initial iteration, and error between the reconstructed and the true distributions at $K=0$ (initial) and $K=78$-th iteration using a unified colormap.}
    \label{fig:wr_three_basis_f_fNhist_Np4000}
\end{figure}

\medskip
\textbf{Example 2: 2D attractive-repulsive interaction.}
The second example regards an inverse problem posed over 2D: $w(r): r\in \R^2 \to \R^2$. The following basis functions are used
\begin{equation*}
b_i(r) = (-1)^i \frac{1}{2\pi\theta_i^2} r e^{-\frac{|r|^2}{2\theta_i^2}} \,\quad\text{with}\quad \theta_1 = 0.25\,,\theta_2 = 1\,,
\end{equation*}
where the sign of the basis represent attractive or repulsive interaction respectively. We set the ground-truth parameter to be $(a_1^\ast, a_2^\ast)=(1.5,0.8)$ so that we have a strong near-field attraction and a weak far-field repulsion. Note that each basis is a vector function of two components. We plot the ground-true kernel, the interaction potential $\phi^*(r)=(-1)^{i-1} \frac{1}{2\pi} e^{-\frac{|r|^2}{2\theta_i^2}}$  where $w^\ast(r)=:\nabla_r \phi^\ast(r)$ as the potential function, and the dynamics of the distribution function at different times in Figure~\ref{fig:wstar_2d_forward}, where we used the initial data
\begin{equation*}
f_0=\frac{1}{2}\frac{1}{2\pi\theta}e^{-\frac{(x_2-0.425)^2}{2\theta}}\left(e^{-\frac{(x_1-0.5)^2}{2\theta}} + e^{-\frac{(x_1+0.5)^2}{2\theta}}\right)\,, \quad \theta = 0.05\,, 
\end{equation*}
and the final time is set $T=4$.
It can be clearly observed that the two Gaussian distributions concentrate to have narrower width, as a response to the attractive near-field potential, with the centers depart from each other as time evolves, as a response to the repulsive far-field potential.

\begin{figure}[htbp]
    \centering
    \begin{subfigure}[b]{0.32\textwidth}
         \centering
         \includegraphics[width=\textwidth]{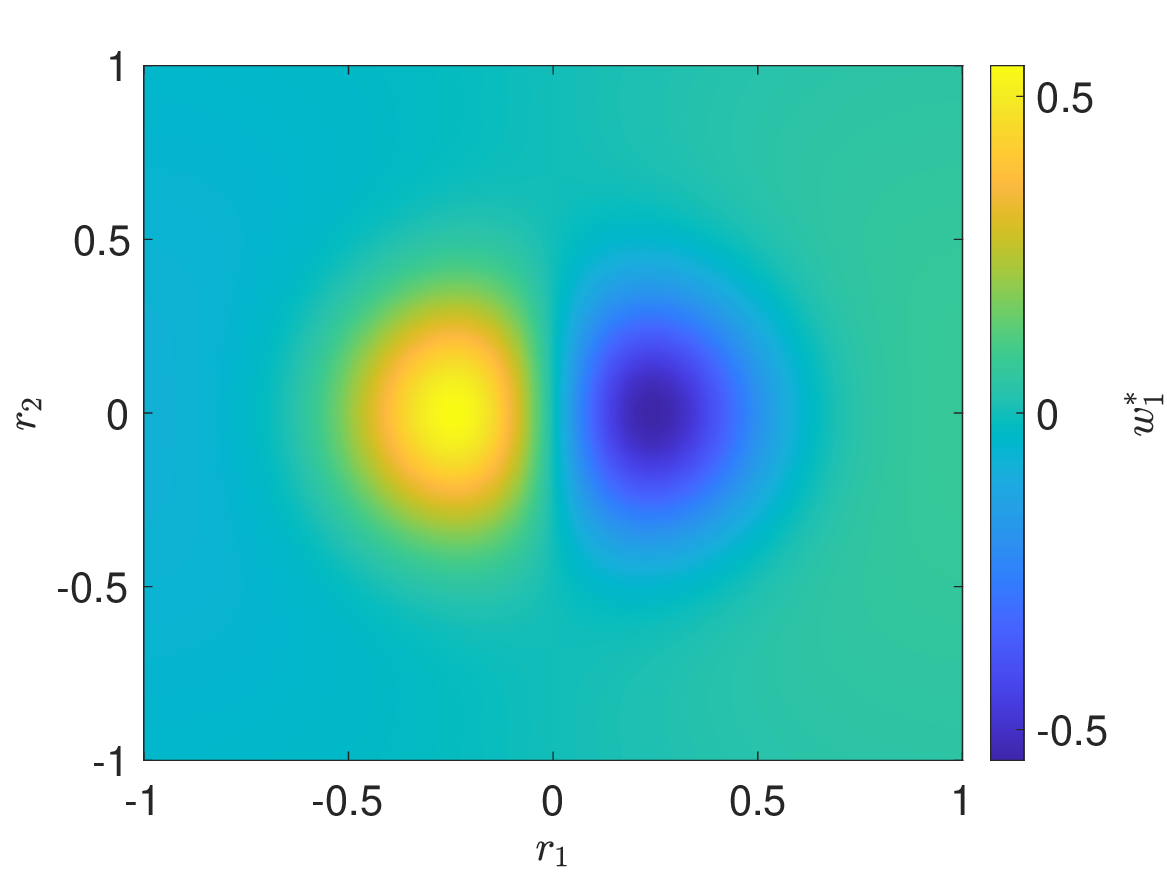}
    \caption{$w^{\ast}_{1}(r)$}
    \label{fig:wx_2d_star}
    \end{subfigure}
    \hfill
    \begin{subfigure}[b]{0.32\textwidth}
         \centering
         \includegraphics[width=\textwidth]{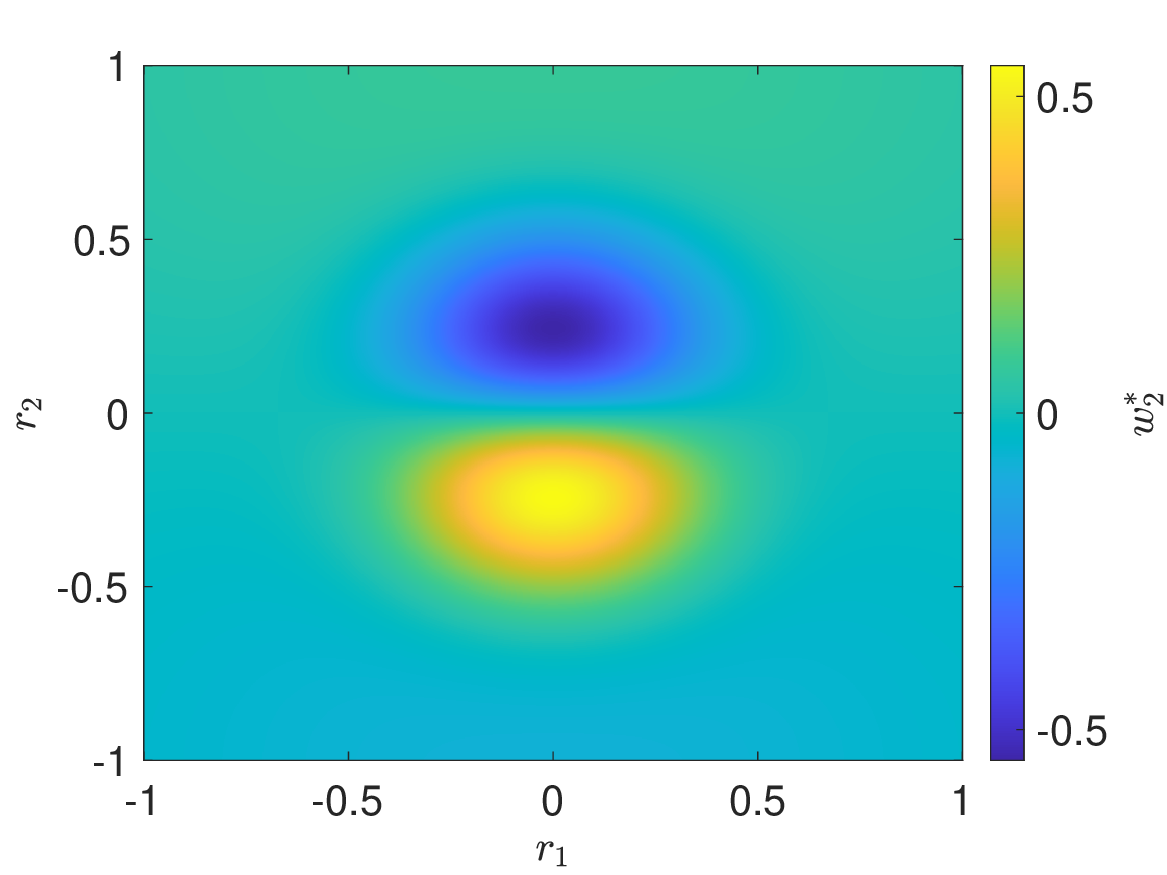}
    \caption{$w^{\ast}_{2}(r)$}
    \label{fig:wy_2d_star}
    \end{subfigure}
    \begin{subfigure}[b]{0.32\textwidth}
         \centering
         \includegraphics[width=\textwidth]{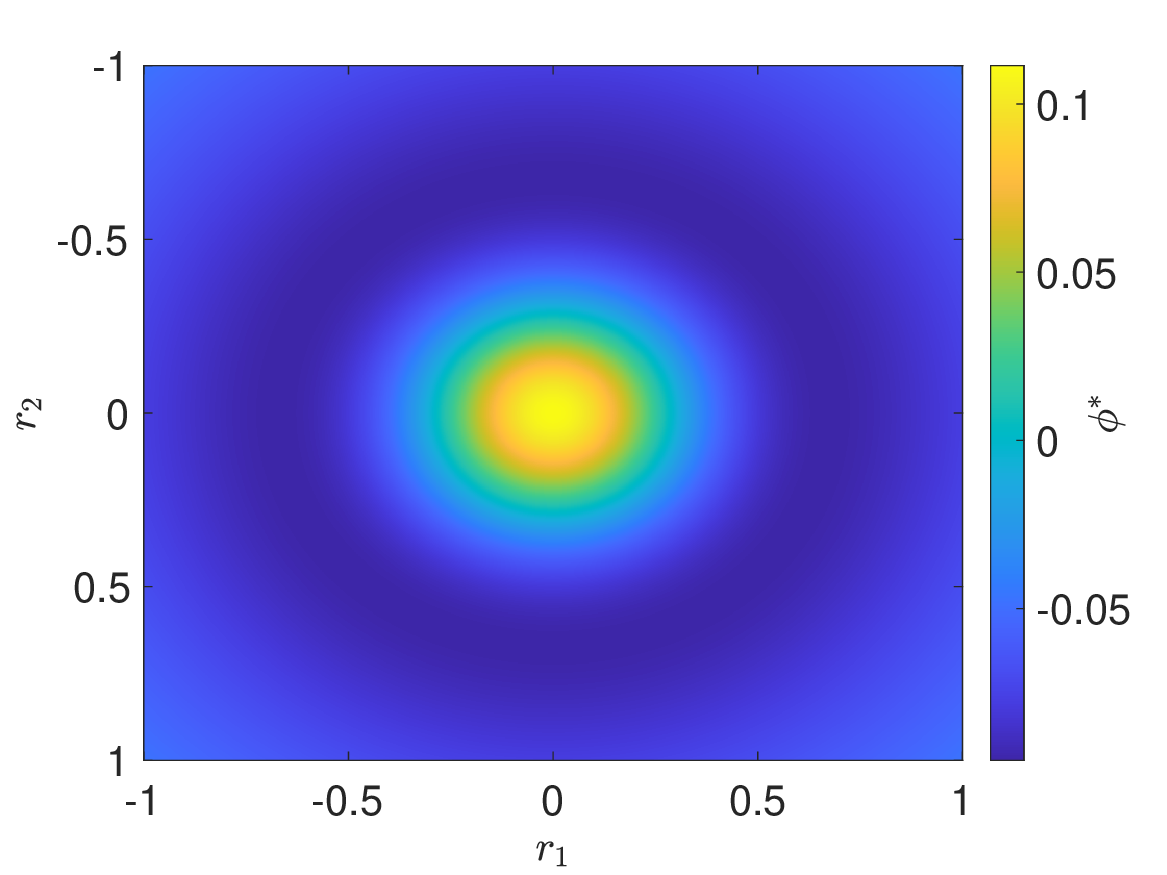}
    \caption{$\phi^\ast(r)$}
    \label{fig:phir_2d_potential}
    \end{subfigure}
    \\
    \begin{subfigure}[b]{0.32\textwidth}
         \centering
         \includegraphics[width=\textwidth]{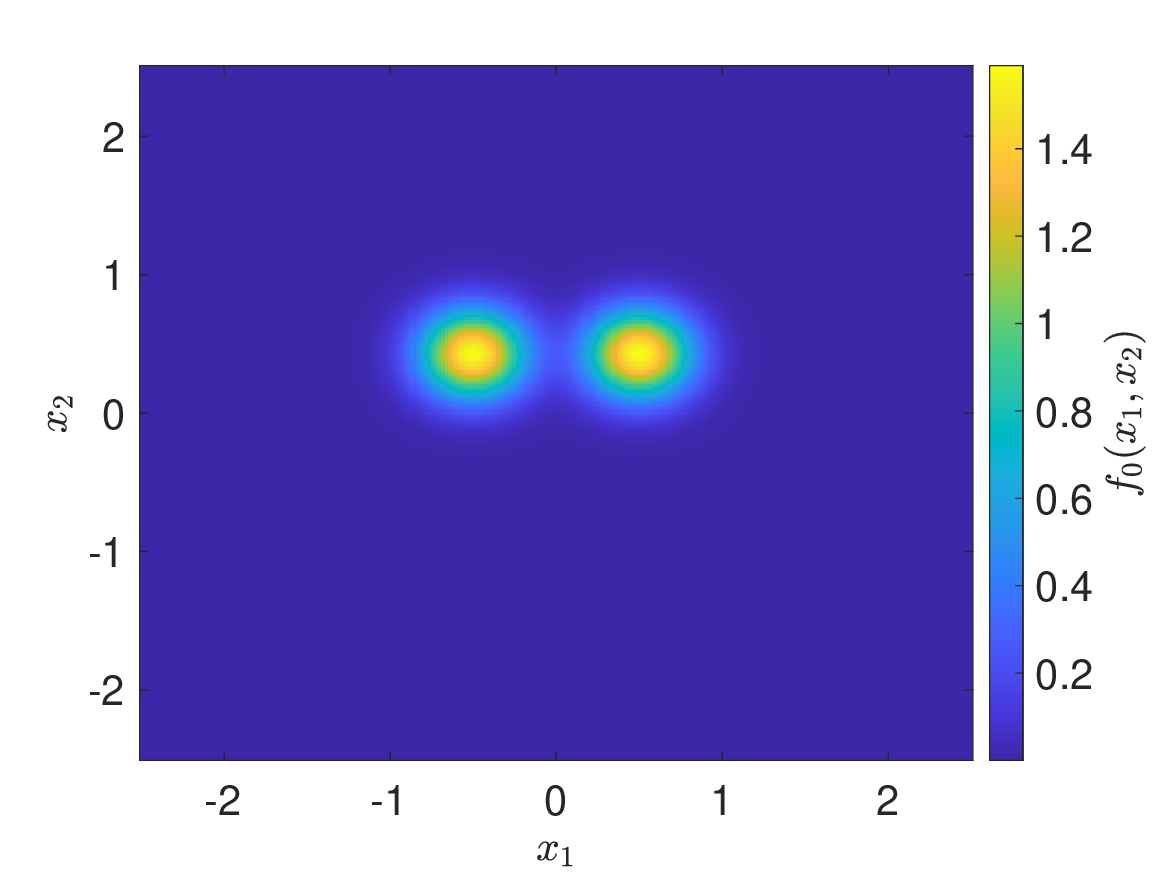}
    \caption{$f_0$}
    \label{fig:wr_2d_star_f0}
    \end{subfigure}
    \hfill
    \begin{subfigure}[b]{0.32\textwidth}
         \centering
         \includegraphics[width=\textwidth]{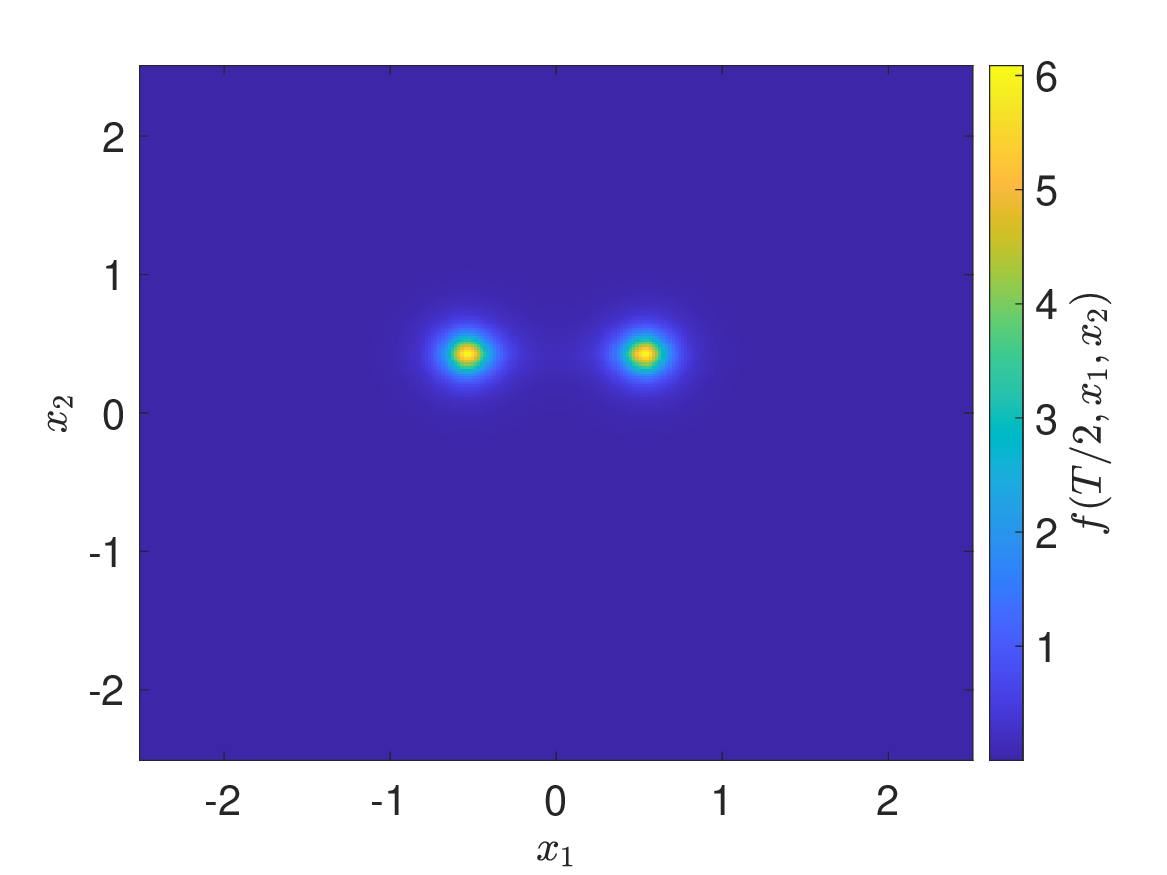}
    \caption{$f(T/2,x_1,x_2)$}
    \label{fig:wr_2d_star_f_halfT}
    \end{subfigure}
    \hfill
    \begin{subfigure}[b]{0.32\textwidth}
         \centering
         \includegraphics[width=\textwidth]{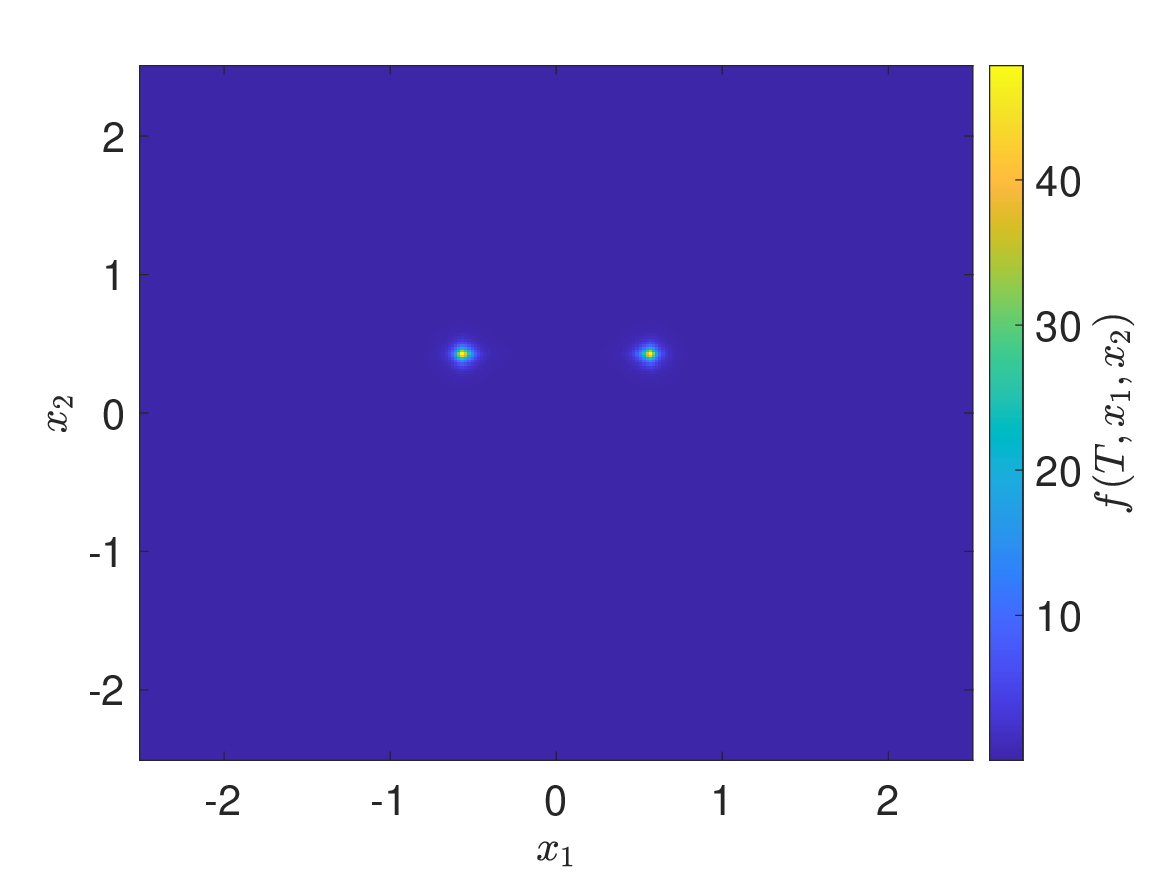}
    \caption{$f(T,x_1,x_2)$}
    \label{fig:wr_2d_star_f_T}
    \end{subfigure}
    \caption{Example~2: (A)-(B) present the ground-truth $w^\ast$ and (C) shows the ground-truth potential $\phi^*$ that induces $w^\ast$. (D)-(F) presents three time snapshots of the forward PDE simulation at $t=0$, $t=T/2$ and $t=T$, respectively.}
    \label{fig:wstar_2d_forward}
\end{figure}

In the inverse problem setting, we set the initial data as $f_0(x_1,x_2) := \phi_{\eps}\ast \mathbf{1}_{[-0.5, 0.5]\times[-0.5,0.5]}$, with $\phi_{\eps}$ being a Gaussian mollifier in $2D$, and the measurement to be the second moment $\nu(r) = \frac{1}{2}|r|^2$, measured at $T=1$. The reconstruction is performed using $N=15000$ particles drawn initially from $f_0$ through rejection sampling. The initial guess for the parameters is $(a_1^{(0)}, a_2^{(0)})=(2,0.4)$.

In Figure~\ref{fig:wr_2d_reconstruct} we compare the ground-true kernel with the reconstructed solution obtained at  various iterations, ranging from $k=0$ to $k=80$. The relative error is shown in the right three columns, with two rows showing results for the two components of $w=:(w_1, w_2)^\top$, $w_1$ and $w_2$ respectively. In Figure~\ref{fig:wr_2d_loss} we show one slice of $w_1$ at different iterations, and the decay of the relative $l^\infty$-error~\eqref{eq:L_infty_error}.

In Figure~\ref{fig:wr_2d_f_fN_dist}, we show the reference continuous density obtained using the ground-truth kernel (panels (A)-(C)), together with histograms of the empirical measure generated by the particle system~\eqref{eq:dX_couple_w} using the reconstructed kernel. The initial guess $w^{(0)}$ (panels (D)-(E)) fails to reproduce the reference density; in particular, the corresponding empirical density $f_N^{(0)}$ is significantly more concentrated near the origin at the final time $T$. After optimization, the reconstructed kernel $w^{(K)}$ yields a density (panels (F)-(G)) that closely matches the ground-truth solution.


\begin{figure}[htbp]
    \centering
    \begin{subfigure}[b]{0.24\textwidth}
         \centering
         \includegraphics[width=\textwidth]{fig_2D_wxstar_bird.eps}
    \caption{$w^{\ast}_{1}(r)$}
    \label{fig:wx_star_2d_N15k}
    \end{subfigure}
    \hfill
    \begin{subfigure}[b]{0.24\textwidth}
         \centering
         \includegraphics[width=\textwidth]{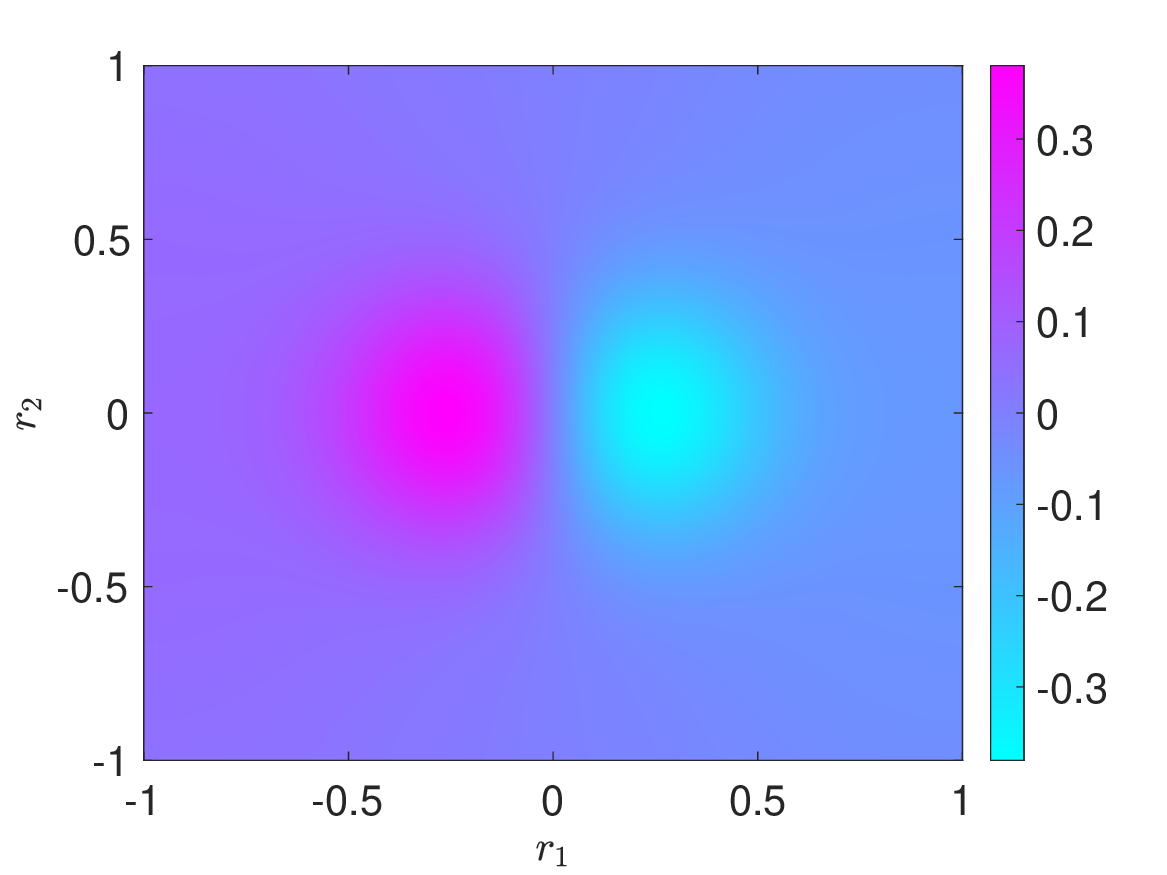}
    \caption{$(w^{(0)}_{1}-w^{\ast}_{1})/\|w^{\ast}\|_{l^\infty}$}
    \label{fig:wx_in_N15k}
    \end{subfigure}
    \hfill
    \begin{subfigure}[b]{0.24\textwidth}
         \centering
         \includegraphics[width=\textwidth]{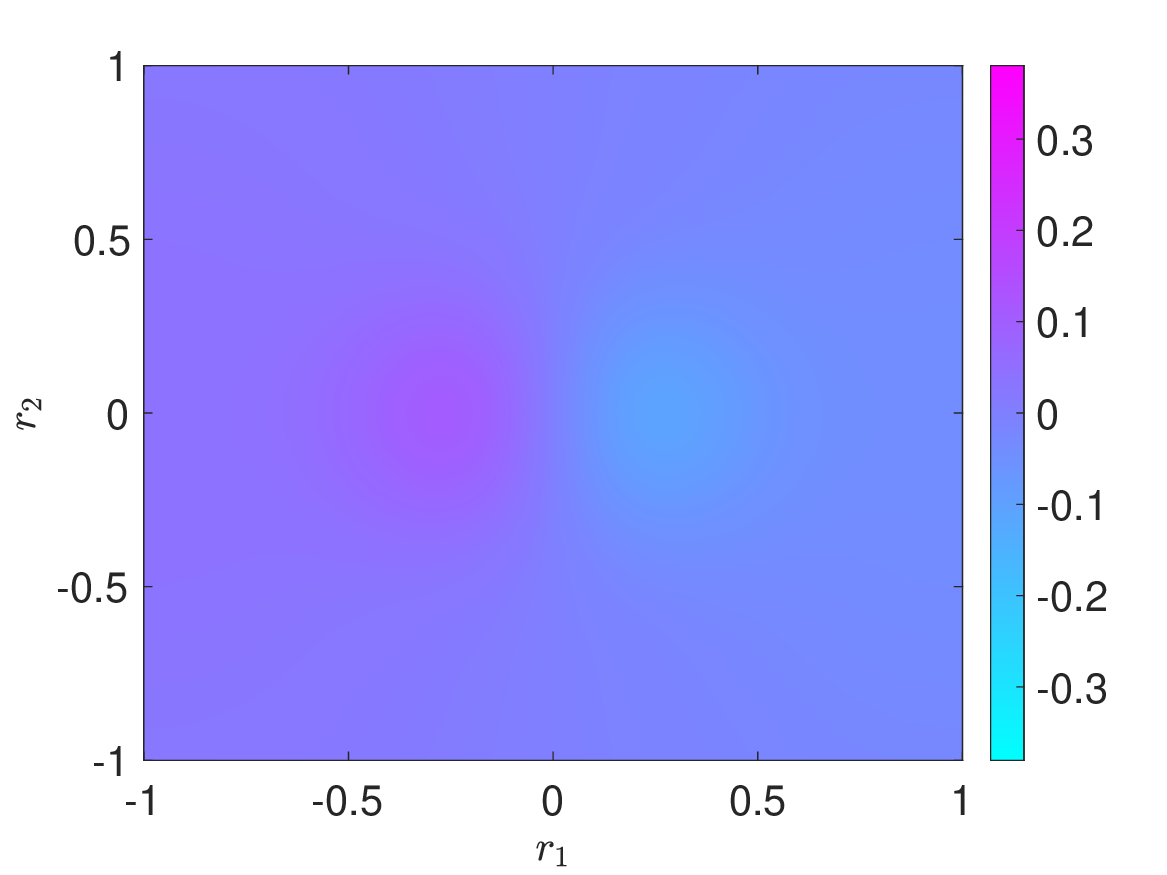}
    \caption{$(w^{(5)}_{1}-w^{\ast}_{1})/\|w^{\ast}\|_{l^\infty}$}
    \label{fig:wx_iter5_N15k}
    \end{subfigure}
    \hfill
    \begin{subfigure}[b]{0.24\textwidth}
         \centering
         \includegraphics[width=\textwidth]{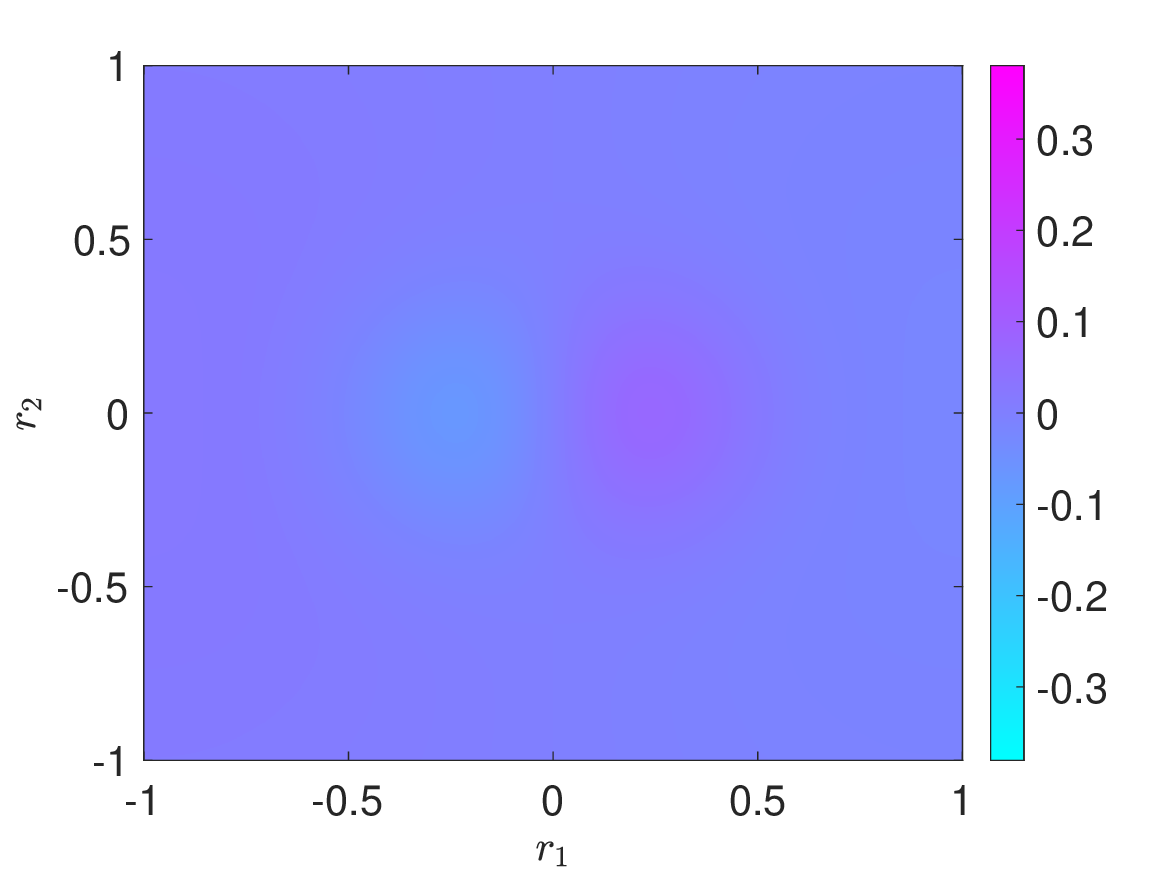}
    \caption{$(w^{(K)}_{1}-w^{\ast}_{1})/\|w^{\ast}\|_{l^\infty}$}
    \label{fig:wx_K_2d_N15k}
    \end{subfigure}
    \\
    \begin{subfigure}[b]{0.24\textwidth}
         \centering
         \includegraphics[width=\textwidth]{fig_2D_wystar_bird.eps}
    \caption{$w^{\ast}_{2}(r)$}
    \label{fig:wy_star_2d_N15k}
    \end{subfigure}
    \hfill
    \begin{subfigure}[b]{0.24\textwidth}
         \centering
         \includegraphics[width=\textwidth]{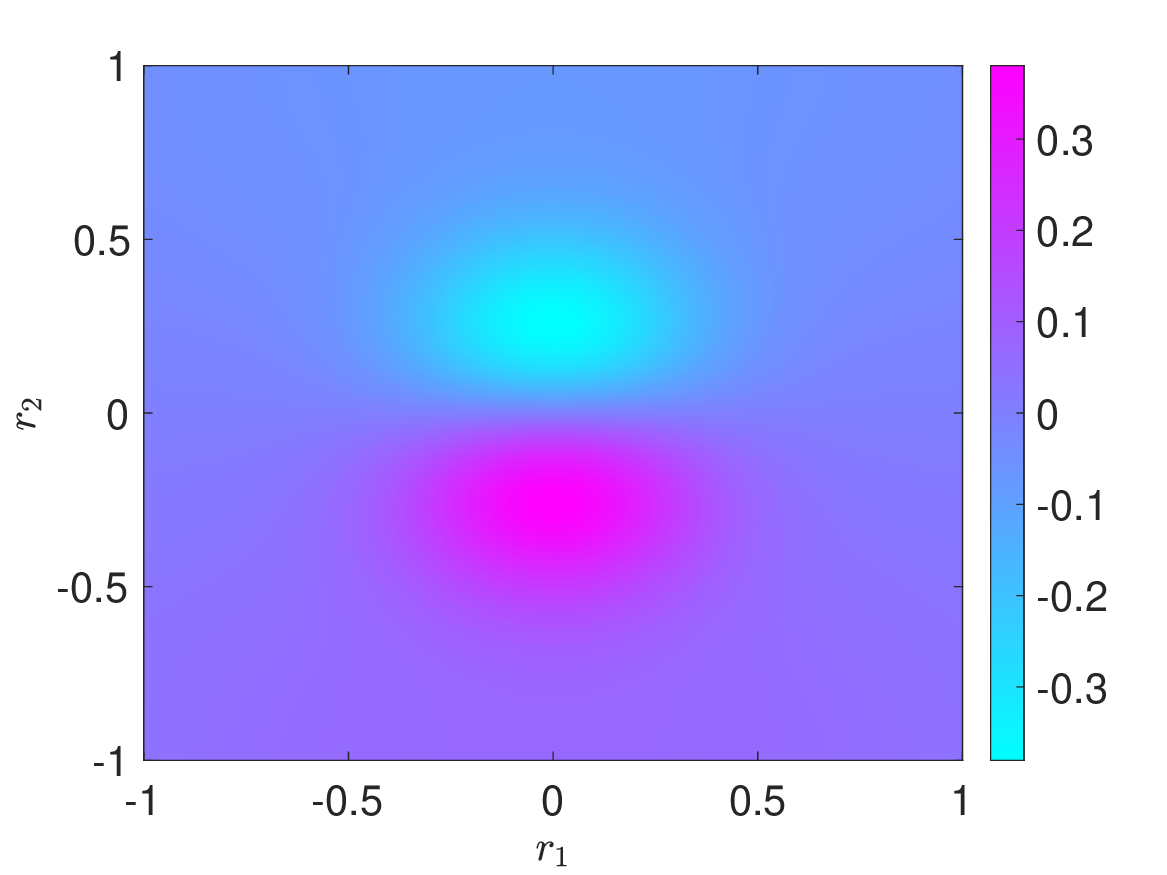}
    \caption{$(w^{(0)}_{2}-w^{\ast}_{2})/\|w^{\ast}\|_{l^\infty}$}
    \label{fig:wy_in_2d_N15k}
    \end{subfigure}
    \hfill
    \begin{subfigure}[b]{0.24\textwidth}
         \centering
         \includegraphics[width=\textwidth]{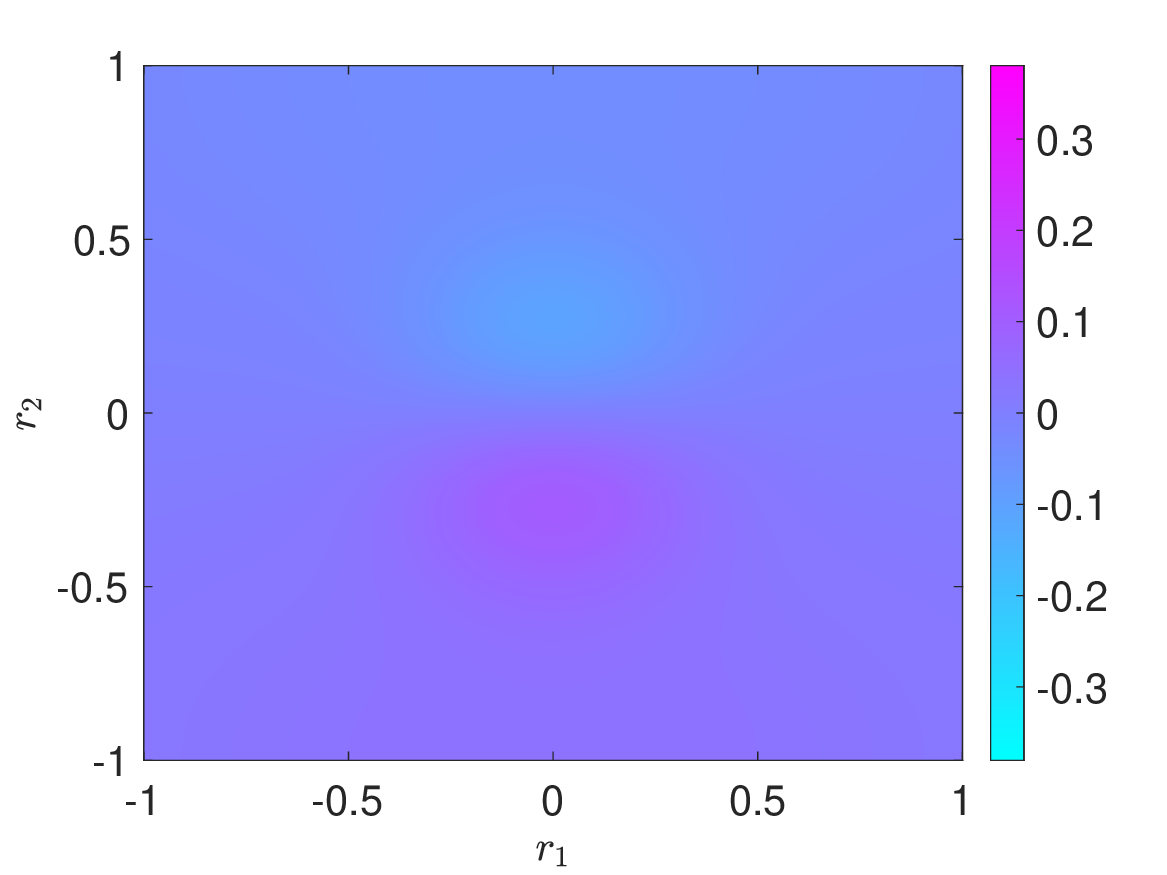}
    \caption{$(w^{(5)}_{2}-w^{\ast}_{2})/\|w^{\ast}\|_{l^\infty}$}
    \label{fig:wy_iter5_2d_N15k}
    \end{subfigure}
    \hfill
    \begin{subfigure}[b]{0.24\textwidth}
         \centering
         \includegraphics[width=\textwidth]{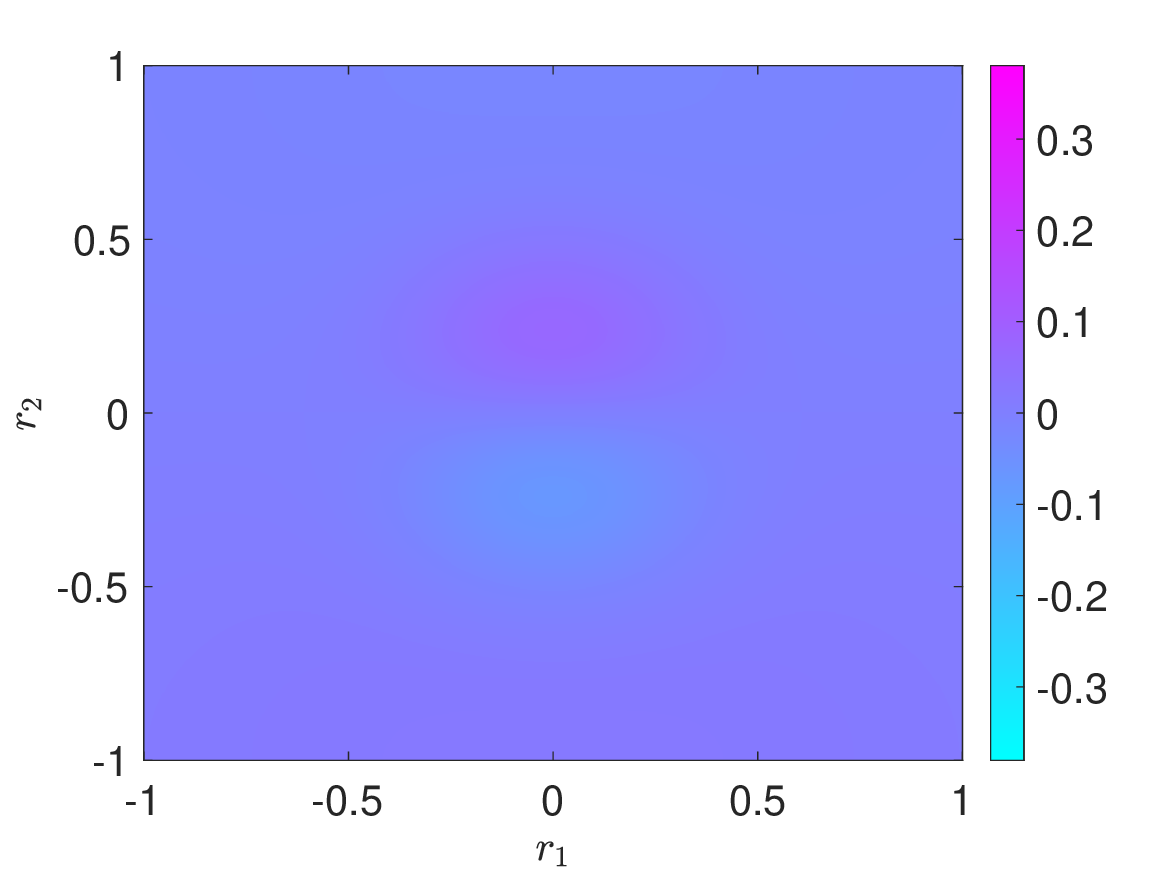}
    \caption{$(w^{(K)}_{2}-w^{\ast}_{2})/\|w^{\ast}\|_{l^\infty}$}
    \label{fig:wy_K_2d_N15k}
    \end{subfigure}
    \caption{Example~2: The first column shows the ground-truth kernel, while the next three columns display the error of the reconstructed kernel $w$ at three optimization iterations: the first, the 5th, and the final iteration $K=80$, all shown using a unified colormap.}
    \label{fig:wr_2d_reconstruct}
\end{figure}

\begin{figure}[htbp]
    \centering
    \begin{subfigure}[b]{0.42\textwidth}
         \centering
         \includegraphics[width=\textwidth]{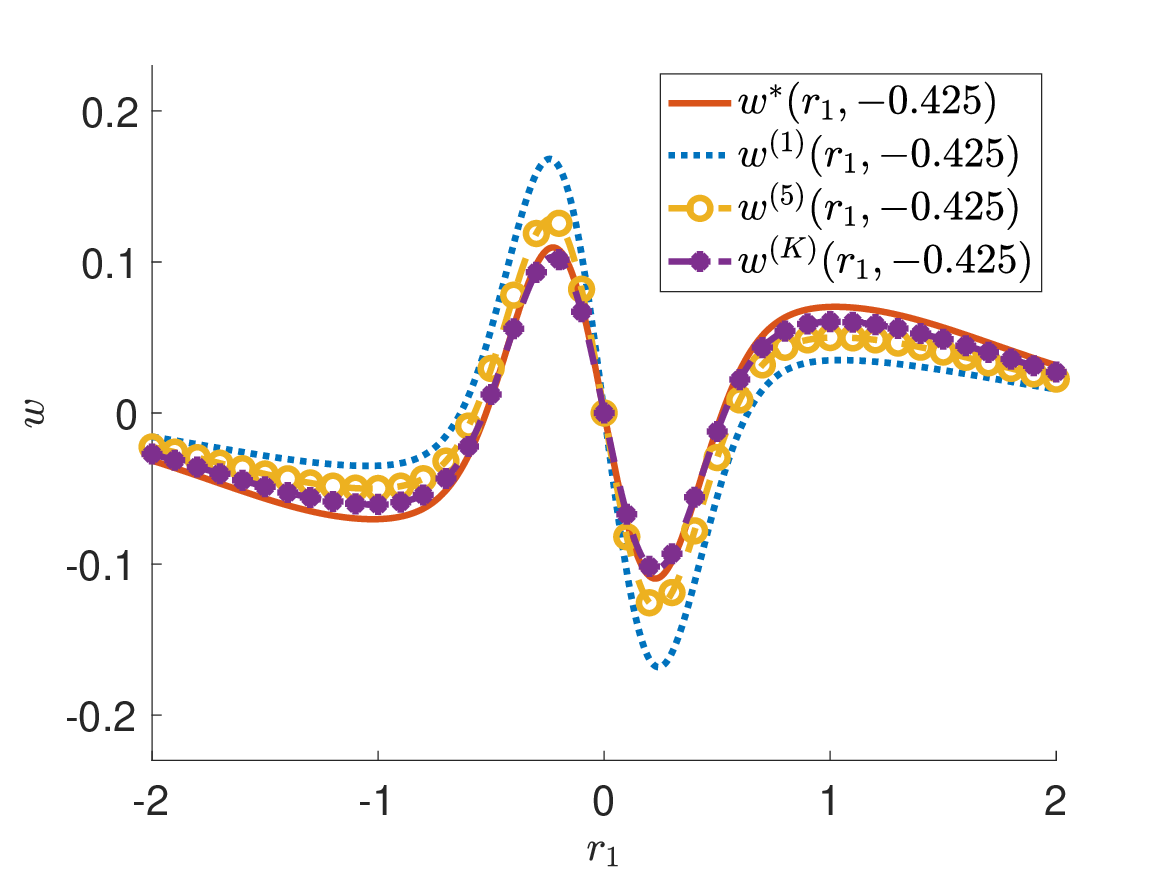}
    \caption{$w_1^{(k)}$}
    \label{fig:wr_2d_fixx1_iter_N15k}
    \end{subfigure}
    \hfill
    \begin{subfigure}[b]{0.42\textwidth}
         \centering
         \includegraphics[width=\textwidth]{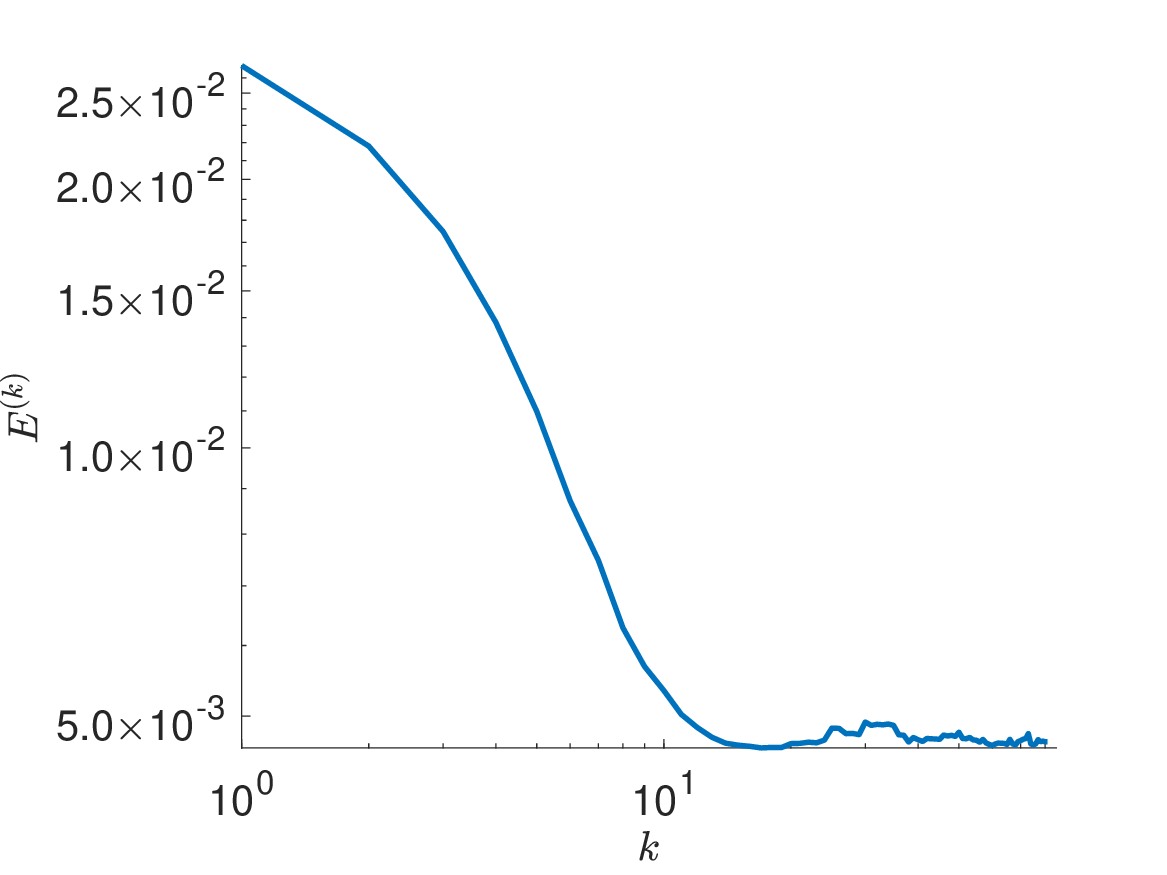}
    \caption{Relative $l^\infty$-error}
    \label{fig:wr_2d_Ek}
    \end{subfigure}
    \caption{Example~2: (A) presents a slice of the reconstructed $w_1$ ($x_2=-0.425$) at three optimization iterations, and (B) shows the decay of error along optimization iteration.}
    \label{fig:wr_2d_loss}
\end{figure}

\begin{figure}[htbp]
    \centering
    \begin{subfigure}[b]{0.32\textwidth}
         \centering
         \includegraphics[width=\textwidth]{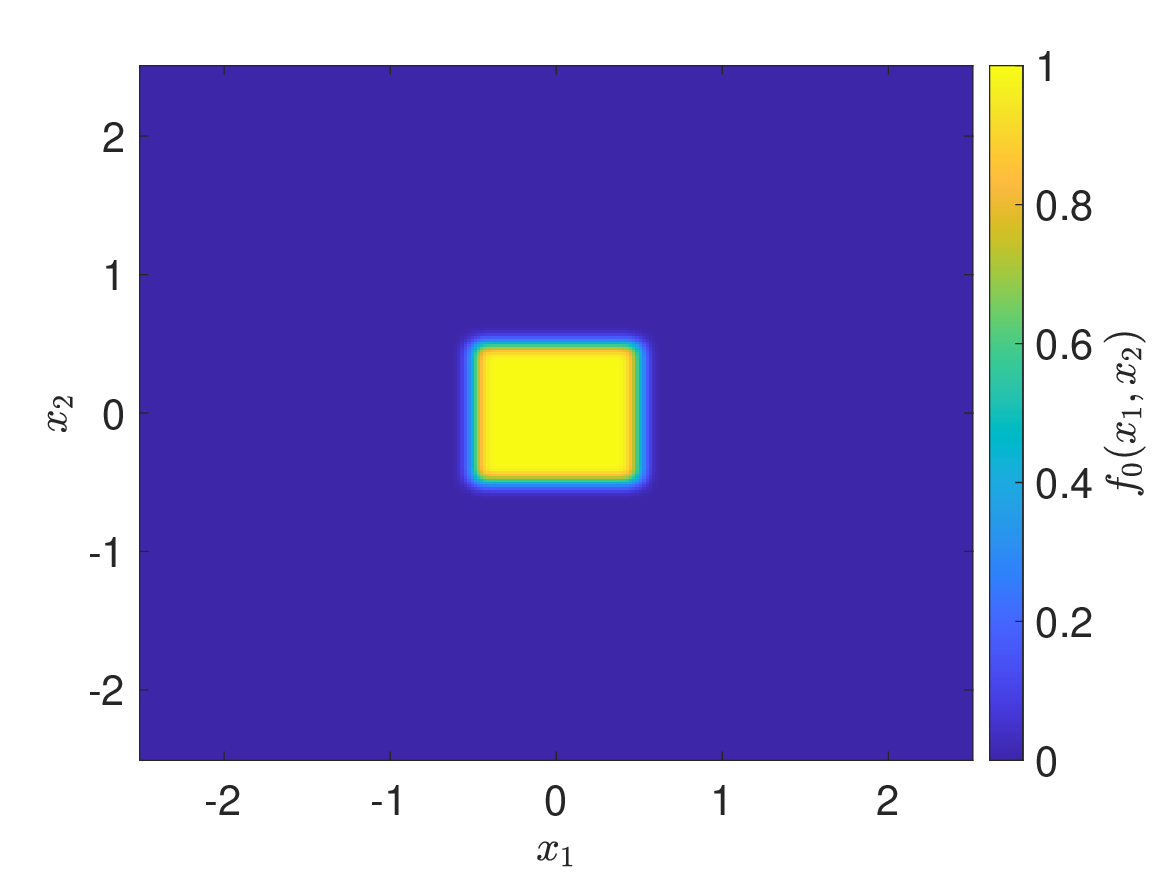}
    \caption{$f_0(x_1,x_2)$}
    \label{fig:wr_2d_fstar_at_0}
    \end{subfigure}
    \hfill
    \begin{subfigure}[b]{0.32\textwidth}
         \centering
         \includegraphics[width=\textwidth]{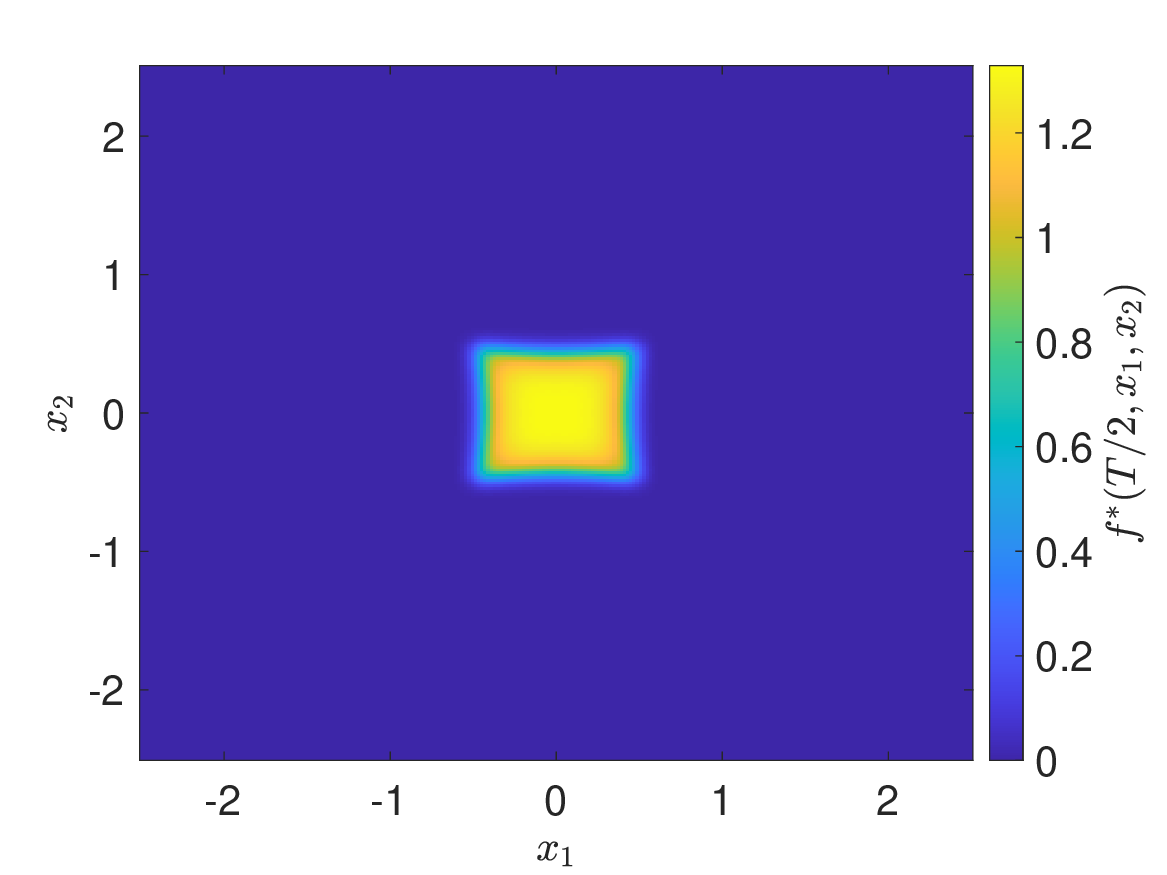}
    \caption{$f^\ast(t=T/2,x_1,x_2)$}
    \label{fig:wr_2d_fstar_at_halfT}
    \end{subfigure}
    \hfill
    \begin{subfigure}[b]{0.32\textwidth}
         \centering
         \includegraphics[width=\textwidth]{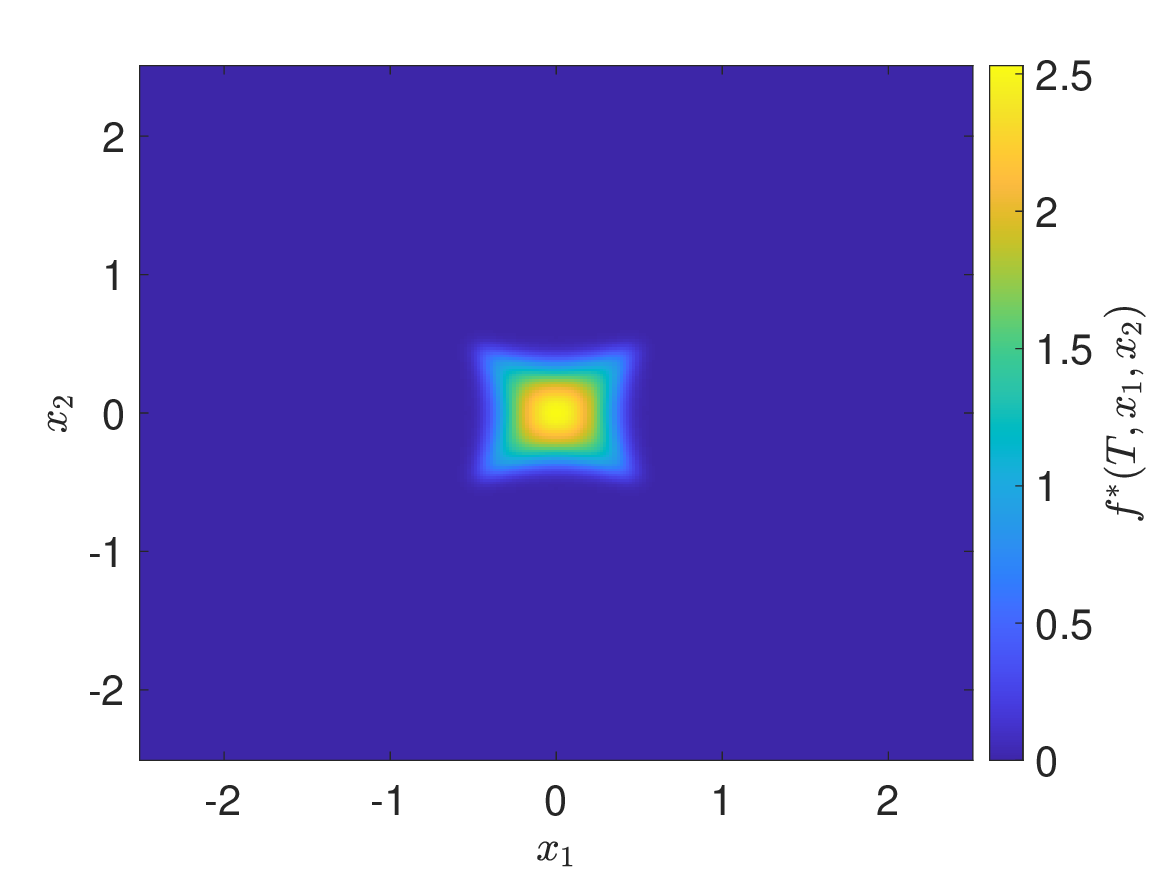}
    \caption{$f^{\ast}(t=T,x_1,x_2)$}
    \label{fig:wr_2d_fstar_at_T}
    \end{subfigure}
    \\
    \hspace{0.32\textwidth}
    \begin{subfigure}[b]{0.32\textwidth}
         \centering
         \includegraphics[width=\textwidth]{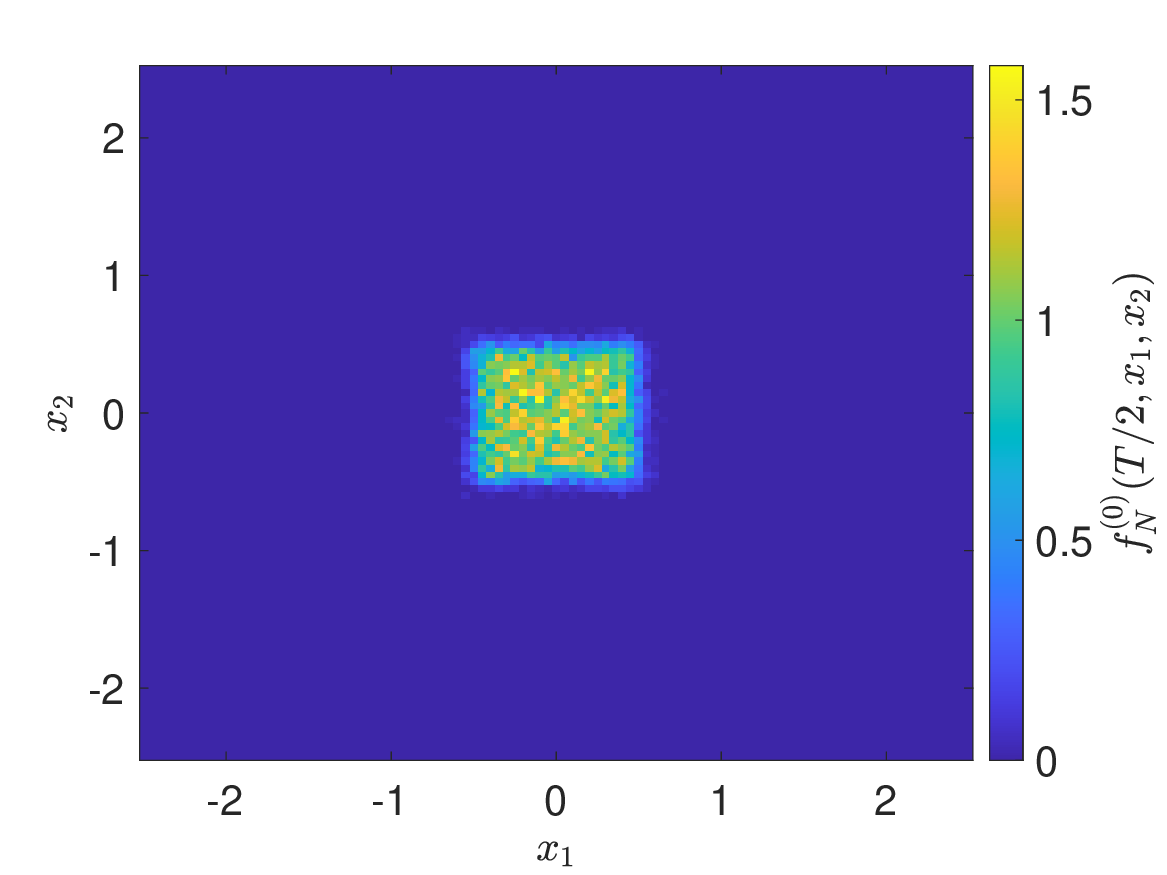}
    \caption{$f_N^{0}(t=T/2,x_1,x_2)$}
    \label{fig:wr_2d_fN0_at_halfT}
    \end{subfigure}
    \hfill
    \begin{subfigure}[b]{0.32\textwidth}
         \centering
         \includegraphics[width=\textwidth]{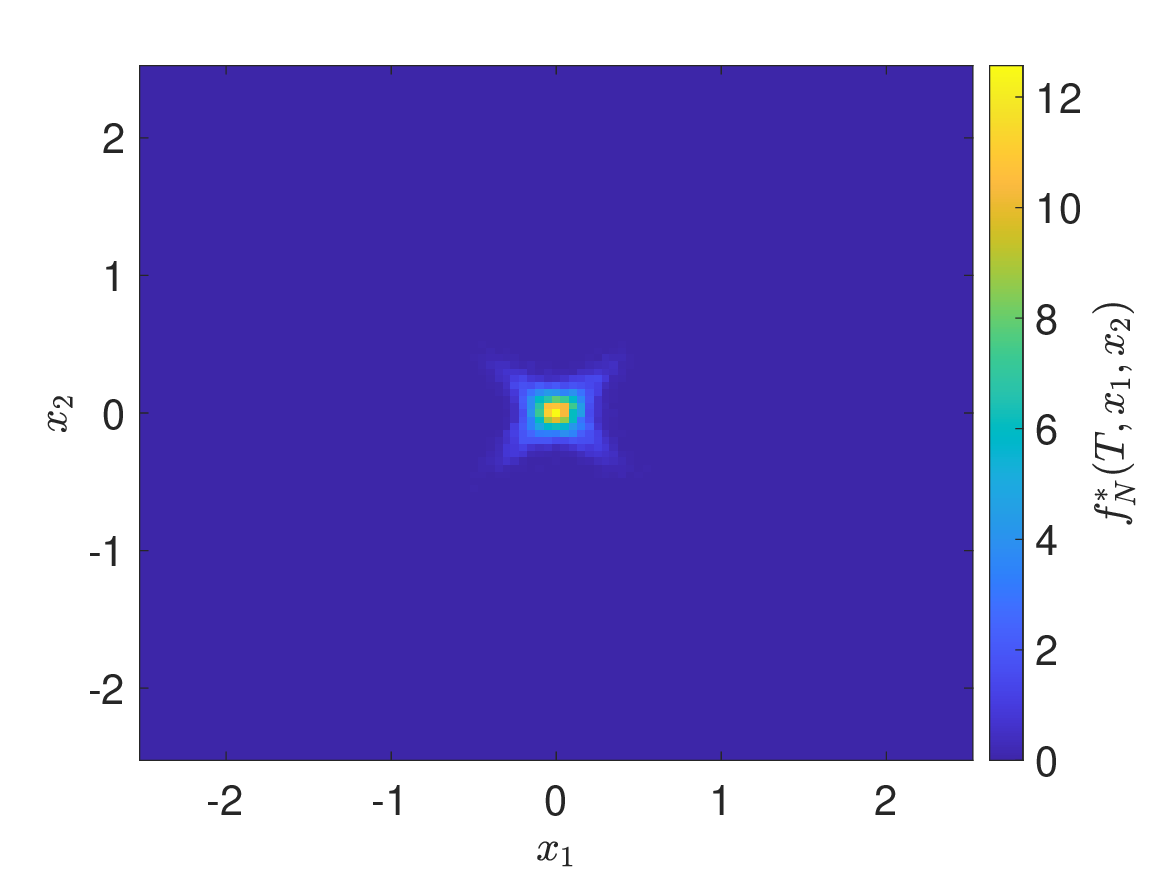}
    \caption{$f_N^{(0)}(t=T,x_1,x_2)$}
    \label{fig:wr_2d_fN0_at_T}
    \end{subfigure}
    \\
    \hspace{0.32\textwidth}
    \begin{subfigure}[b]{0.32\textwidth}
         \centering
         \includegraphics[width=\textwidth]{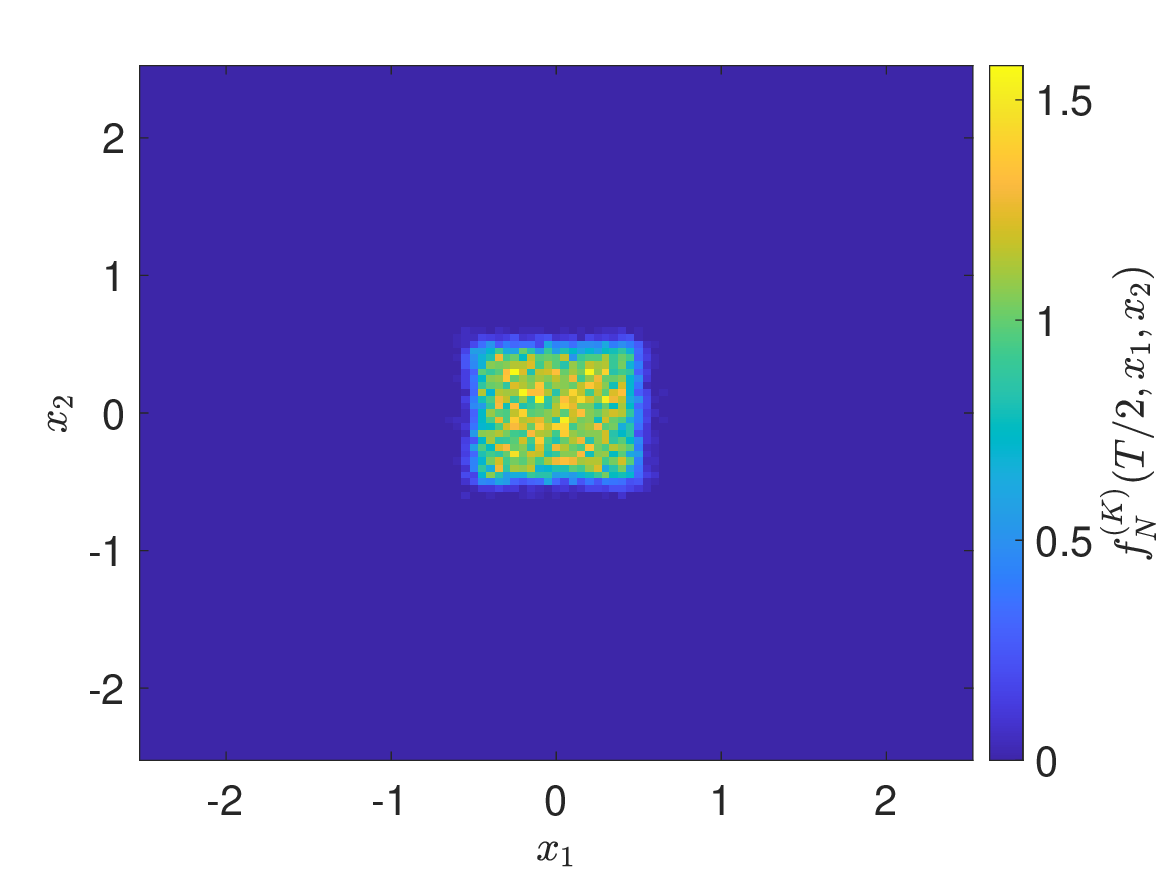}
    \caption{$f_N^{(K)}(t=T/2,x_1,x_2)$}
    \label{fig:wr_2d_fNK_at_halfT}
    \end{subfigure}
    \hfill
    \begin{subfigure}[b]{0.32\textwidth}
         \centering
         \includegraphics[width=\textwidth]{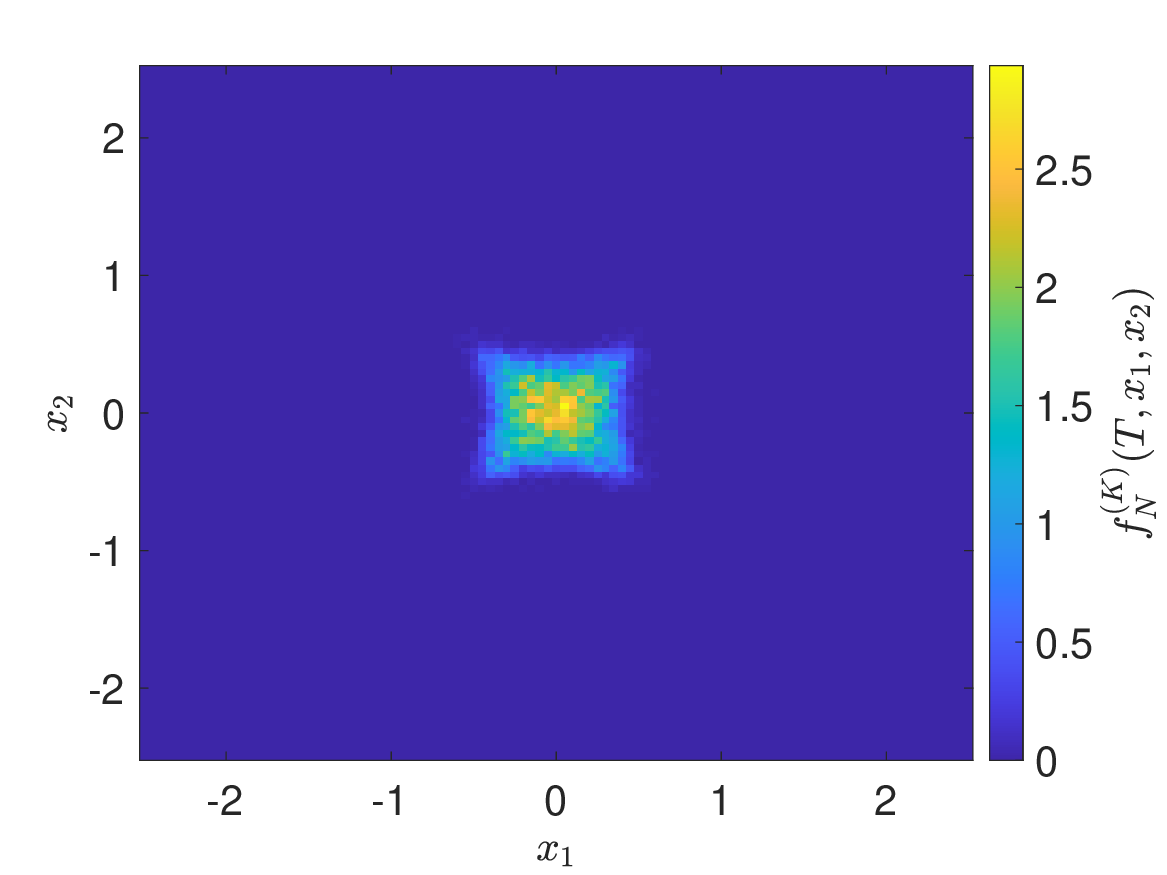}
    \caption{$f_N^{(K)}(t=T,x_1,x_2)$}
    \label{fig:wr_2d_fNK_at_T}
    \end{subfigure}
    \caption{Example~2: Comparison between the simulated distribution and the ground-true distribution (first row) at two different time ($T/2$ and $T$) using initial guess parameters (second row) and the optimized parameter (third row). Note that the simulated solution is produced as the histogram of the coupled ODE system~\eqref{eq:dX_wr_interact}, and the Monte Carlo fluctuation is unavoidable.}
    \label{fig:wr_2d_f_fN_dist}
\end{figure}

\section*{Acknowledgment.} 
P.~Chen and Q.~Li are supported by the National Science Foundation-DMS No.~2308440. P.~Chen is also supported by the Hirschfelder Scholar Fellowship award with funding from the Wisconsin Foundation \& Alumni Association and the Department of Mathematics at the University of Wisconsin-Madison. Q.~Li is further supported as Vilas Associate. L.~Wang is partially supported by National Science Foundation DMS-2513336.   Y.~Yang is partially supported by National Science Foundation under
Grant No.~DMS-2409855.

This material is based upon work supported by the National Science Foundation under Grant No.~DMS-2424139, while P.~Chen, Q.~Li and L.~Wang were in residence at the Simons Laufer Mathematical Sciences Institute in Berkeley, California, during the Fall 2025 semester.



\bibliographystyle{plain}
\bibliography{my_ref}

\appendix
\section{Proofs in Section~\ref{sec:linear_transport_equation}}\label{app:A}
\subsection{Proof of Lemma~\ref{prop:PDE_gradient}}
\begin{proof}
Let $g(t,x), g_0(x)$ be the Lagrange multipliers, and the Lagrangian is defined:
\begin{align*}
\mathcal{L}[a, f, f_{0},g,g_0]:=& \int \nu(x) f(T,x)\,\d x \\
& - \int_0^T\int \left[\partial_t f + \nabla_x\cdot (a(x)f)\right] g\,\d x\d t 
- \int \left(f(0,x) - f_{0}(x)\right) g_0(x)\,\d x\,.
\end{align*}
By integration by parts, we obtain:
\begin{align*}
\mathcal{L}[a, f,f_0,g,g_0]=& \int \left(\nu(x)-g(T,x)\right) f(T,x)\,\d x 
+ \int f(0,x) \left(g(0,x) - g_0(x)\right)\,\d x 
+ \int f_{0}(x) g_0(x)\,\d x\\
& + \int_0^T \int \left(\partial_t g + a(x) \cdot \nabla_x g \right) f \, \d x\d t \,.
\end{align*}
Then we take the variation of $\mathcal{L}$ with respect to each term:
\begin{align}
    \frac{\delta\mathcal{L}}{\delta f}(t,x) & = \partial_t g(t,x) + a(x)\cdot \nabla_x g(t,x)\,, \\
    \frac{\delta\mathcal{L}}{\delta f|_{t=T}}(x) & = \nu(x) - g(T,x)\,,
\end{align}
So the adjoint system~\eqref{eq:transport_adj_eq_g} is obtained by setting the above to be $0$.
Moreover, noticing that $\mathcal{L}$ is linear in $a$, the gradient is
\begin{equation}
\frac{\delta \mathcal{J}}{\delta a}(x) 
= \frac{\delta \mathcal{L}}{\delta a}(x) 
= \int_0^T (\nabla_x g) f\,\d t\,. 
\end{equation}
\end{proof}
\subsection{Proof of Lemma~\ref{prop:particle_gradient}}
\begin{proof}
Denote the Lagrange multipliers as $Y(t), Y_0$, and define the Lagrangian function as:
\begin{align*}
\mathsf{L}[a, X,X_0,Y,Y_0] := &\, \nu(X(T)) - \int_0^T \left[\dot{X}(t) -a(X(t))\right] Y(t)\,\d t - (X(0)-X_0)Y_0\,, \\
= &\, \nu(X(T)) - X(T)Y(T) + X(0)(Y(0) - Y_0) + X_0Y_0\\
& + \int_0^T X(t)\dot{Y}(t) + a(X(t)) Y(t)\,\d t \,,
\end{align*}
where the second line is obtained by integration by parts.
Then the variations of $\mathsf{L}$ with respect to each term can be computed:
\begin{align}
\frac{\delta \mathsf{L}}{\delta X}(t) & = \dot{Y}(t) + \nabla_x a(X(t))Y(t)\,,\\ 
\frac{\delta \mathsf{L}}{\delta X|_{t=T}} & = \nabla_x \nu(X(T)) - Y(T)\,,
\end{align}
and the adjoint system~\eqref{eq:transport_adjoint_Y} is obtained by imposing the above to $0$. 
Let $\tilde a$ be the perturbation of $a$, the perturbation of $\mathsf{L}$ induced by $\tilde{a}$ is
\begin{equation*}
\begin{aligned}
\mathsf{L}[a+\tilde a] - \mathsf{L}[a] & = \int_0^T \tilde a(X(t)) Y(t)\, \d t
= \int_0^T \int \tilde{a}(x) Y(t) \delta(x-X(t)) \, \d x \d t\,.
\end{aligned}
\end{equation*}
Hence, with $X(t)$ solving the forward problem~\eqref{eq:transport_forward_X}, we have
\begin{align*}
\frac{\delta \mathsf{J}}{\delta a}(x) = \frac{\delta \mathsf{L}}{\delta a}(x) & = \int_0^T Y(t,x) \delta(x-X(t)) \,\d t\,.
\end{align*}
\end{proof}

\section{Proofs in Section~\ref{sec:interaction}}\label{app:proof_thm_graderr}

\subsection{Proof of Proposition~\ref{thm:H_op_f_mean_field}}\label{subsec:pf_prop_31_deltah}
\begin{proof}
Denote $\delta h:=h_1-h_2$, where $h_1$ and $h_2$ solve~\eqref{eq:df_wr_mean_field} with $f_1$ and $f_2$ respectively. Then $\delta h$ satisfies:
\begin{equation}\label{eq:dt_deltah_At_Bt}
\begin{aligned}
    \partial_t \delta h + (w\ast f_1) \cdot \nabla_x \delta h 
    + \mathcal{A}_t[\delta h] 
    + \mathcal{B}_t[\delta f; h_2] = 0\,,\quad
    \delta h(T,x) = 0\,,
\end{aligned}
\end{equation}
where the linear operators $\mathcal{A}_t[\cdot], \mathcal{B}_t[\cdot]$ are:
\begin{equation*}
\begin{aligned}
\mathcal{A}_t[\delta h] &:= - \nabla w \ast (\delta h f_1) + (\nabla w \ast f_1) \delta h \,,\\
\mathcal{B}_t[\delta f; h_2] &:= (w\ast \delta f) \cdot \nabla_x h_2 - \nabla w \ast (\delta f h_2) + (\nabla w \ast \delta f) h_2\,.
\end{aligned}
\end{equation*}
Let $\xi_1(t)$ be the characteristics of~\eqref{eq:dt_deltah_At_Bt} solving $\frac{\d}{\d t}\xi_1(t) = (w\ast f_1)(t,\xi_1(t))$.
Confined to this characteristic, $\delta h$ solves:
\begin{equation}
\begin{aligned}
    \frac{\d}{\d t} \delta h(t,\xi_1(t))
    + \mathcal{A}_t[\delta h](t,\xi_1(t)) 
    + \mathcal{B}_t[\delta f; h_2](t,\xi_1(t)) & = 0\,,
    \quad \delta h(T,\xi_1(T)) = 0 \,.
\end{aligned}
\end{equation}
By Duhamel's principle,
\begin{equation}
\begin{aligned}
    \delta h(t,\xi_1(t))
    =\int_t^T \mathcal{A}_t[\delta h](s,\xi_1(s))\,\d s 
    + \int_t^T \mathcal{B}_t[\delta f; h_2](s,\xi_1(s))\,\d s\,.
\end{aligned}
\end{equation}
Since $w$ is Lipschitz (see Assumption~\ref{assump:property_w(r)}), for any fixed $(t,x)$, there exists a unique solution to the characteristic ODE $\xi_1(s), 0\leq s \leq T,$ such that $\xi_1(t)=x$. That is,
\begin{equation}\label{eq:deltah_At_Bt}
    \delta h(t,x)
    =\int_t^T \mathcal{A}_t[\delta h](s,\xi_1(s))\,\d s 
    + \int_t^T \mathcal{B}_t[\delta f; h_2](s,\xi_1(s))\,\d s\,.
\end{equation}
Note that as $f_1(t,\cdot)$ is a probability density, $\|f_1(t,\cdot)\|_{L^1_x}\equiv 1$, then 
\begin{equation}\label{eq:supx_Ath_Cbb}
\sup_{x}|\mathcal{A}_t[\delta h]|(t,x) \leq 2\Cbdd \sup_{x}|\delta h(t,x)|\,.
\end{equation}
As $\nabla_x h_2$ is bounded~\eqref{eq:u_on_xi_bdd}, 
the first term of $\mathcal{B}_t[\cdot]$ is bounded:
\begin{equation}\label{eq:supx_w_df_Clip}
\begin{aligned}
    \left|w \ast \delta f\right|(t,x) 
    = & \CLip \left|\E_{Z_1\sim f_1(t,\cdot)} \left[\frac{w(x-Z_1)}{\CLip}\right]
    - \E_{Z_2\sim f_2(t,\cdot)} \left[\frac{w(x-Z_2)}{\CLip}\right]\right|\,,\\
    \sup_x \left|w \ast \delta f\right|(t,x) \leq & \CLip \mathcal{W}_1(f_1(t,\cdot), f_2(t,\cdot))\,,
\end{aligned}
\end{equation}
where $\mathcal{W}_1(\cdot,\cdot)$ denotes the Wasserstein-$1$ distance and we use its duality definition. By Assumption~\ref{assump:property_w(r)}, $\frac{1}{\CLip} w(\cdot)$ is a $1$-Lipschitz function.
Similar estimates can be applied to other terms, and we have:
\begin{equation}
\begin{aligned}
    \sup_x |\mathcal{B}_t[\delta f; h_2]|(t,x) \leq & \CLip \mathcal{W}_1(f_1(t,\cdot), f_2(t,\cdot)) |\nabla_x h_2|(t,x) 
    + 2 \CLip \mathcal{W}_1(f_1(t,\cdot), f_2(t,\cdot)) |h_2|(t,x)\,,\\
    \leq & C(\CLip, \Cbdd, T, \nu) \mathcal{W}_1(f_1(t,\cdot), f_2(t,\cdot))\,.
\end{aligned}
\end{equation}
We refer to Lemma~\ref{thm:semigroup_At_h_xi_sol} and the proof of Proposition~\ref{prop:D2} for estimates of $h_2, \nabla_x h_2$.
Plugging~\eqref{eq:supx_Ath_Cbb} and~\eqref{eq:supx_w_df_Clip} into~\eqref{eq:deltah_At_Bt}, the following inequality can be obtained: 
\begin{equation}
\begin{aligned}
    |\delta h|(t,x)
    & \leq 2\Cbdd\int_t^T \sup_{y}|\delta h(s,y)| \,\d s 
    + \int_t^T \mathcal{B}_t[\delta f; h_2](s,\xi_1(s))\,\d s\,, \\
    & \leq 2\Cbdd \int_t^T \sup_{y}|\delta h(s,y)| \,\d s 
    + C(\CLip, \Cbdd, T, \nu) \int_t^T \mathcal{W}_1(f_1(s,\cdot), f_2(s,\cdot))\,\d s\,, \\
    \sup_{x} |\delta h|(t,x) & \leq C(\CLip, \Cbdd, T, \nu) \int_{t}^T \mathcal{W}_1(f_1(s,\cdot), f_2(s,\cdot))\,\d s\,,\\
    \E \sup_{x} |\delta h|(t,x) & \leq C(\CLip, \Cbdd, T, \nu) \int_{t}^T \E\mathcal{W}_1(f_1(s,\cdot), f_2(s,\cdot))\,\d s\,,
\end{aligned}
\end{equation}
thanks to Gronwall's inequality, where we use the face that $\int_t^T \mathcal{W}_1(f_1(s,\cdot), f_2(s,\cdot))\,\d s$ is a decreasing function in $t$, and the linearity of expectation.
By setting $f_1 = f_{N}, f_2 = f$ and applying~\eqref{eq:wass_dist_f_N_f},  we have that for every $(t,x)$,
\begin{equation}
    \E |\delta h|(t,x) \leq C(\CLip, \Cbdd, T, \nu) \sup_{0\leq s\leq T} \E\mathcal{W}_1(f_N(s,\cdot), f(s,\cdot)) \leq C N^{-\lambda} \to 0\,,\ \text{as } N\to\infty\,.
\end{equation}

\end{proof}

\subsection{Auxiliary lemmas for Theorem~\ref{thm:gradient_err_estimate}}\label{app:proofs_thm}

\begin{lemma}\label{thm:semigroup_At_h_xi_sol}
Let $h$ solve~\eqref{eq:nonlinear_adj_dxg_h}, then $h(t,x)$
is continuous in $x$ and bounded in $[0,T]\times \R^d$, i.e., there exists some positive constant $C$ such that
\begin{equation}\label{eq:h_tx_Linfty_bdd}
\|h\|_{L^{\infty}([0,T]; L^{\infty}(\R^d))} 
\leq C(T,\Cbdd, \|\nu\|_{W^{1,\infty}})\,.
\end{equation}
The same holds for $h_N$ solving~\eqref{eq:nonlinear_fN_induce_hN}.
\end{lemma}

\begin{proof}[Proof of Lemma~\ref{thm:semigroup_At_h_xi_sol}]
Note that the characteristic of~\eqref{eq:nonlinear_adj_dxg_h} is $X(t)$ solving~\eqref{eq:dX_couple_w}. 
Since $w$ is globally Lipschitz (Assumption~\ref{assump:property_w(r)}) and $f\in L^1(\R^d)$ for all $t$, there exists a unique solution to~\eqref{eq:dX_couple_w}, $X(t;t_0,x_0)$, given any point on the trajectory $(t_0,x_0)$, such that $X(t_0;t_0,x_0)=x_0$, and $X(t;t_0,x_0)$ is continuous (stable) with respect to $x_0$.

Confined on the characteristic, $h(t,X(t))$ solves:
\begin{equation}\label{eq:dhdt_on_xi_Ah}
\frac{\d}{\d t} h(t,X(t)) = \mathcal{A}_t[h](t,X(t))\,,
\quad h(T,X(T)) = \nabla_x \nu(X(T))
\end{equation}
where $\mathcal{A}_t[\cdot]$ denotes the linear operator that 
\begin{equation}\label{eq:generator_def_At_h}
\mathcal{A}_t[h](t,x):=\nabla w\ast (h f)(t,x) - (\nabla w\ast f) h(t,x)\,.
\end{equation}
Then $\mathcal{A}_t[\cdot]: L^{\infty}([0,T],L^{\infty}(\R^d)) \to L^{\infty}([0,T],L^{\infty}(\R^d))$ is uniformly bounded. 
Indeed, for any $h(t,x)\in L^{\infty}([0,T],L^{\infty}(\R^d))$, the following holds:
\begin{equation*}
\begin{aligned}
\left|\mathcal{A}_t[h](t,x)\right| 
\leq & \left|\nabla w\ast (h f)(t,x)\right| + \left| (\nabla w\ast f)\right| |h(t,x)|\,,\\
\leq & \Cbdd \left|\int h(t,y) f(t,y)\, \d y\right| + \Cbdd \left|\int f(t,y)\,\d y \right| |h(t,x)|\,, 
\end{aligned}
\end{equation*}
where we recall that $\|\nabla w\|_{L^\infty}\leq \Cbdd$ from Assumption~\ref{assump:property_w(r)}. Since $\|f(t,\cdot)\|_{L_x^1}$ is conserved for all time, we can apply the H\"older inequality to the first term and obtain
\begin{equation*}
\left|\mathcal{A}_t[h](t,x)\right| \leq 
\Cbdd \|h(t,\cdot)\|_{L^{\infty}_x} \|f(t,\cdot)\|_{L^{1}_x} 
+ \Cbdd \| f(t,\cdot)\|_{L^{1}_x} |h(t,x)|\,.
\end{equation*}
Taking $\sup_{t,x}$ on both sides gives
\begin{equation*}
\sup_{t,x} |\mathcal{A}_t[h](t,x)| \leq C \sup_{t,x} |h(t,x)|\,,
\end{equation*}
here $C = 2\Cbdd\|f_0\|_{L^{1}_x}$. Therefore, $\mathcal{A}_t[h] \in L^{\infty}([0,T], L^{\infty}(\R^d))$ with $\|\mathcal{A}_t\|_{\mathrm{op}} \leq C$.

Moreover, this bounded linear operator $\mathcal{A}_t[\cdot]$ generates a linear, uniformly continuous semigroup $\mathcal{S}_t$~\cite{pazy1992semigroups} such that
\begin{equation}\label{eq:dS_semigroup}
\frac{\d \mathcal{S}_t}{\d t} = \mathcal{A}_t \mathcal{S}_t, \mathcal{S}_0 = I,
\text{with}\, \|\mathcal{S}_t\|_{\mathrm{op}}\leq C(T, \|\mathcal{A}_t\|_{\mathrm{op}})\,,
\end{equation}
and the unique solution to~\eqref{eq:dhdt_on_xi_Ah} is continuously differential, represented as
\begin{equation}\label{eq:h_on_xi_sol}
    h(t,X(t))=\mathcal{S}_{t-T}[h(T,X(T)] 
    = \mathcal{S}_{t-T}[\nabla_x \nu (X(T))]\,,
\end{equation}
by the given terminal time condition. The proof is concluded by choosing the characteristic to be $X(s) = X(s;t,x) $ for any given $(t,x)$.

Carrying the same argument for~\eqref{eq:nonlinear_fN_induce_hN} leads to $\|h_N\|_{L^{\infty}([0,T]; L^{\infty}(\R^d))} \leq C(T,\Cbdd, \|\nu\|_{W^{1,\infty}}).$
\end{proof}

\begin{lemma}\label{lem:deltah_sol}
Let $h$ and $h_N$ solve~\eqref{eq:nonlinear_adj_dxg_h} and~\eqref{eq:nonlinear_fN_induce_hN} respectively, and $X_i(t)$ solve~\eqref{eq:dX_wr_interact}.
Denote $\delta h(t,x) :=(h-h_N)(t,x)$, 
and $\delta f(t,x) :=(f-f_N)(t,x)$ then there exists some positive constant $C(\Cbdd, T)$, such that
\begin{equation}
\frac{1}{N}\sum_{i=1}^N |\delta h|(T-t, X_i(T-t))
\leq
C(\Cbdd, T)\int_0^{T-t} \frac{1}{N}\sum_{i=1}^N |S(s,X_i(s))|\,\d s\,,
\end{equation}
where $S$ is defined as in~\eqref{eq:H_dh_source_def}.
\end{lemma}

\begin{proof}[Proof of Lemma~\ref{lem:deltah_sol}]
By taking the difference between~\eqref{eq:nonlinear_adj_dxg_h} and~\eqref{eq:nonlinear_fN_induce_hN}, 
$\delta h$ solves
\begin{equation}\label{eq:deltah_At_H}
\partial_t \delta h 
+ (w\ast f_N) \nabla_x \delta h 
= \mathcal{A}_t[\delta h]
+ H\,, 
\quad \delta h (T,x) =0
\end{equation}
where $\mathcal{A}_t[\cdot]$ is a linear operator, defined as 
\begin{equation}\label{eq:At_deltah_def}
\mathcal{A}_t[\delta h]:=
\nabla w\ast (\delta{h} f_N)(t,x) - (\nabla w\ast f_N) \delta{h}(t,x)\,,
\end{equation}
and the source term is
\begin{equation}\label{eq:source_term_def}
H(t,x):= -(w\ast \delta f) \nabla_x h
+ \nabla w\ast (h_N\delta f) 
- (\nabla w \ast \delta f) h_N \,.
\end{equation}

Define the backward time variable $\tau := T - t$ and write each term of~\eqref{eq:deltah_At_H} in terms of $\tau$,
then $\delta h$ solves
\begin{equation}\label{eq:delta_h}
\begin{cases}
-\partial_\tau \delta h(\tau,x)
+ (w * f_N)(\tau,x)\cdot\nabla_x \delta h(\tau,x)
=
\mathcal{A}_\tau[\delta h](\tau,x) + H(\tau,x),
\quad 0\leq \tau \leq T
\\
\delta h(0,x) = 0.
\end{cases}
\end{equation}
The characteristics of $\delta h$ is $X_i(\tau) = X_i(T-t)$,
indeed,
\begin{equation}
\frac{\d}{\d\tau} X_i(\tau)
=
-\frac{\d}{\d t} X_i(T-t)
=
-\frac{1}{N}\sum_{j=1}^N w(X_i - X_j)
=
-(w * f_N)(\tau, X_i(\tau))\,.
\end{equation}
Confined on the characteristics, the equation of $\delta{h}$~\eqref{eq:delta_h} becomes
\begin{equation}\label{eq:delta_h_on_Xi}
\frac{\d}{\d\tau}
\delta h(\tau,X_i(\tau))
=
-\mathcal{A}_\tau[\delta h](\tau,X_i(\tau))
- H(\tau,X_i(\tau)).
\end{equation}
By evaluating~\eqref{eq:At_deltah_def} and~\eqref{eq:source_term_def} on $(\tau,X_i(\tau))$, $|\delta h|$ satisfies:
\begin{equation}
\begin{aligned}
\frac{\d}{\d\tau}|\delta h(\tau,X_i(\tau))|
\le & 
\frac{1}{N}\sum_{j=1}^N
|\nabla w(X_i(\tau)-X_j(\tau))|\,
|\delta h(\tau,X_j(\tau))|
\\
&+
\frac{1}{N}\sum_{j=1}^N
|\nabla w(X_i(\tau)-X_j(\tau))|\,
|\delta h(\tau,X_i(\tau))|
+ |H(\tau,X_i(\tau))|\,,
\end{aligned}
\end{equation}
where we substitute $f_N(\tau,x)$ by its definition $\frac{1}{N}\sum_{j=1}^N \delta(x-X_j(\tau))$.
Averaging over $N$ trajectories $\{X_i(\tau)\}_{i=1}^N$ gives
\begin{equation}\label{eq:dtau_dsumi_h}
\begin{aligned}
\frac{\d}{\d\tau}
\left(
\frac{1}{N}\sum_{i=1}^N |\delta h(\tau,X_i(\tau))|
\right)
\le\;&
\frac{1}{N^2}\sum_{i,j}
|\nabla w(X_i(\tau)-X_j(\tau))|\,
|\delta h(\tau,X_j(\tau))|
\\
&+
\frac{1}{N^2}\sum_{i,j}
|\nabla w(X_i(\tau)-X_j(\tau))|\,
|\delta h(\tau,X_i(\tau))|
+
\frac{1}{N}\sum_i |H(\tau,X_i(\tau))|\,.
\end{aligned}
\end{equation}
According to the boundedness assumption on $w$ (Assumption~\ref{assump:property_w(r)}), and denote
$
S(\tau):=\frac{1}{N}\sum_{i=1}^N |\delta h(\tau,X_i(\tau))|\,,
$
the above inequality~\eqref{eq:dtau_dsumi_h} is rewritten into
\begin{equation}
\frac{\d}{\d\tau} S(\tau)
\le
2\Cbdd \,S(\tau)
+
\frac{1}{N}\sum_{i=1}^N |H(\tau,X_i(\tau))|\,.
\end{equation}
By Grönwall's inequality,
\begin{equation}
S(\tau)
\le
\int_0^\tau
e^{2 \Cbdd (\tau-s)}
\frac{1}{N}\sum_{i=1}^N |H(s,X_i(s))|\,\d s\,,
\end{equation}
since $S(0)=0$, and the proof is concluded by the change of variable $\tau=T-t$.
\end{proof}

\end{document}